\documentclass[11 pt,a4paper,leqno,english]{smfart}

\usepackage{smfthm}

\usepackage[ansinew]{inputenc}  

\usepackage[english]{babel}

\usepackage{latexsym,amsmath,amscd,amsfonts,amssymb,mathrsfs}

\usepackage[all]{xy}

\author{Jean-Pierre Ramis}
\address{Institut de France (Acad{\'e}mie des Sciences)
and Institut de Math{\'e}matiques, CNRS UMR 5219, {\'E}quipe {\'E}mile Picard,
U.F.R. M.I.G., Universit{\'e} Paul Sabatier (Toulouse 3),
31062 Toulouse CEDEX 9}
\email{ramis.jean-pierre@wanadoo.fr}

\author{Jacques Sauloy}
\address{
Institut de Math{\'e}matiques, CNRS UMR 5219, {\'E}quipe {\'E}mile Picard, 
U.F.R. M.I.G., Universit{\'e} Paul Sabatier (Toulouse 3),
31062 Toulouse CEDEX 9}
\email{sauloy@math.univ-toulouse.fr}
\urladdr{www.math.univ-toulouse.fr/~sauloy/}

\title[The $q$-analogue of the wild fundamental group (III)]
{The $q$-analogue of the wild fundamental group and the inverse problem 
of the Galois theory of $q$-difference equations}
\alttitle{Le $q$-analogue du groupe fondamental sauvage et le probl\`eme
inverse de la th\'eorie de Galois aux $q$-diff\'erences}

\def\C{{\mathbf C}}
\def\Q{{\mathbf Q}}
\def\Z{{\mathbf Z}}
\def\R{{\mathbf R}}
\def\N{{\mathbf N}}

\def\dim{\text{dim}}

\def\Diag{\text{Diag}}

\def\Im{\text{Im}}
\def\Ker{\text{Ker}}
\def\Id{\text{Id}}
\def\Sp{\text{Sp}}

\def\lmod{\left |}            
\def\rmod{\right |}           
\def\tq{~ | ~}              
\def\ie{\emph{i.e.}}

\def\cf{\emph{cf.}}

\def\Hom{\text{Hom}}          
\def\Aut{\text{Aut}}          
\def\End{\text{End}}          
\def\1{\underline{1}}         
\def\Ker{\text{Ker~}}         
\def\Id{\text{Id}}            
\def\Mat{{\text{Mat}}}        
\def\GL{{\text{GL}}}          
\def\gl{{\text{gl}}}          
\def\Ad{{\text{Ad}}}          
\def\ad{{\text{ad}}}          

\def\ii{{\text{i}}}

\def\sq{\sigma_q}

\def\F{{\mathcal{F}}}

\def\O{{\mathcal{O}}}
\def\M{{\mathcal{M}}}

\def\Co{{\mathcal{C}}}
\def\D{{\mathcal{D}_{q,K}}}
\def\gr{{\text{gr}}}

\def\Eq{{\mathbf{E}_{q}}}
\def\EE{{\mathcal{E}}}
\def\G{{\mathfrak{G}_{A_{0}}}}
\def\g{{\mathfrak{g}_{A_{0}}}}

\def\Ra{{\C\{z\}}}
\def\Ka{{\C(\{z\})}}

\def\Raq(d){{\C\{\xi\}_{q,(\delta)}}}
\def\Kaq(d){{\C(\{\xi\})_{q,(\delta)}}}

\def\Rf{{\C[[z]]}}
\def\Kf{{\C((z))}}
\def\Rw{{\mathcal{O}(\C^{*})}}
\def\Kw{{\mathcal{M}(\C^{*})}}
\def\Rwg{{\mathcal{O}(\C^{*},0)}}
\def\Kwg{{\mathcal{M}(\C^{*},0)}}

\def\Der{{\dot{\Delta}}}

\def\homa{{\hat{\omega}^{(0)}_{a}}}
\def\oma{{\omega^{(0)}_{a}}}
\def\Vect{{\text{Vect}}}

\def\Gal{{\text{Gal}}}
\def\Lie{{\text{Lie}}}
\def\St{\mathfrak{St}}
\def\st{\mathfrak{st}}
\def\stt{{\tilde{\st}}}
\def\Rep{{\text{Rep}}}
\def\Repa{{\text{Rep}_{A_{0}}}}
\def\Repc{{\text{Rep}_{\C}}}
\def\U{{\mathfrak U}_{A_{0}}}
\def\Res{{\text{Res}}}
\def\X{{\mathcal{X}}}
\def\Y{{\mathcal{Y}}}
\def\L{{\mathbf{L}}}
\def\mmu{{\underline{\mu}}}

\def\Lib{{\text{Lib}}}
\def\Ad{{\text{Ad}}}

\begin{document}
\frontmatter

\begin{abstract}
In \cite{RS1,RS2}, we defined $q$-analogues of alien derivations
for linear analytic $q$-difference equations with integral slopes
and proved a density theorem (in the Galois group) and a freeness 
theorem. In this paper, we completely describe the wild fundamental 
group and apply this result to the inverse problem in $q$-difference 
Galois theory.
\end{abstract}

\subjclass{39A13. Secondary: 34M50}

\keywords{q-difference equations, Stokes phenomenon, alien derivations,
Galois theory, inverse problem}

\maketitle

\tableofcontents

\mainmatter



\section{Introduction}


\subsection{The problems}

The main purpose of this paper is to give a \emph{new} and probably 
definitive version of the local meromorphic classification of $q$-difference 
modules in the \emph{integral} slopes case\footnote{This is explained in 
section \ref{subsection:overallstructure}. For the definition and properties 
of slopes, see section \ref{section:previousresults} and \cite{JSFIL}.}. 
Using this result we shall get \emph{a complete solution} of the inverse 
problem for the $q$-difference Galois theory in the local case, \emph{for all} 
$q\in\C^*$, $\vert q\vert \neq 1$, and a solution of the inverse problem 
for \emph{connected reductive algebraic groups} in the global case, 
also \emph{for all} $q\in\C^*$, $\vert q\vert \neq 1$ (for the case of 
the \emph{exceptional simple groups}, in particular, this result is 
new\footnote{For the simple groups $SL(n,\C)$, $SO(n,\C)$, $Sp(2n,\C)$ 
there are \emph{explicit} solutions with generalized $q$-hypergeometric 
difference equations due to J. Roques, \cf\ section 
\ref{subsection:knownresults}}).

\subsubsection{The $q$-wild fundamental group}

In \cite{RSZ} we gave \emph{three} versions of the local meromorphic 
classification of $q$-difference modules (in the integral slopes case). 
The first one uses algebraic normal forms and index theorems, it improves 
some results of Birkhoff and Guenther \cite{Birkhoff3}, there is no analog 
in the differential case. The second method uses a $q$-analog of Poincar\'e 
asymptotics expansions and the non abelian cohomology $H^1(\Eq,\Lambda)$ of 
some sheaves $\Lambda$ on the (loxodromic) elliptic curve $\Eq:=\C^*/q^{\Z}$, 
it parallels some results of Malgrange and Sibuya (after Birkhoff, 
Balser-J\"urkat-Lutz) in the differential case. The third method uses 
$q$-multisummability, it parallels \cite{MR3} in the differential case. \\

The new version of the classification exposed here is based upon 
a ``fundamental group'' $\pi_{1,q,w,1}^{(0)}$ that we named the 
\emph{$q$-wild fundamental group}\footnote{In $\pi_{1,q,w,1}^{(0)}$, 
the subsript $1$ is for the analogy with $\pi_1$, $q$ is clear, $w$ 
is for \emph{wild}, the last $1$ is for \emph{integral} slopes case 
(\ie with denominator $1$) and the superscript ${}^{(0)}$ is for \emph{local 
at $0$}.}, a $q$-analog of the \emph{wild fundamental group} introduced by 
the first author in the differential case \cite{DMR}, \cite{MR3}. There is 
an equivalence of (tannakian) categories between the category of finite 
dimensional representations of this $q$-wild fundamental group and the 
category of $q$-difference modules (with integral slopes), moreover the 
image of a representation is ``the" $q$-difference Galois group of the 
corresponding module. This classification is in the style of the 
\emph{Riemann-Hilbert} correspondence for regular-singular meromorphic 
linear differential equations and should have similar (important...) 
applications.

Of course there is a ``trivial" candidate for a $q$-wild fundamental group 
satisfying our requirements: the tannakian Galois group 
$\text{Gal}(\EE_1^{(0)})$ of the tannakian category $\EE_1^{(0)}$ of 
our $q$-modules, but this (proalgebraic) group is ``too abstract and 
too big", our purpose was to get a \emph{smaller} fundamental group 
(as small as possible !) which is Zariski dense in the tannakian Galois 
group and to describe it \emph{explicitly}. (As a byproduct, we shall 
get finally a complete description of the tannakian Galois group itself.) 
It is important to notice that the tannakian Galois group is an 
\emph{algebraic object}, but that the construction of the smaller group 
is based upon \emph{transcendental techniques} (complex analysis). 
This is similar to what happens with the Riemann-Hilbert correspondance. \\

We will see that it is possible to write:
\[
\text{Gal}(\EE_1^{(0)})=\mathfrak{St}\, \rtimes \text{Gal}(\EE_{p,1}^{(0)})
\]
where\footnote{This a priori strange notation is motivated by the fact that 
this group is the Galois group of the category of \emph{pure} modules.}, 
by definition, $\text{Gal}(\EE_{p,1}^{(0)}):= \Hom_{gr}(\Eq,\C^*)\times \C$ 
and $\mathfrak{St}$ is a \emph{prounipotent} group (named \emph{the Stokes 
group}). We can replace $\text{Gal}(\EE_1^{(0)})$ by an equivalent datum, 
the action of $\text{Gal}(\EE_{p,1}^{(0)})$ on the Lie algebra $\mathfrak{st}$ 
of $\mathfrak{St}$. We denote this datum  as a semi-direct product
$\mathfrak{st}\, \rtimes \text{Gal}(\EE_{p,1}^{(0)})$. \\

We build a \emph{free} Lie algebra $L$ generated by an \emph{infinite} 
family of symbols $\dot\Delta_i^{(\delta,\bar c)}$ 
($\delta\in\N^*,~\bar c\in\Eq,~i=1,\ldots ,\delta$) and $\dot\Delta^{(0)}$, 
the (pointed) $q$-alien derivations, endowed with an action of 
$\text{Gal}(\EE_{p,1}^{(0)})_s:= \Hom_{gr}(\Eq,\C^*)$, and a natural 
$\text{Gal}(\EE_{p,1}^{(0)})_s$-equivariant
map $L \rightarrow \tilde{\mathfrak{st}} := 
\mathfrak{st} \oplus \C \log \dot\Delta^{(0)}$. 
Then, \emph{by definition}:
\[
\pi_{1,q,w,1}^{(0)}:=L\rtimes \text{Gal}(\EE_{p,1}^{(0)})_s
\]
and we prove that the natural map
\[
\text{Rep}\, (\text{Gal}(\EE_1^{(0)})) \rightarrow 
\text{Rep}\, (\pi_{1,q,w,1}^{(0)})
\]
is an \emph{isomorphism}. \\

As a byproduct, we prove that, for some convenient pronilpotent completion 
$L^{\dag}$ (introduced in section \ref{subsection:firststelinverseproblem}
and studied in the appendix) of the free Lie-algebra the map:
\[
\exp(L^{\dag})\rtimes G^{(0)}_{p,1,s} \rightarrow 
\exp(\tilde{\mathfrak{st}})\,\rtimes G^{(0)}_{p,1,s} =
\mathfrak{St}\, \rtimes G^{(0)}_{p,1}=G^{(0)}_1
\]
is an isomorphism of proalgebraic groups. It is an ``explicit description" 
of the tannakian group $G^{(0)}_1$. \\

The construction of $L$ and the proof of its main properties is the outcome 
of a quite long process (in three steps: \cite{RS1}, \cite{RS2} and the 
present article) and uses some deep results of \cite{RSZ}. 
In \cite{RS1} we built some (pointed) $q$-alien derivations 
$\Der_a^{\delta}$ belonging to $\mathfrak{st}$\footnote{The pointed 
$q$ alien derivations are $q$-analog of the \emph{algebraic} pointed alien 
derivation introduced in \cite{MR2}. The name comes from the fact that 
\emph{in the simplest cases} the Martinet-Ramis pointed alien derivations 
``coincide" with the derivations introduced before by J. Ecalle under 
this name. For a proof \cf\ \cite{LoR}.}, we interpreted them using 
$q$-Borel-Ramis transform and we got the ``first level" of our construction 
(the ``linear case" as in the two-slopes case). In \cite{RS2} we proved 
the \emph{Zariski density} of the Lie algebra generated by the $q$-alien 
derivations and we gave a first (awkward...) tentative of devissage in order 
to ``free" a convenient \emph{subset} of an extended set of alien derivations. 
Here we finally give ``the good" devissage and we prove the \emph{freeness 
theorem} (theorem \ref{theo:freenesstheorem}). The freeness property is 
\emph{absolutely crucial}, it allows a very easy \emph{computation} of 
the representations of the $q$-wild fundamental group and in particular 
the solution of the inverse problem. \\
 
The ($q$-Gevrey) devissage used in the present article is based upon the  
($q$-Gevrey) devissage  of the non-abelian cohology sets of some sheaves 
of unipotent groups on $\Eq$ and its relations with the $q$-alien derivations.
 We think that this devissage \emph{is interesting by itself} and will give 
later some relations between some $H^1(\Eq,\Lambda)$ and some representations 
of algebraic groups. \\

The underlying idea of our construction is that the knowledge of a 
$q$-difference module is equivalent to the knowledge of its \emph{formal 
invariants} and of the corresponding $q$-Stokes phenomena (in the sense of 
\cite{RSZ}). This is similar to what happens in the differential case, but 
unfortunately there is a major difference, here the entries of the Stokes 
matrices are $q$-constants, that is elliptic functions on $\Eq$, and we would 
like instead some matrices belonging to $\GL_n(\C)$ (the $q$-difference 
Galois groups are defined on $\C$). This motivates the replacement of Stokes 
matrices by $q$-alien derivations (using residues) introduced in \cite{RS1}: 
a trick to \emph{reduce} the field of constants from $\mathcal{M}(\Eq)$ 
to $\C$. \\

As a byproduct of our classification theorem we get a $q$-analog of the 
Ramis density theorem of the differential case \cite{MR3}. \\

At the end of the story there is a fascinating parallel between the 
differential and the $q$-difference case. However, it was impossible 
(in any case for us...) to mimick the differential approach which is
essentially based upon the concept of \emph{solution}, because in 
the $q$-difference case the solutions behave badly by tensor products.  
Hence we followed a new path   using (roughly speaking) categories 
in place of solutions. \\

For more details about the analogies between the $q$-wild fundamental 
group and the wild fundamental group of the differential case the reader 
can have a look to the introduction of \cite{RS1}\footnote{In fact it is 
possible to get a perfect analogy if one replaces the free resurgent 
algebra of the wild fundamental group by a bigger free Lie algebra endowed 
with an action not only of $\Z$ but of its proalgebraic completion 
$\Hom_{gr}(\C^*,\C^*)\times \C$, we will return to this problem in 
a future paper.}. \\

For each point $\bar \alpha\in\Eq$, we can consider the semi-direct 
product of the free Lie algebra generated by the symbols 
$\dot\Delta_{\bar \alpha}^{\delta,\bar \alpha^{\delta}}$ ($\delta\in\N^*$)  
by $\C^*$ (the action of $\C^*$ corresponding to the grading $\delta$).
The corresponding category of representations is isomorphic to the 
category of representations of a quotient of $\pi_{1,q,w,1}$.
Similar groups appear in the linear differential case, in the non linear 
differential case (Lie algebras of Ecalle pointed alien derivations
\footnote{The Lie algebra generated by the Ecalle pointed alien derivations 
$\{\dot \Delta_n\}_{n\in\N*}$ is free, the grading corresponding to the 
rescaling of $e^{-1/x}$. There is a dictionary between Martinet-Ramis 
classification of saddle-nodes and some representations of this algebra  
\cite{Sauz}.}) and in the theory of the cosmic Galois group of 
Connes-Marcolli \cite{CoMarc}. These groups are in some sense ``motivic 
groups" (\cf\ also \cite{And} 5. Coda\footnote{\emph{``Ce groupe d'une 
ubiquit\'e stup\'efiante"}, page 16}), therefore we can interpret our 
result as a ``motivic version" of the local classification of the 
$q$-difference modules.

\subsubsection{The inverse problem of the Galois theory 
of $q$-difference equations}
 
Using the $q$-wild fundamental group we can imitate the solution of the 
local inverse problem in the differential case due to the first author. 
The problem is to find necessary and sufficient conditions on a complex 
linear algebraic group in order that this group be the $q$-difference 
Galois group of a local meromorphic $q$-difference module with integral 
slopes ($q\in\C^*,~\vert q\vert \neq 1$). \\

As in the differential case we get easily some \emph{necessary conditions} 
using the algebraic group $V(G):=G/L(G)$ (where $L(G)$ is the invariant 
subgroup generated by all the maximal tori of $G$) and a tannakian argument. 
In the differential case the corresponding conditions are sufficient, but 
here it is no longer the case, there appears a new necessary condition 
involving some type of co-weight on a maximal torus (existence of a 
$\Theta$-structure\footnote{\cf\ the definition \ref{deftheta}}.). Adding 
this condition we get a set of \emph{necessary and sufficient conditions}. 
It follows in particular that a Borel subgroup of a reductive group is 
the $q$-difference Galois group of a local meromorphic $q$-difference module 
with integral slopes. \\

In \cite{JSAIF} and \cite{JSGAL} the second author proved a classification 
theorem for \emph{regular singular} $q$-difference modules, involving the 
local modules at $0$ et $\infty$ and an invertible elliptic connection 
matrix (in Birkhoff style) and derived a description of the corresponding 
Galois group and of a Zariski dense subset of this group. We extend these 
results to the general case. Using this extension and the solution of the 
local inverse problem we get a partial solution of the global inverse 
problem. We prove in particular that every \emph{connected reductive group} 
is the $q$-difference Galois group of a \emph{rational} $q$-difference module.


\subsection{Contents of the paper}

We now briefly sketch the organisation of the paper. General notations
and conventions are explained in the next subsection 
\ref{subsection:generalnotations}. \\

Sections 2 to 4 are devoted to the ``direct problem'' of the description
of the Galois group of a $q$-difference module (or system, or equation)
with integral slopes. In section 2, we review results from our previous 
work \cite{RS1,RS2} and adapt them to our present needs. In section 3, 
we proceed to a complete description of the local Galois group. 
In section 4, we combine this with previous results from \cite{JSGAL}
to obtain a description of the global Galois group (when it makes sense);
this is less complete that section 3 but nevertheless sufficient for
our use in section 7. \\

Sections 5 to 7 are devoted to the inverse problem. This is introduced
in section 5, as well as an important technical tool, the notion of
$\Theta$-structure. In section 6, the local inverse problem is solved. 
In section 7, the global inverse problem is tackled.


\subsection{General notations}
\label{subsection:generalnotations}

Let $q \in \C$ be a complex number with modulus $|q| > 1$. We write 
$\sq$ the $q$-dilatation operator, so that, for any map $f$ on an 
adequate domain in $\C$, one has: $\sq f(z) = f(qz)$. Thus, $\sq$ 
defines a ring automorphism in each of the following rings: $\Ra$ 
(convergent power series), $\Rf$ (formal power series), $\Rw$ 
(holomorphic functions over $\C^{*}$), $\Rwg$ (germs at $0$ of elements 
of $\Rw$). Likewise, $\sq$ defines a field automorphism in each of their 
fields of fractions: $\Ka$ (convergent Laurent series), $\Kf$ (formal 
Laurent series), $\Kw$ (meromorphic functions over $\C^{*}$), $\Kwg$ 
(germs at $0$ of elements of $\Kw$).
The $\sq$-invariants elements of $\Kwg$ actually belong to $\Kw$ and can 
be considered as meromorphic functions on the quotient Riemann surface 
$\Eq = \C^{*}/q^{\Z}$. Through the mapping $x \mapsto z = e^{2 \ii \pi x}$, 
the latter is identified to the complex torus\footnote{Note however that
we shall rather use the \emph{multiplicative} notation for the group 
structure on $\Eq$.} $\C/(\Z + \Z \tau)$, where 
$q = e^{2 \ii \pi \tau}$. Accordingly, we shall identify the fields 
$\Kwg^{\sq}$, $\Kw^{\sq}$ and $\M(\Eq)$. We shall write 
$a \mapsto \overline{a}$ the canonical projection map
$\pi: \C^{*} \rightarrow \Eq$ and 
$[c;q] = \pi^{-1}\left(\overline{c}\right) = c q^{\Z}$
(a discrete logarithmic $q$-spiral).
Last, we shall have use for the function $\theta \in \Rw$, a Jacobi Theta 
function such that $\sq \theta = z \theta$ and $\theta$ has simple zeroes 
along $[-1;q]$. One then puts $\theta_{c}(z) = \theta(z/c)$, so that
$\theta_{c} \in \Rw$ satisfies $\sq \theta_{c} = (z/c) \theta_{c}$ and 
$\theta_{c}$ has simple zeroes along $[-c;q]$. \\

For any two (pro)algebraic groups $G$, $H$, the set of morphisms from
$G$ to $H$ is written $\Hom_{gralg}(G,H)$. When we want to consider all
morphisms of abstract groups, forgetting the (pro)algebraic structure,
we write $\Hom_{gr}(G,H)$.


\subsubsection*{Acknowledgements}

The final redaction of this work was achieved while the second author
was an invited professor for three months at the School of Mathematics
and Statistics of ``Wuda'' (Wuhan University, Wuhan, Hubei, People's 
Republic of China): he wishes to express his gratitude to Wuda for the 
excellent working conditions there.



\section{Previous results on the structure of the local Galois group}
\label{section:previousresults}

In this section, we recall the notations and results of \cite{RS1,RS2} 
and make more precise some of them. \\
 
A linear analytic $q$-difference equation at $0 \in \C$ is an equation:
\begin{equation}
\label{equation:qED}
\sq X = A X, 
\end{equation}
where $A \in \GL_{n}(\Ka)$. We shall identify it with the $q$-difference 
module\footnote{A difference module over a difference field 
$(K,\sigma)$ (\ie\ $\sigma$ is an automorphism of the commutative 
field $K$) is a pair $M := (V,\Phi)$, where $V$ is a finite 
dimensional vector space over $K$ and $\Phi$ a $\sigma$-linear 
automorphism: 
$\forall a \in K \;,\; \forall x \in V \;,\; \Phi(a x) = \sigma(a) \Phi(x)$.
Equivalently, $M$ is a finite length left module over the 
ring $\D := K\left<\sigma,\sigma^{-1}\right>$ of difference 
operators $\sum a_{i} \sigma^{i}$. Difference modules over 
$(\Ka,\sq)$ are called $q$-difference modules.}:
\begin{equation}
\label{equation:modstandard}
M_{A} := (\Ka^{n},\Phi_{A}), \text{~where~} \Phi_{A}(X) := A^{-1} \sq X.
\end{equation}
If $B \in \GL_{p}(\Ka)$, morphisms from $M_{A}$ to $M_{B}$ are described by:
\begin {equation}
\label{equation:morphstandard}
\Hom\bigl(M_{A},M_{B}\bigr) = \{F \in \Mat_{p,n}(\Ka) \tq (\sq F) A = B F\}.
\end{equation}

The $q$-difference modules over $\Ka$ form a $\C$-linear neutral tannakian 
category $\EE^{(0)}$, of which we shall now distinguish some particular 
subcategories. First note that to each $q$-difference module is attached 
a Newton polygon, which can be described as a sequence 
$\mu_{1} < \cdots < \mu_{k}$ of rational slopes coming with multiplicities 
$r_{1},\ldots,r_{k} \in \N^{*}$. Modules with integral slopes form the full 
subcategory $\EE_{1}^{(0)}$ of $\EE^{(0)}$. Modules having only one slope 
are called pure isoclinic; direct sums of pure isoclinic modules are called 
pure and they form the full subcategory $\EE_{p}^{(0)}$ of $\EE^{(0)}$. Pure 
modules with integral slopes form the full subcategory $\EE_{p,1}^{(0)}$ of 
both $\EE_{p}^{(0)}$ and $\EE_{1}^{(0)}$. Pure isoclinic modules of slope $0$ 
are called fuchsian; they form the full subcategory $\EE_{f}^{(0)}$ of 
$\EE_{p,1}^{(0)}$. All these categories are tannakian subcategories of 
$\EE^{(0)}$. Before describing their Galois groups, we shall have a look 
at their fiber functors. \\

For any $q$-difference module $M$, holomorphic solutions in $\sq$-invariant
open subsets of $(\C^{*},0)$ form a sheaf $\F_{M}$ over $\Eq$. This sheaf
is locally free over the structural sheaf of $\Eq$ and thereby defines a
holomorphic vector bundle which we also write $\F_{M}$. In case $M$ is
given in matricial form $M_{A} = (\Ka^{n},\Phi_{A})$, these sheaf and bundle 
admit the following descriptions:
\begin{align*}
\F_{M}(V) &= \{X \in \O\left(\pi^{-1}(V),0\right)^{n} \tq \sq X = A X \}, \\
\F_{M} &= \dfrac{(\C^{*},0) \times \C^{n}}{(z,X) \sim (qz,A(z)X)}
\longrightarrow \dfrac{(\C^{*},0)}{z \sim qz} = \Eq.
\end{align*}
In the right hand side of the first (resp. the second) equality, 
solutions $X \in \O\left(\pi^{-1}(V),0\right)^{n}$ are taken to be 
germs at $0 \in \C^*$ (resp. the bundle $(\C^{*},0) \times \C^{n}$ to 
be quotiented is taken to be trivial over the germ of $\C^*$ at $0$). \\

The functor $M \leadsto \F_{M}$ is exact, faithful and $\otimes$-compatible
and provides a fiber functor on $\EE^{(0)}$ over the base $\Eq$. Lifting
$\F_{M}$ through $\pi$ to an equivariant (trivial) bundle over $\C^{*}$,
then taking fibers, we get a family $(\oma)_{a \in \C^{*}}$ of fiber functors
on $\EE^{(0)}$ over $\C$, thus a Galois groupoid with base $\C^{*}$ over
the field $\C$. (The reason to consider points in $\C^{*}$ rather than
in $\Eq$ is that we want to use transcendental constructions\footnote{It
is not feasible in the setting of $q$-difference equations to define a
fiber functor as the space of solutions in some big field $K$. Indeed,
in order to get a fiber functor in this way, one has to take $K$ rather
big; then the fiber functor is defined over the field of constants of
$K$, which will be bigger than $\C$. For instance, the natural choice
$K = \M(\C^*)$ yields the a Galois group over $\M(\Eq)$.}.) \\

On the other hand, $\EE^{(0)}$ is endowed with a family 
$(F_{\leq \mu})_{\mu \in \Q}$ of endofunctors such that, for each module $M$,
the $F_{\leq \mu} M$ form a filtration of $M$ by subobjects, with jumps
at the slopes of $M$. The associated graded module:
$$
\gr M := \bigoplus \dfrac{F_{\leq \mu} M}{F_{< \mu} M}
$$
is pure and we get a functor $M \leadsto \gr M$ from $\EE^{(0)}$ to
$\EE_{p}^{(0)}$, which is exact, faithful and $\otimes$-compatible.
It is also a retraction of $\EE_{p}^{(0)} \subset \EE^{(0)}$. This yields
a new family of fiber functors on $\EE^{(0)}$:
$$
\homa := \oma \circ \gr.
$$
In some sense, $\EE_{p}^{(0)}$ is the ``formalisation'' of $\EE^{(0)}$ 
and we see the $\homa$, resp. the $\oma$, as points in a formal, resp. 
an analytic neighborhood of $0$. (The reason for this is that, over
the formal category, $\gr$ is isomorphic to the identity functor, see
\cite{JSFIL}.) \\

Whatever the fiber functor used to define it, the Galois group
$\Gal(\EE^{(0)})$ is the semi-direct product of the ``formal'' Galois 
group $\Gal(\EE_{p}^{(0)})$ by a prounipotent group, the kernel of the 
morphism $\Gal(\EE^{(0)}) \rightarrow \Gal(\EE_{p}^{(0)})$ dual to $\gr$. 
(This follows from the existence of the filtration.) Restricting to 
$\EE_{1}^{(0)}$, one gets:
\begin{equation}
\label{equation:gal=prodsemidir}
\Gal(\EE_{1}^{(0)}) = \St \rtimes \Gal(\EE_{p,1}^{(0)}),
\end{equation}
where $\St$ is a prounipotent group. \\

The goal of this series of papers is the description of the \emph{Stokes 
group} $\St$ and the \emph{Stokes Lie algebra}\footnote{Actually, we shall
extend here $\st$ to a Lie algebra $\stt$ which contains the ``Stokes
operators of level $0$'', that is the unipotent part of the fuchsian
Galois group, corresponding to the $q$-logarithm.} $\st := \Lie(\St)$ 
and its application to the inverse problem in $q$-difference Galois 
theory. The main tool on the side of $q$-difference equations is theorem 
\ref{theo:bigth1}, which describes all Galois groups of systems with
integral slopes in terms of representations of a \emph{wild fundamental
group}, actually, the semi-direct product of an infinite dimensional Lie
algebra with a proalgebraic group, the tannakian formal Galois group
of the category of systems with integral slopes. We obtain it with the 
help of an explicit family of \emph{galoisian Stokes operators} built
by the authors together with Changgui Zhang in \cite{RSZ} and
used there to get an analytic classification of $q$-difference
modules. It was proved in previous work \cite{RS1,RS2} that we
thus obtain a generating family. The analytic classification and
representations of $\st$ are, in some sense, two models of the 
same thing, which allows us to give a precise description of the
latter. In this comparison, the filtration above plays a crucial 
role and we shall now have a closer look at it.

\subsubsection*{Convention}

As already said, any object of $\EE^{(0)}$ is equivalent to some $M_A$.
It can moreover be shown that one may always choose $A$ in so-called
Birkhoff-Guenther normal form; in our case of interest, this is explained 
at the beginning of \ref{subsection:overallstructure}. This implies that
$A \in \GL_{n}(\C[z,z^{-1}]) \subset \GL_{n}(\Ka) \cap \GL_n(\O(\C^*))$,
so that the above definitions are simplified to:
\begin{align*}
\F_{M}(V) &= \{X \in \O\left(\pi^{-1}(V)\right)^{n} \tq \sq X = A X \}, \\
\F_{M} &= \dfrac{\C^{*} \times \C^{n}}{(z,X) \sim (qz,A(z)X)}
\longrightarrow \dfrac{\C^{*}}{z \sim qz} = \Eq.
\end{align*}
Moreover, starting from a module $M_{A} = (\Ka^{n},\Phi_{A})$ such that 
$A \in \GL_{n}(\C[z,z^{-1}])$, a module $M_{B} = (\Ka^{p},\Phi_{B})$ such 
that $B \in \GL_{p}(\C[z,z^{-1}])$, and a morphism $F: M_A \rightarrow M_B$, 
$F \in \Mat_{p,n}(\Ka)$, it follows from the relation 
$(\sq F) A = B F \Rightarrow \sq F = B F A^{-1}$ that $F$ is holomorphic 
over $\C^*$ (the functional equation allows one to expand by a factor 
$\lmod q \rmod > 1$ any punctured disk of convergence). Thus, in order to 
have a more concrete description of the fiber functors $\oma$ and $\homa$, 
we shall now restrict to the essential full tannakian subcategory of 
$\EE^{(0)}$ made of $q$-difference modules $M_{A}$ such that 
$A \in \GL_{n}(\C[z,z^{-1}])$. We shall keep the notation $\EE^{(0)}$ 
for this smaller (but equivalent) category. Then, one has canonical 
identifications $\oma(M_A) = \C^n$, $\oma(M_B) = \C^p$ and $\oma(F) = F(a)$. 
A similar description of $\homa$ will be given in 
\ref{subsection:overallstructure}.


\subsection{Some facts about tannakian filtrations}
\label{subsection:tannakianfiltration}

We shall axiomatize the situation described above in the following way.
The tannakian constructions given here essentially follow from
\cite[chap. IV, \S 2]{Saavedra}. \\

We consider a $\C$-linear tannakian category $\mathcal{C}$ endowed 
with a family of endofunctors $F_{\leq \mu}$, $\mu \in \Q$ subject 
to the following constraints:
\begin{enumerate}
\item{For each object $M$ of $\mathcal{C}$, the family $F_{\leq \mu} M$
is an ascending filtration of $M$ by subobjects. This filtration is 
separated, exhaustive and it has a finite number of jumps.}
\item{Defining $\gr^{(\mu)} M := (F_{\leq \mu} M)/(F_{< \mu} M)$ and
$\gr M := \bigoplus\limits_{\mu \in \Q} \gr^{(\mu)} M$, we obtain an
endofunctor $\gr$ which is exact, faithful and $\otimes$-compatible.}
\item{The essential image of $\gr$ is a neutral tannakian subcategory 
$\mathcal{C}_p$ (we call its objects \emph{pure}) and the restriction
of $\gr$ to $\mathcal{C}_p$ is the identity functor.}
\end{enumerate}
Choosing a fiber functor $\omega_p$ from $\mathcal{C}_p$ to $\Vect_{\C}^f$,
we can extend it to a fiber functor $\omega := \omega_p \circ \gr$ on 
$\mathcal{C}$, thus making the latter a neutral tannakian category. 
We call $G,G_p$ the corresponding Galois groups, seen as proalgebraic
groups over $\C$: these are respectively the groups of $\C$-valued points 
of the affine group schemes $\underline{\Aut}^{\otimes}(\omega)$ and
$\underline{\Aut}^{\otimes}(\omega_p)$. \\

By tannakian duality, the inclusion functor $i$ from $\mathcal{C}_p$
into $\mathcal{C}$ and its retraction $\gr$ induce morphisms of
proalgebraic groups $\gr^*: G_p \rightarrow G$ and its section
$i^*: G \rightarrow G_p$. Calling $S$ the kernel of $i^*$, we
get an exact sequence of proalgebraic groups:
$$
1 \rightarrow S \rightarrow G \rightarrow G_p \rightarrow 1.
$$
Since $i^* \circ \gr^* = (\gr \circ i)^* = \Id_{G_p}$, this is
actually a semi-direct product:
$$
G = S \rtimes G_p.
$$

\begin{exem}
\label{exem:pente1}
The above description applies to $\EE^{(0)}$ and $\EE_p^{(0)}$
as explained in the previous subsection, and so will do all
constructions in the present subsection. However, in this paper,
from \ref{subsection:overallstructure} on, we shall rather 
take $\EE_1^{(0)}$ and $\EE_{p,1}^{(0)}$. Hence, the filtrations on
subobjects will have jumps at the integers only, and the functor
$\gr$ will be $\Z$-graded rather than $\Q$-graded. If we restrict to 
$\mu \in \Z$, we shall therefore have $F_{< \mu} M = F_{\leq \mu -1} M$.
The corresponding Galois groups will be $G := \Gal(\EE_{1}^{(0)})$ and 
$G_p := \Gal(\EE_{p,1}^{(0)})$ and the factor $S$ will be denoted $\St$
and called the \emph{Stokes group}.
\end{exem}

We shall have use for the following description of the proalgebraic 
structures on $G$, $G_p$ and $S$. We start with $G$. For each object 
$M$ of $\mathcal{C}$, write $<M>$ the full subcategory whose objects 
are subquotients of direct sums of tensors 
$M^{\otimes n} \otimes (M^{\vee})^{\otimes p}$: 
this is a tannakian subcategory and there is a restriction morphism 
from $G$ to its Galois group $G(M)$, which is a linear algebraic group 
over $\C$. Actually, the map $\gamma \mapsto \gamma(M)$ is an algebraic 
morphism of $G$ into $\GL(\omega(M))$ and its image is closed; and the 
similar map from $G(M)$ to $\GL(\omega(M))$ has the same image and is 
an isomorphism, whence a canonical identification of $G(M)$ with an 
algebraic subgroup of $\GL(\omega(M))$. Say that $M' \preceq M$ if 
$M'$ is an object of $<M>$. Then there is a surjective morphism of 
algebraic groups from $G(M)$ to $G(M')$. In this way, the groups 
$G(M)$ form a filtered inverse system and the surjective maps 
$G \rightarrow G(M)$ make $G$ the inverse limit of the $G(M)$. 
Morphisms from $G$ to an arbitrary algebraic group $G'$ are defined
through the equality:
$$
\Hom(G,G') := \lim_{\to} \Hom_{gralg}(G(M),G').
$$
(The notation $\Hom_{gralg}$ was defined in \ref{subsection:generalnotations}.)
Note that the subobjects $F_{\leq \mu} M$, the subquotients $\gr^{(\mu)} M$
and therefore $\gr M$ itself all are objects of $<M>$, whence a surjective
map of algebraic groups $G(M) \rightarrow G(\gr M)$. But, from the equality 
$\omega_p \circ \gr = \omega$ follows a canonical identification of $G(\gr M)$
with $G_p(M)$, whence a surjective map $G(M) \rightarrow G_p(M)$. Calling
$S(M)$ the kernel of this map (an algebraic group), one has first a 
semidirect decomposition $G(M) = S(M) \rtimes G_p(M)$, second an inverse 
system of exact sequences
$$
1 \rightarrow S(M) \rightarrow G(M) \rightarrow G_p(M) \rightarrow 1
$$
whose inverse limit is the sequence seen before. \\

For any $M$ in $\mathcal{C}$, the vector space
$V := \omega(M) = \omega(\gr M) = \omega_p(\gr M)$ admits:
\begin{itemize}
\item{a separated exhaustive descending filtration by the 
$V_{\leq \mu} := \omega(F_{\leq \mu}M)$, $\mu \in \Q$, with a finite number 
of jumps;}
\item{a graduation by the $V^{(\mu)} := \omega(\gr^{(\mu)}M)$, $\mu \in \Q$, 
with a finite number of non trivial components.}
\end{itemize}
Moreover, the filtration is the one associated with the graduation:
$$
V_{\leq \mu} = \bigoplus\limits_{\nu \leq \mu} V^{(\nu)}.
$$ 
In the terminology of \cite{Saavedra}, it is \emph{scind\'ee} (split), 
thus \emph{admissible}. These structures are determined by $M_0 := \gr M$. 
Using them, one can define the following linear algebraic groups:
\begin{enumerate}
\item{The group of automorphisms of $V$ respecting the filtration:
$$
\GL_{M_0}(V) := 
\{f \in \GL(V) \tq \forall \mu \in \Q \;,\; f(V_{\leq \mu}) \subset V_{\leq \mu}\}.
$$}
\item{The group of automorphisms of $V$ respecting the graduation:
$$
\overline{\GL}_{M_0}(V) := 
\{f \in \GL(V) \tq \forall \mu \in \Q \;,\; f(V^{(\mu)}) \subset V^{(\mu)}\} =
\bigoplus_{\mu \in \Q} \GL(V^{(\mu)}).
$$}
\item{The group of automorphisms of $V$ respecting the filtration and 
inducing the identity on the graduation:
$$
\mathfrak{G}_{M_0}(V) := 
\{f \in \GL(V) \tq 
\forall \mu \in \Q \;,\; (f-\Id_V)(V_{\leq \mu}) \subset V_{< \mu}\}.
$$}
\end{enumerate}
Roughly, these three groups respectively correspond (in matricial form) 
to block upper triangular matrices, to block diagonal matrices and to 
unipotent block upper triangular matrices. The latter two groups are
subgroups of the first one with trivial intersection. \\

There is a natural surjective morphism of algebraic groups from
$\GL_{M_0}(V)$ onto $\overline{\GL}_{M_0}(V)$. Its kernel is 
$\mathfrak{G}_{M_0}(V)$, whence an exact sequence:
$$
1 \rightarrow \mathfrak{G}_{M_0}(V) \rightarrow \GL_{M_0}(V) 
\rightarrow \overline{\GL}_{M_0}(V) \rightarrow 1.
$$
Actually, the surjective morphism admits an obvious section, whence
a semidirect decomposition:
$$
\GL_{M_0}(V) = \mathfrak{G}_{M_0}(V) \rtimes  \overline{\GL}_{M_0}(V).
$$
The group $\mathfrak{G}_{M_0}(V)$ is unipotent (more on this further 
below). Its (nilpotent) Lie algebra can be characterized as the set 
of endomorphisms on $V$ respecting the filtration and trivial (null)
on the graduation:
$$
\mathfrak{g}_{M_0}(V) := 
\{f \in \gl(V) \tq \forall \mu \in \Q \;,\; f(V_{\leq \mu}) \subset V_{< \mu}\}.
$$

Now any ``galoisian'' automorphism $\gamma \in G(M) \subset \GL(\omega(M))$ 
stabilises images of subobjects, in particular the $\omega(F_{\leq \mu}M)$
and the $\omega(F_{< \mu}M)$; and descends to subquotients, in particular
the $\omega(\gr^{(\mu)}(M)$. With the notations just introduced, we thus
have an embedding of exact sequences:
$$
\xymatrix{
0 \ar@<0ex>[r] & S(M) \ar@<0ex>[r] \ar@{^{(}->}[d] & 
G(M) \ar@<0ex>[r] \ar@{^{(}->}[d] & G_p(M) \ar@<0ex>[r] \ar@{^{(}->}[d] & 0 \\
0 \ar@<0ex>[r] & \mathfrak{G}_{M_0}(V) \ar@<0ex>[r] & 
\GL_{M_0}(V) \ar@<0ex>[r] & \overline{\GL}_{M_0}(V) \ar@<0ex>[r] & 0
}
$$

Note in particular that, since $S(M)$ embeds into $\mathfrak{G}_{M_0}(V)$,
it is a unipotent group and its Lie algebra $s(M)$ is nilpotent. We have:
\begin{align*}
S(M) &= \exp s(M) = 1 + s(M), \\
\mathfrak{G}_{M_0}(V) &= \exp \mathfrak{g}_{M_0}(V) = 1 + \mathfrak{g}_{M_0}(V).
\end{align*}
We put, for all $\delta \geq 0$:
$$
\mathfrak{g}_{M_0}^{> \delta}(V) := 
\{f \in \gl(V) \tq 
\forall \mu \in \Q \;,\; f(V_{\leq \mu}) \subset V_{< \mu - \delta}\}.
$$
(In matricial form, this corresponds to upper triangular nilpotent
matrices with decreasing orders of nilpotency.) This gives a descending 
separated exhaustive filtration of $\mathfrak{g}_{M_0}(V)$ by ideals, with 
a finite number of jumps, and we have:
$$
\mathfrak{g}_{M_0}^{> 0}(V) = \mathfrak{g}_{M_0}(V) \text{~and~}
[\mathfrak{g}_{M_0}^{> \delta}(V),\mathfrak{g}_{M_0}^{> \delta'}(V)]
\subset \mathfrak{g}_{M_0}^{> \delta+\delta'}(V).
$$
We are interested in the induced filtration on the Stokes Lie algebras:
$$
s^{> \delta}(M) := s(M) \cap \mathfrak{g}_{M_0}^{> \delta}(V).
$$
This gives again a descending separated exhaustive filtration of 
$s(M)$ by ideals, with a finite number of jumps, such that:
$$
s^{> 0}(M) = s(M) \text{~and~}
[s^{> \delta}(M),s^{> \delta'}(M)] \subset s^{> \delta + \delta'}(M).
$$

Last, we introduce the group counterparts:
\begin{align*}
\mathfrak{G}_{M_0}^{> \delta}(V) &:= 1 + \mathfrak{g}_{M_0}^{> \delta}(V), \\
S^{> \delta}(M) & := 1 + s^{> \delta}(M) = S(M) \cap \mathfrak{G}_{M_0}^{> \delta}(V).
\end{align*}
This yields two descending separated exhaustive filtrations by normal
subgroups, with a finite number of jumps, such that:
\begin{align*}
\mathfrak{G}_{M_0}^{> 0}(V) &= \mathfrak{G}_{M_0}(V) &\text{~and~}
& [\mathfrak{G}_{M_0}^{> \delta}(V),\mathfrak{G}_{M_0}^{> \delta'}(V)]
&\subset \mathfrak{G}_{M_0}^{> \delta+\delta'}(V) \\
S^{> 0}(M) &= S(M) &\text{~and~}
& [S^{> \delta}(M),S^{> \delta'}(M)] &\subset S^{> \delta + \delta'}(M).
\end{align*}

Now we return to our global Galois groups. Since the exact sequence
$1 \rightarrow S \rightarrow G \rightarrow G_p \rightarrow 1$ is the
inverse limit of the exact sequences
$1 \rightarrow S(M) \rightarrow G(M) \rightarrow G_p(M) \rightarrow 1$,
we see that $S$ is a prounipotent group, equipped with a separated 
exhaustive filtration by normal subgroups $S^{> \delta}$. Its Lie
algebra (see \cite[A.7]{DG}) $s = \Lie~S := \lim\limits_{\leftarrow} s(M)$
is likewise equipped with a separated exhaustive filtration by ideals
$s^{> \delta}$, and it is a pronilpotent Lie algebra. Actually, $s$ 
inherits the inverse limit topology of the discrete topologies on 
the $s(M)$, which has as a fundamental system of neighborhoods of $0$ 
this filtration; and $s$ is Hausdorff and complete for that topology.


\subsection{Overall structure and representations of $\Gal(\EE_{1}^{(0)})$}
\label{subsection:overallstructure}

We now make an important assumption: \\

\textbf{\emph{From now on, we shall restrict to modules with integral 
slopes.}} \\

The reason is that we then have explicit normal forms, and we are going 
to use them heavily\footnote{In the general case of rational slopes, 
van der Put and Reversat obtained a precise description of pure modules 
and of the Galois group of $\EE_{p}^{(0)}$, see \cite{vdPR}. Relying on  
these results, Virginie Bugeaud has started to extend the methods of the 
present series of papers to the case of two arbitrary slopes.}. Indeed, 
any pure module $M_{0}$ with integral slopes $\mu_{1} < \cdots < \mu_{k}$ 
and multiplicities $r_{1},\ldots,r_{k}$ can be described as 
$M_{A_{0}} := (\Ka^{n},\Phi_{A_{0}})$, and any module $M$ such that 
$\gr M \approx M_{0}$ can be described as $M_{A} := (\Ka^{n},\Phi_{A})$ 
(see equation \eqref{equation:modstandard}), with:
\begin{equation}
\label{eqn:formesstandards}
A_{0} := \begin{pmatrix}
z^{\mu_{1}} A_{1} & \ldots & \ldots & \ldots & \ldots \\
\ldots & \ldots & \ldots  & 0 & \ldots \\
0      & \ldots & \ldots   & \ldots & \ldots \\
\ldots & 0 & \ldots  & \ldots & \ldots \\
0      & \ldots & 0       & \ldots & z^{\mu_{k}} A_{k}    
\end{pmatrix} \text{~and~}
A := \begin{pmatrix}
z^{\mu_{1}} A_{1} & \ldots & \ldots & \ldots & \ldots \\
\ldots & \ldots & \ldots  & U_{i,j} & \ldots \\
0      & \ldots & \ldots   & \ldots & \ldots \\
\ldots & 0 & \ldots  & \ldots & \ldots \\
0      & \ldots & 0       & \ldots & z^{\mu_{k}} A_{k}   
\end{pmatrix},
\end{equation}
where, for $1 \leq i \leq k$, $A_{i} \in \GL_{r_{i}}(\C)$ and where, 
for $1 \leq i < j \leq k$, $U_{i,j} \in \Mat_{r_{i},r_{j}}(\Ka)$; moreover, 
one can assume that the coefficients of each block $U_{i,j}$ belong to 
$\sum\limits_{\mu_{i} \leq \ell < \mu_{j}} \C z^{\ell}$ (Birkhoff-Guenther 
normal form). Modules $M_{A} := (\Ka^{n},\Phi_{A})$ form an essential 
tannakian subcategory of $\EE_{1}^{(0)}$, so that we can restrict all 
our definitions and constructions to such objects. \\

The fiber functors $\homa$ and $\oma$ admit the following concrete
description. Let $A,A_0$ be as in \eqref{eqn:formesstandards} and
write for short $M := M_A$, $M_0 := M_{A_0}$, so that $M_0 = \gr M$.
Then $\homa(M) = \oma(M) = \oma(M_0) = \C^n$. 
Now define similarly $B \in \GL_p(\Ka)$ in Birkhoff-Guenther normal 
form with slopes $\nu_1 < \cdots < \nu_l$ having multiplicities
$s_1,\ldots,s_l$ and $B_0$ its graded (block diagonal) component
and put $N := M_B$, $N_0 := M_{B_0}$, so that $N_0 = \gr N$. Then
any morphism $M \rightarrow N$ is a matrix $F \in \Mat_{p,n}(\Ka)$
such that $\sq F = B F A^{-1}$, so that one easily shows that 
$F \in \Mat_{p,n}(\O(\C^*))$. The corresponding graded morphism
$F_0 := \gr(F) \in \Mat_{p,n}(\Ka) \cap \Mat_{p,n}(\O(\C^*))$ has
$k l$ blocks of sizes $r_i \times s_j$, those such that $\mu_i = \nu_j$ 
coming from $F$, all the other ones being trivial. Then one has: 
\begin{align*}
\oma(F) &= F(a), \\
\homa(F) &= F_0(a).
\end{align*}

The Galois groups of $\EE_{f}^{(0)}$ and $\EE_{p,1}^{(0)}$ are abelian, 
so that we can use any fiber functor to describe them. Using the 
subscript ``f'' for ``fuchsian'' and the subscript ``p'' for ``pure'',
we have:
\begin{align*} 
G_{f}^{(0)} := \Gal(\EE_{f}^{(0)}) & = \Hom_{gr}(\C^{*}/q^{\Z},\C^{*}) \times \C, \\
G_{p,1}^{(0)} := \Gal(\EE_{p,1}^{(0)}) & = \C^{*} \times G_{f}^{(0)}.
\end{align*}
(The notation $\Hom_{gr}$ was defined in \ref{subsection:generalnotations}.)
We also write $G_{f,s}^{(0)} = \Hom_{gr}(\C^{*}/q^{\Z},\C^{*})$ the semi-simple
component of the fuchsian group $G_{f}^{(0)}$; its elements are identified
with (abstract group) morphisms $\C^{*} \rightarrow \C^{*}$ that send $q$
to $1$. Likewise, we write $G_{f,u}^{(0)} = \C$ the unipotent component of
$G_{f}^{(0)}$ and $T_{1}^{(0)} = \C^{*}$ the ``theta torus'' component of 
$G_{p,1}^{(0)}$; the latter should be compared\footnote{For details on
this analogy, see the introduction of \cite{RS1} and the conclusion
of \cite{RS2}.} with the ``exponential torus'' component of the wild 
fundamental group of differential equations. \\

Taking again $A$ in form \eqref{eqn:formesstandards}, the representation 
of $G_{p,1}^{(0)} = G_{f,s}^{(0)} \times G_{f,u}^{(0)} \times T_{1}^{(0)}$ 
corresponding to $M := M_{A}$ by tannakian duality is the following:
$$
(\gamma,\lambda,t) \mapsto 
\begin{pmatrix}
t^{\mu_{1}} \gamma(A_{1,s}) A_{1,u}^{\lambda} & \ldots & \ldots & \ldots & \ldots \\
\ldots & \ldots & \ldots  & 0 & \ldots \\
0      & \ldots & \ldots   & \ldots & \ldots \\
\ldots & 0 & \ldots  & \ldots & \ldots \\
0      & \ldots & 0       & \ldots & t^{\mu_{k}} \gamma(A_{k,s}) A_{k,u}^{\lambda}  
\end{pmatrix}.
$$
We wrote $A_{i} = A_{i,s} A_{i,u}$ the Jordan decomposition into semi-simple
and unipotent component, and $\gamma(A_{i,s})$ means $\gamma$ operating on 
eigenvalues of $A_{i}$. \\

As explained in example \ref{exem:pente1}, we now write $\St$ the kernel
of $i^*: G_1^{(0)} \rightarrow G_{p,1}^{(0)}$, a prounipotent proalgebraic
group, whence the semidirect decomposition of \eqref{equation:gal=prodsemidir}:
$$
G_1^{(0)} = \St \rtimes G_{p,1}^{(0)}.
$$
We write $\st$ the Lie algebra of $\St$; it is pronilpotent, see the end
of \ref{subsection:tannakianfiltration}.


\subsection{First look at the structure of $\St$ and $\st$}
\label{subsection:firstlookstructure}

Let us characterize \emph{Stokes operators}, \ie\ elements of 
the \emph{Stokes group} $\St$ and \emph{alien derivations}, \ie\ 
elements of the \emph{Stokes Lie algebra} $\st$. Let $s \in \St$, 
resp. $D \in \st$. Their respective images by the representation 
associated to matrix $A$ (meaning: to module $M_{A}$) are
\begin{align*}
s(A) \in \St(A) \subset \G(\C) \subset \GL_n(\C), 
&\text{~~where~~} \St(A) := \St(M_{A}), \\
D(A) \in \st(A) \subset \g(\C) \subset \gl_n(\C), 
&\text{~~where~~} \st(A) := \st(M_{A}),
\end{align*}
where we introduce the following unipotent algebraic group $\G$ 
and its Lie algebra $\g$:
\begin{align*}
\G &:= \left\{
\begin{pmatrix}
I_{r_{1}} & \ldots & \ldots & \ldots & \ldots \\
\ldots & \ldots & \ldots  & \star & \ldots \\
0      & \ldots & \ldots   & \ldots & \ldots \\
\ldots & 0 & \ldots  & \ldots & \ldots \\
0      & \ldots & 0       & \ldots &   I_{r_{k}}
\end{pmatrix}
\right\} \subset \GL_{n}, \\
\g &:= \left\{
\begin{pmatrix}
0_{r_{1}} & \ldots & \ldots & \ldots & \ldots \\
\ldots & \ldots & \ldots  & \star & \ldots \\
0      & \ldots & \ldots   & \ldots & \ldots \\
\ldots & 0 & \ldots  & \ldots & \ldots \\
0      & \ldots & 0       & \ldots &   0_{r_{k}}
\end{pmatrix}
\right\} \subset \gl_{n} = \Mat_{n}.
\end{align*}
Here $I_{r}$ and $0_{r}$ respectively denote the identity and the null
matrix of size $r \times r$. The rectangular block $\star$ indexed by
$(i,j)$ such that $1 \leq i < j \leq k$ has size $r_{i} \times r_{j}$
and links the diagonal square blocks corresponding to slopes $\mu_{i}$
and $\mu_{j}$. \\

Globally, $s$ and $D$ are characterized as follows. They must be 
functorial: if $(\sq F) A = B F$, then 
$$
s(B) F_{0}(a) = F_{0}(a) s(A) \text{~and~} D(B) F_{0}(a) = F_{0}(a) D(A)
$$
(for the chosen base point $a \in \C^{*}$). They must be $\otimes$-compatible:
$$
s(A \otimes B) = s(A) \otimes s(B) \text{~and~}
D(A \otimes B) = D(A) \otimes I_{p} + I_{n} \otimes D(B).
$$
Last, they must be trivial on pure modules: 
$$
s(A_{0}) = I_{n} \text{~and~} D(A_{0}) = 0_{n}.
$$

The character group of the semi-simple component of $G_{p,1}^{(0)}$ is:
$$
X\left(G_{f,s}^{(0)} \times T_{1}^{(0)}\right) = \Eq \times \Z.
$$

To describe the adjoint action of this group on $\st$ therefore amounts
to give the decomposition in eigenspaces; note that for the projective
limit $\st$, we have to complete the direct sum:
\begin{equation}
\label{decompositiondest}
\st = \hat{\bigoplus_{\delta \geq 1}} \st^{(\delta )}, \text{~where~}
\st^{(\delta )} = \hat{\bigoplus_{\overline{c} \in \Eq}} \st^{(\delta,\overline{c})}.
\end{equation}
(Note that only the weights such that $\delta \geq 1$ are required, because
of the triangular structure coming from the functorial filtration theorem.) 
This decomposition is expressed elementwise as a Fourier decomposition:
$$
\forall D \in \st \;,\; \forall \sigma \in G_{f,s}^{(0)} \times T_{1}^{(0)} \;,\;
\sigma D \sigma^{-1} = 
\sum_{\chi \in X\left(G_{f,s}^{(0)} \times T_{1}^{(0)}\right)}
\left<\chi,\sigma\right> D^{(\chi)},
$$
where, for $\chi = (\delta,\overline{c}) \in \Z \times \Eq$ and for
$\sigma = (t,\gamma) \in \C^{*} \times \Hom_{gr}(\C^{*}/q^{\Z},\C^{*})$:
$$
\left<\chi,\sigma\right> = t^{\delta} \gamma(\overline{c}).
$$
Thus, $D = \sum D^{(\delta,\overline{c})}$ (with unicity of the decomposition)
and:
$$
\sigma D^{(\delta,\overline{c})} \sigma^{-1} = 
t^{\delta} \gamma(\overline{c}) D^{(\delta,\overline{c})}.
$$

Since $G_{p,1}^{(0)}$ is abelian, conjugacy under elements of its unipotent 
component $G_{f,u}^{(0)}$ fixes all each $\st^{(\delta,\overline{c})}$. We shall 
write $\tau$ the (Zariski-) generator $1 \in \C = G_{f,u}^{(0)}$, so that:
$$
\tau \, \st^{(\delta,\overline{c})} \, \tau^{-1} = \st^{(\delta,\overline{c})}.
$$


\subsection{First look at the representations of $\St$ and $\st$}
\label{subsection:firstlookrepresentations}

More generally, the semi-simple component of $G_{p,1}^{(0)}$ operates on $\g$ 
through $G_{p,1}^{(0)}(A) = G_{p,1}^{(0)}(A_{0})$, whence a decomposition:
$$
\g = \bigoplus_{\delta \geq 1} \g^{(\delta)}, \text{~where~}
\g^{(\delta )} = \bigoplus_{\overline{c} \in \Eq} \g^{(\delta,\overline{c})}.
$$
(And, of course, $\st^{(\delta )}(A) = \st(A) \cap \g^{(\delta)}$, etc.)
More concretely, one can divide matrices in $\g$ in rectangular blocks
numbered $(i,j)$ with $1 \leq i < j \leq k$; the block $i,j$ has size
$r_{i} \times r_{j}$ and links the (null) square diagonal blocks corresponding
to slopes $\mu_{i}$ and $\mu_{j}$. If one assumes moreover that the matrices 
$A_{i}$ are divided in diagonal blocks corresponding to their eigenvalues, 
then one can further divide each block $(i,j)$ in rectangular blocks numbered 
$(d,e) \in \Sp A_{i} \times \Sp A_{j}$. The action of 
$\sigma = (\gamma,t) \in G_{f,s}^{(0)} \times T_{1}^{(0)}$
(through its image in $\G$) on the block $\bigl((i,j),(d,e)\bigr)$ is 
multiplication by the nonzero scalar 
$\dfrac{t^{\mu_i}}{t^{\mu_j}}
\dfrac{\gamma(\overline{d})}{\gamma(\overline{e})}$.
Thus, the matrices of $\g^{(\delta )}$ are those such that blocks with 
$\mu_{j} - \mu_{i} \neq \delta$ are all zero and the matrices of 
$\g^{(\delta,\overline{c})}$ are those matrices of $\g^{(\delta )}$ such that 
blocks with $d/e \not\equiv c \pmod{q^{\Z}}$ are all zero. We shall 
frequently identify $\g^{(\delta)}$, resp. $\g^{(\delta,\overline{c})}$ with 
the corresponding vector spaces of rectangular matrices, forgetting their null 
components. For instance, in the case of two slopes $\mu < \nu$ with
multiplicities $r,s \in \N^{*}$, the (abelian) Lie algebra $\g$ has a
single nontrivial component $\g^{(\delta )}$, with $\delta = \nu - \mu$,
and we identify it with $\Mat_{r,s}(\C)$. \\

The conjugacy action  of the unipotent component of $G_{p,1}^{(0)}(A)$ 
leaves stable each $\g^{(\delta,\overline{c})}$. Recall its Zariski-generator 
$\tau$ defined at the very end of \ref{subsection:firstlookstructure}.
Writing:
$$
U := \tau(A) = \tau(A_{0}) = 
\begin{pmatrix}
A_{1,u} & \ldots & \ldots & \ldots & \ldots \\
\ldots & \ldots & \ldots  & 0 & \ldots \\
0      & \ldots & \ldots   & \ldots & \ldots \\
\ldots & 0 & \ldots  & \ldots & \ldots \\
0      & \ldots & 0       & \ldots & A_{k,u}  
\end{pmatrix},
$$
we see that:
$$
U \, \g^{(\delta,\overline{c})} \, U^{-1} = \g^{(\delta,\overline{c})}.
$$

Now fix $M_{0},A_{0}$ in $\EE_{p,1}^{(0)}$ as above and call $\rho_0$
the attached representation of $G_{p,1}^{(0)}$. We consider objects 
$M,A$ in $\EE_{1}^{(0)}$ above $M_{0},A_{0}$ (that is, $\gr M = M_0$). 
By tannakian duality, they correspond to representations $\rho$ of 
$G_1^{(0)} = \St \rtimes G_{p,1}^{(0)}$ which restrict to $\rho_0$ on 
$G_{p,1}^{(0)}$. These representations $\rho$ are in turn in one to
one correspondance with representations of $\St$ that are compatible
with $\rho_0$. Translated in terms of representations of $\st$, this
gives:

\begin{prop}
Those representations of $\st$ corresponding to objects $M,A$ above 
$M_{0},A_{0}$ are exactly those such that:
\begin{enumerate}
\item{Each $\st^{(\delta,\overline{c})}$ is mapped to 
$\g^{(\delta,\overline{c})}$;}
\item{The conjugation by $\tau$ in $\st$ is intertwined with the conjugation 
by $U$ in $\g$, \ie\ $\rho(\tau D \tau^{-1}) = U \rho(D) U^{-1}$.}
\end{enumerate}
We write $\Repa(\st)$ the set of these representations.
\end{prop}
\begin{proof}
Indeed, the first condition expresses compatibility with the semi-simple
component of the representation $\rho_0$.
\end{proof}

In this paper, we shall extend the definition of the Stokes Lie algebra
to include the fuchsian unipotent component and put:
$$
\stt := \st \rtimes \Lie(G_{f,u}^{(0)}) = \C \nu \oplus \st,
$$
that is, $\stt$ is generated by $\st$ and by $\Lie(G_{f,u}^{(0)}) = \C \nu$,
where $\nu := \log \tau$. Since $G_{f,u}^{(0)}$ commutes with 
$G_{f,s}^{(0)} \times T_{1}^{(0)}$, the adjoint action of this group 
on $\C \nu$ is trivial and we write $\stt^{(\chi)} := \st^{(\chi)}$ 
and $\stt^{(0)} := \C \nu$.

\begin{coro}
Those representations of $\stt$ corresponding to objects $M,A$ above 
$M_{0},A_{0}$ are exactly those such that:
\begin{enumerate}
\item{Each $\stt^{(\delta,\overline{c})}$ is mapped to 
$\g^{(\delta,\overline{c})}$;}
\item{The element $\nu$ is mapped to $\log U$.}
\end{enumerate}
We write $\Repa(\stt)$ the set of these representations.
\end{coro}
\hfill $\Box$


\subsection{Explicit generators of $\St$}
\label{section:explicitgeneratorsofSt}

Let $A_{0},A$ be as in \eqref{eqn:formesstandards}. Then, there is 
a unique $F \in \G(\Kf)$ such that $F[A_{0}] = A$. We write it 
$\hat{F}_{A}$. The components of the $(i,j)$ block of $\hat{F}_{A}$ 
have $q$-Gevrey level $\delta := \mu_{j} - \mu_{i}$, meaning that 
they are divergent formal series with coefficients $a_k$ having 
a growth of order $q^{k^{2}/2 \delta}$ (up to some $O(R^{k})$ factor). 
Stokes operators, to be defined herebelow, are obtained by ``summing''
this formal object in various directions then taking quotients of 
such summations (ambiguities). We consider as candidate ``directions 
of summation'' the $q$-spirals $[c;q]$ in $\C^{*}$, equivalently, 
the points $\overline{c} \in \Eq$. Define:
$$
\Sigma_{A_{0}} := \{\overline{c} \in \Eq \tq 
q^{\Z} c^{\mu_{i}} \Sp(A_{i}) \cap q^{\Z} c^{\mu_{j}} \Sp(A_{j}) \neq \emptyset
\text{~for some~} 1 \leq i < j \leq k\},
$$
thus a finite subset of $\Eq$. Then \cite{JSStokes}:

\begin{prop}
For all $\overline{c} \in \Eq \setminus \Sigma_{A_{0}}$, there is a unique 
$F \in \G(\Kw)$ such that $F[A_{0}] = A$ and subject to the following
constraints: components of the $(i,j)$ block are meromorphic over $\C^{*}$ 
with at worst poles over $[-c;q]$, of order $\leq \mu_{j} - \mu_{i}$.
\end{prop}
\hfill $\Box$

One proves in \cite{RSZ} that, in some adequate sense, this $F$ is 
asymptotic to $\hat{F}_{A}$. We write it $S_{\overline{c}} \hat{F}_{A}$
and we consider it as a \emph{summation of $S_{\overline{c}} \hat{F}_{A}$
in the ``direction'' $\overline{c} \in \Eq$}. Thus, elements of 
$\Sigma_{A_{0}}$ are \emph{prohibited} directions of summation. 
The Stokes operators are then defined as:
$$
S_{\overline{c},\overline{d}} \hat{F}_{A} :=
\left(S_{\overline{c}} \hat{F}_{A}\right)^{-1} S_{\overline{d}} \hat{F}_{A}.
$$
These are meromorphic automorphisms of $A_{0}$, and they are galoisian 
in the following sense: evaluating them at a fixed base point $a \in \C^*$ 
that is not a pole will yield elements of $\St(A)$ for the corresponding 
fiber functor $\homa$. More precisely \cite{RS2}:

\begin{prop}
For all $\overline{c}, \overline{d} \in \Eq \setminus \Gamma$ such that
$a \not\in [-c;q] \cup [-d;q]$ (so that $a$ is not a pole):
$$
S_{\overline{c},\overline{d}} \hat{F}_{A}(a) \in \St(A)
$$
and these elements, together with their conjugates under the action of 
$G_{p,1}^{(0)}(A)$, are Zariski-generators of $\St(A)$. 
\end{prop}
\hfill $\Box$

Since 
$S_{\overline{c},\overline{d}} \hat{F}_{A} =
(S_{\overline{c_0},\overline{c}} \hat{F}_{A})^{-1}
S_{\overline{c_0},\overline{d}} \hat{F}_{A}$, we may as well fix $\overline{c_0}$
and consider the family of all $S_{\overline{c_0},\overline{c}} \hat{F}_{A}(a)$.
The question of their relations thus comes next.


\subsection{Explicit generators of $\st$}
\label{section:explicitgeneratorsofst}

In order to try to``free'' these generators, one goes to the Lie algebra.
Fix an arbitrary $\overline{c_{0}} \in \Eq \setminus \Sigma_{A_{0}}$. 
For a given $A$, the map:
$$
\overline{c} \mapsto \log S_{\overline{c_{0}},\overline{c}} \hat{F}_{A}(a)
$$
is meromorphic on $\Eq$ with poles on $\Sigma_{A_{0}}$, with values 
in $\st(A)$. Its residue at $\alpha \in \Sigma_{A_{0}}$ is written:
$$
\Delta_{\alpha}(A) := 
\Res_{\beta = \alpha} \log S_{\overline{c_{0}},\beta} \hat{F}_{A}(a)\in \st(A).
$$
Residues at points $\alpha \not \in \Sigma_{A_{0}}$ are null, except maybe at 
the particular point $\overline{a}$, where $a$ encodes the fiber functor; 
but this one has no intrinsic significance and we shall have no use for it. \\

Now the above statement may be reinforced as follows. From \cite{RS1,RS2},
it follows that the mapping $A \mapsto \Delta_{\alpha}(A)$ is functorial
and tensor compatible in the sense of the Stokes Lie algebra (see section
\ref{subsection:firstlookstructure}) when defined on all operands; by
continuity, this remains true without condition:

\begin{lemm}
Each mapping $A \leadsto \Delta_{\alpha}(A)$ defines an element $\Delta_{\alpha}$ 
of $\st$.
\end{lemm}
\hfill $\Box$

According to \eqref{decompositiondest} in \ref{subsection:firstlookstructure}, 
$\Delta_{\alpha}$ admits a decomposition:
$$
\Delta_{\alpha} = \hat{\bigoplus} \Delta_{\alpha}^{(\delta)}, \quad
\Delta_{\alpha}^{(\delta)} = \hat{\bigoplus} \Delta_{\alpha}^{(\delta,\overline{c})}.
$$
We see the components $\Delta_{\alpha}^{(\delta,\overline{c})}$ as $q$-analogs
of alien derivations. It was proved in \cite[theorem 3.5]{RS2} (with slightly
different notations) that: 

\begin{prop}
The $\Delta_{\alpha}$, together with their conjugates under the action of 
$G_{p,1}^{(0)}$, are topological generators of $\st$. 
\end{prop}
\hfill $\Box$

From the preceding section, we draw:

\begin{theo}
\label{theo:generateursdetilde(st)}
The ``$q$-alien derivations'' $\Delta_{\alpha}^{(\delta,\overline{c})}$ together 
with $\nu$ generate topologically the Lie algebra $\stt$.
\end{theo}
\hfill $\Box$

\begin{rema}
It was conjectured at the end of \cite{RS2} that those ``$q$-alien 
derivations'' $\Delta_{\alpha}^{(\delta,\overline{c})}$ such that 
$\alpha^{\delta} = \overline{c}$ (remember we use a multiplicative notation
for the group $\Eq$), together with their conjugates under the action of 
$G_{p,1}^{(0)}$, are topological generators of $\st$. This will be proved
in section \ref{subsection:linkinganalyticclassificationtorepresentations}.
Therefore, those $\Delta_{\alpha}^{(\delta,\overline{c})}$ such that 
$\alpha^{\delta} = \overline{c}$ together with $\nu$ generate topologically 
the Lie algebra $\stt$. The condition on $\alpha,\delta,\overline{c}$ can 
be interpreted in terms of ``directions of maximal growth'' as in the theory 
of differential equations. \\
From considerations related to the classification theory (see section
\ref{subsection:resultsonclassif}), one can predict that these generators 
are not free: there should be $\delta$ of them for each pair 
$\delta,\overline{c}$, but there are $\delta^{2}$. In this respect, 
the ``freeness theorem'' of \cite{RS2} is quite incomplete. We shall
here complete it by theorem \ref{theo:freenesstheorem} at the end of
\ref{subsection:freeingalienderivations}.
\end{rema}


\subsection{$q$-Gevrey interpolation}
\label{subsection:qGevreyinterpolation}

Here, we use \cite[\S 3.3.3]{RS2}. For each level 
$\delta \in \N \cup \{\infty\}$, we define a category $\EE^{\delta}$ 
with the same objects as $\EE_{1}^{(0)}$ but morphisms having coefficients 
in the field of $q$-Gevrey series of level $> \delta$ (see definition
at the beginning of \ref{section:explicitgeneratorsofSt}). For 
$\delta = \infty$, the morphisms are analytic and
$\EE^{\infty} = \EE_{1}^{(0)}$. For $\delta = 0$, any $\hat{F}_{A}$
is a morphism, so that any $A$ is equivalent to $A_{0}$ and
$\EE^{0} = \EE_{p,1}^{(0)}$. In between, the interpolating categories 
$\EE^{\delta}$ are related by essentially surjective and (not fully)
faithful $\otimes$-compatible inclusion functors 
$\EE^{\delta} \hookrightarrow \EE^{\delta-1}$, whence the following
diagram:
$$
\xymatrix{
\EE_{1}^{(0)} \ar@{=}[d] \ar@<0ex>[rrrr]^{\gr} & & & &
\EE_{p,1}^{(0)} \ar@{=}[d] \\
\EE^{\infty} \ar@<0ex>[r] \ar@<0ex>[drr]_\homa & 
\cdots \ar@<0ex>[r] &
\EE^{\delta} \ar@<0ex>[r] \ar@<0ex>[d] & 
\cdots \ar@<0ex>[r] &
\EE^{0} \ar@<0ex>[dll]^\oma \\
& & \Vect_{\C}^f & & 
}
$$
Each $\EE^{\delta}$ is tannakian, with the same fiber functors as 
$\EE_{1}^{(0)}$, and its Galois group is a closed subgroup of $G_{1}^{(0)}$ 
(its elements are $\otimes$-automorphisms of the fiber functor with more 
constraints imposed by functoriality since there are more morphisms;
this is a particular case of \cite[prop. 2.21 (b), p. 139]{DM}). 
Actually:
$$
\Gal(\EE^{\delta}) = \St^{\leq \delta} \rtimes G_{p,1}^{(0)},
$$
where $\St^{\leq \delta}$ is the subgroup of $\St$ with Lie 
algebra\footnote{This was denoted $\st(\delta)$ in \cite{RS2}.}:
$$
Lie(\St^{\leq \delta}) = \st^{\leq \delta} :=
\sum_{\delta' \leq \delta} \st^{(\delta')}.
$$
Thus, $\st^{\leq \delta}$
contains in particular all the $\Delta_{\alpha}^{(\delta',\overline{c})}$
for $\delta' \leq \delta$. \\

We now define:
$$
\stt^{\leq \delta} := \sum_{\delta' \leq \delta} \stt^{(\delta')} =
\C \nu \oplus \st^{\leq \delta}.
$$
Then, from what was said before and the grading, one draws:

\begin{prop}
The Lie algebra $\stt^{\leq \delta}$ is generated by $\nu$ and the
$\Delta_{\alpha}^{(\delta',\overline{c})}$ such that 
$\alpha^{\delta'} = \overline{c}$ for $\delta' \leq \delta$.
\end{prop}
\hfill $\Box$



\section{Structure of the Stokes component}
\label{section:structStokes}

In this section, we shall describe in detail the structure of $\st$
and its representations. We first recall some necessary facts about 
classification. 


\subsection{Some useful results on local analytic classification}
\label{subsection:resultsonclassif}

These results come from \cite{RSZ,JSStokes}. Fix a pure module $M_{0}$
with matrix $A_{0}$ in form \eqref{eqn:formesstandards}. The modules 
formally equivalent to $M_0$ are those such that $\gr M \approx M_0$.
In order to classify them analytically, one rigidifies the situation
by introducing ``marked pairs'' $(M,g)$ made up of an analytic 
$q$-difference module $M$ and an isomorphism $g: \gr(M) \rightarrow M_{0}$.
We then define two such marked pairs $(M,g)$ and $(M',g')$ to be
equivalent if there exists a morphism $f: M \rightarrow M'$ such that 
$g = g' \circ \gr(f)$. By standard commutative algebra, such a morphism 
$f$ is automatically an isomorphism. \\

The set of equivalence classes of marked pairs is written $\F(M_{0})$ and 
we see it as the space of \emph{isoformal analytic classes in the formal 
class of $M_{0}$}. The corresponding classification problem was solved 
in \cite{RSZ} and we shall use it in
\ref{subsection:linkinganalyticclassificationtorepresentations} to get 
an alternative description of $\Repa(\st)$. \\

We define the sheaf $\Lambda_{I}(M_{0})$ of \emph{meromorphic automorphisms 
of $M_{0}$ infinitely tangent to identity} as: 
$$
\Lambda_{I}(M_{0})(V) := \{F \in \G(\O(\pi^{-1}(V))) \tq F[A_{0}] = A_{0}\}.
$$
($V$ denoting an open subset of $\Eq$.) This is a sheaf of unipotent 
groups over $\Eq$, and it is abelian only in the case that $M_{0}$ has 
one or two slopes; in the former case, it is trivial, in the latter case, 
it is a vector bundle \cite{JSStokes}. \\

Now let $M$ in the formal class of $M_{0}$, with matrix $A$ in form 
\eqref{eqn:formesstandards}. The family of all the
$S_{\overline{c},\overline{d}} \hat{F}_{A}$ for all
$\overline{c},\overline{d} \in \Eq \setminus \Sigma_{A_{0}}$ is a cocycle
for the above sheaf:
$$
(S_{\overline{c},\overline{d}} \hat{F}_{A})_{\overline{c},\overline{d}}
\in Z^{1}(\U,\Lambda_{I}(M_{0})).
$$
Here, $\U$ is the covering of $\Eq$ by the Zariski open sets 
$\Eq \setminus \{\overline{-c}\}$, 
$\overline{c} \in \Eq \setminus \Sigma_{A_{0}}$. The conditions on the poles
of summations $S_{\overline{c}} \hat{F}_{A}$ imply that each
$S_{\overline{c},\overline{d}} \hat{F}_{A}$ has only poles 
on $[-c;q] \cup [-d;q]$, with multiplicities $\leq \mu_{j} - \mu_{i}$
for the coefficients of the block $(i,j)$. We call \emph{privileged}
such a cocycle and write $Z^{1}_{pr}(\U,\Lambda_{I}(M_{0}))$ the space
of privileged cocycles.

\begin{theo} \cite{RSZ,JSStokes}
The maps sending $A$ to this cocycle and the latter to its cohomology class
induce isomorphisms:
$$
\F(M_{0}) \rightarrow Z^{1}_{pr}(\U,\Lambda_{I}(M_{0}))
\rightarrow H^{1}(\Eq,\Lambda_{I}(M_{0})).
$$
\end{theo}

We now describe a $q$-Gevrey interpolation of this classification. Write
$\G^{\geq \delta}$ the subgroup of $\G$ defined by the vanishing of all
blocks $(i,j)$ such that $0 < \mu_{j} - \mu_{i} < \delta$. This is a normal 
subgroup of $\G$ and each quotient $\G^{\geq \delta}/\G^{\geq \delta+1}$
is abelian, indeed isomorphic to $\g^{(\delta)}$, whence an exact sequence:
$$
0 \rightarrow \g^{(\delta)} \rightarrow \G/\G^{\geq \delta+1}
\rightarrow \G/\G^{\geq \delta} \rightarrow 1.
$$
This is actually a central extension. It induces a central extension of
sheaves:
$$
0 \rightarrow \lambda_{I}^{(\delta)}(M_{0}) \rightarrow 
\Lambda_{I}(M_{0})/\Lambda_{I}^{\geq \delta+1}(M_{0})
\rightarrow \Lambda_{I}(M_{0})/\Lambda_{I}^{\geq \delta}(M_{0}) \rightarrow 1.
$$
The sheaf 
$\lambda_{I}(M_{0})^{(\delta)} := 
\Lambda_{I}^{\geq \delta}(M_{0})/\Lambda_{I}^{\geq \delta+1}(M_{0})$
is a sheaf of abelian groups, actually a vector bundle over $\Eq$,
corresponding by the construction at the beginning of section
\ref{section:previousresults} to a $q$-difference module that is
pure isoclinic of slope $\delta$: it is indeed the direct sum of
the equations $\sq f = (z^{\mu_{i}} A_{i}) f (z^{\mu_{j}} A_{j})^{-1}$
for $\mu_{j} - \mu_{i} = \delta$. Now, using some non-abelian cohomology
from \cite{Frenkel}, one gets an exact sequence:
\begin{equation}
\label{equation:lasuiteexacte}
0 \rightarrow V^{(\delta)} \rightarrow \F^{\leq \delta}(M_{0})
\rightarrow \F^{\leq \delta-1}(M_{0}) \rightarrow 1.
\end{equation}
The meaning of this sequence is the following:
\begin{enumerate}
\item{The leftmost term 
$V^{(\delta)} := H^{1}(\Eq,\lambda_{I}^{(\delta)}(M_{0}))$
is a finite dimensional complex vector space (first cohomology of 
a vector bundle); its dimension is:
$$
\dim_{\C} V^{(\delta)} = \delta \sum\limits_{\mu_{j} - \mu_{i} = \delta} r_{i} r_{j}.
$$}
\item{The group $V^{(\delta)}$ operates freely on the mid term, which
is defined as the cohomology pointed set
$\F^{\leq \delta}(M_{0}) := 
H^{1}(\Eq,\Lambda_{I}(M_{0})/\Lambda_{I}^{\geq \delta+1}(M_{0}))$.
(The special point of this pointed set is the class of the trivial
cocycle all of whose components are the identity.)}
\item{The corresponding quotient map is the canonical arrow from
$\F^{\leq \delta}(M_{0})$ to the cohomology pointed set
$\F^{\leq \delta-1}(M_{0}) := 
H^{1}(\Eq,\Lambda_{I}(M_{0})/\Lambda_{I}^{\geq \delta}(M_{0}))$.}
\end{enumerate}
Thus, the fibers\footnote{Actually, each $\F^{\leq \delta}(M_{0})$ 
can be endowed with an affine structure over the vector space 
$\bigoplus\limits_{k \leq \delta} V^{(k)}$, but we shall not need 
this fact.} of $\F^{\leq \delta}(M_{0}) \rightarrow \F^{\leq \delta-1}(M_{0})$
inherit a natural structure of affine space over the vector space 
$V^{(\delta)}$. Accordingly, for $v \in V^{(\delta)}$, we shall write 
$\alpha \mapsto v \oplus \alpha$ the translation by $v$ in 
$\F^{\leq \delta}(M_{0})$ (that is, in each of the fibers just mentioned);
and for two class $\alpha,\alpha' \in \F^{\leq \delta}(M_{0})$ having the 
same image in $\F^{\leq \delta-1}(M_{0})$, we shall write
$\alpha' \ominus \alpha$ the unique element of $V^{(\delta)}$ such that
$\alpha' = v \oplus \alpha$. \\

The interpretation of $\F^{\leq \delta}(M_{0})$ in terms of classification
rests on the same interpolating categories $\EE^{\delta}$ as in subsection
\ref{subsection:qGevreyinterpolation}. An object of $\EE^{\delta}$ can be
identified with a matrix $A$ in $\GL_{n}(\Ka)$, with undetermined blocks
$(i,j)$ for $\mu_{j} - \mu_{i} > \delta$, symbolized here by $\star$:
$$
\begin{pmatrix}
z^{\mu_{1}} A_{1} & \ldots & \ldots & \star  & \star  & \star \\
\ldots & \ldots & \ldots  & \ldots & \star  & \star  \\
\ldots & \ldots & \ldots  & U_{i,j} & \ldots & \star  \\
0      & \ldots & \ldots   & \ldots & \ldots & \ldots \\
\ldots & 0 & \ldots  & \ldots & \ldots & \ldots \\
0      & \ldots & 0       & \ldots & \ldots & z^{\mu_{k}} A_{k}   
\end{pmatrix}
$$
The highest meaningful block diagonal consists in blocks $U_{i,j}$
with level $\mu_{j} - \mu_{i} = \delta$. \\

We fix a block diagonal matrix $A_{0}$ and we classify all matrices $A$ 
with diagonal $A_{0}$, up to $q$-Gevrey gauge equivalence of level
$> \delta$, that is under transforms in $\G(\Kf)$ all of whose coefficients
are series of $q$-Gevrey level $> \delta$. This amounts to the same as 
fixing the pure module $M_{0}$ and doing $q$-Gevrey classification in 
its formal class. The space of isoformal classes above $A_{0}$ in 
$\EE^{\delta}$ received a cohomological description in \cite{RSZ}: 
it is $\F^{\leq \delta}(M_{0})$. Using the Birkhoff-Guenther normal 
form (\emph{loc. cit.}), one can moreover require null blocks $(i,j)$ 
for $\mu_{j} - \mu_{i} > \delta$ and find its dimension as an affine space. 
We shall write $cl(A)$ the class of the module $M_A$ in 
$\F^{\leq \delta}(A_{0}) := \F^{\leq \delta}(M_{0})$. \\

In \ref{subsection:linkinganalyticclassificationtorepresentations},
we shall have use for the corresponding computational description of
the exact sequence \eqref{equation:lasuiteexacte}. Consider $A,A'$ 
in $\F^{\leq \delta}(M_{0})$ having the same image in
$\F^{\leq \delta-1}(M_{0})$. Then 
$\hat{F}_{A,A'} := \hat{F}_{A'} (\hat{F}_{A})^{-1}$ lies in 
$\G^{\geq \delta}(\Kf)$, as well as its summations:
$$
S_{\overline{c}} \hat{F}_{A,A'} := 
S_{\overline{c}} \hat{F}_{A'} (S_{\overline{c}} \hat{F}_{A})^{-1}
$$
We get a cocycle:
$$
S_{\overline{c},\overline{d}} \hat{F}_{A,A'} :=
\left(S_{\overline{c},\overline{d}} \hat{F}_{A,A'}\right)^{-1}
S_{\overline{c},\overline{d}} \hat{F}_{A,A'}
$$
of $\Lambda_{I}(M_{0})$, in which the blocks for  $\mu_{j} - \mu_{i} > \delta$
have no meaning and those for $0 < \mu_{j} - \mu_{i} < \delta$ vanish; thus,
it yields a well defined privileged cocycle of 
$\lambda_{I}(M_{0})^{(\delta)} := 
\Lambda_{I}^{\geq \delta}(M_{0})/\Lambda_{I}^{\geq \delta+1}(M_{0})$,
whence a class in $V^{(\delta)} := H^{1}(\Eq,\lambda_{I}^{(\delta)}(M_{0}))$.
This class is the element $cl(A') \ominus cl(A) \in V^{(\delta)}$ which
sends the class of $A$ to the class of $A'$ in $\F^{\leq \delta}(M_{0})$.


\subsection{Linking representations of $\st$ to isoformal analytic classes}
\label{subsection:linkingisoformalclassestorepresentations}

Let $M_0$ an object of $\EE_{p,1}^{(0)}$ Its fiber by the functor $\gr$ 
from $\EE_{1}^{(0)}$ to $\EE_{p,1}^{(0)}$ can be identified with the
category $\Co(M_0)$ with objects the pairs $(M,u)$, $M$ an object 
of $\EE_{1}^{(0)}$ and $u: \gr M \rightarrow M_0$ an isomorphism; and 
with morphisms $(M,u) \rightarrow (N,v)$ the morphisms $f: M \rightarrow N$
in $\EE_{1}^{(0)}$ such that $u = v \circ \gr f$. Such a morphism is
automatically an isomorphism so that $\C(M_0)$ is a groupoid and $\F(M_0)$
is the set $\pi_0(\C(M_0))$ of its connected components. Its cohomological
description was explained in \ref{subsection:resultsonclassif}, we now
use tannakian duality to get a representation theoretic description. \\

To alleviate notations, in this section, we respectively write $\Co,\Co_0$ 
for $\EE_{1}^{(0)},\EE_{p,1}^{(0)}$ and $G,G_0$ for their Galois groups
$G_1^{(0)} = \Gal(\EE_{1}^{(0)}),G_{p,1}^{(0)} = \Gal(\EE_{p,1}^{(0)})$.
We write $\Repc(G),\Repc(G_0)$ the categories of complex finite dimensional
rational representations of these proalgebraic groups. The choice of the 
fiber functors is here irrelevant, all that we need is the equivalences
of category $\Co$ with $\Repc(G)$ and of category $\Co_0$ with $\Repc(G_0)$. 
Objects of $\Repc(G)$ are rational morphisms $\rho: G \rightarrow \GL(V)$,
$V$ a finite dimensional complex space; and morphisms in $\Repc(G)$ from
$\rho: G \rightarrow \GL(V)$ to $\rho': G \rightarrow \GL(V')$ are linear
maps $\phi: V \rightarrow V'$ such that 
$\forall g \in G \;,\; \rho'(g) \circ \phi = \phi \circ \rho(g)$. The
category $\Repc(G_0)$ is similarly defined. \\

We also introduce the auxiliary comma-category $\overline{\Co}$ with objects
the triples $(M,M_0,u)$ where $M,M_0$ are objects of $\Co,\Co_0$ and where
$u: \gr M \rightarrow M_0$ is an isomorphism; and with morphisms 
$(M,M_0,u) \rightarrow (N,N_0,v)$ the pairs $(f,f_0)$ made up of a 
morphism $f: M \rightarrow N$ and of a morphism $f_0: M_0 \rightarrow N_0$
such that $f_0 \circ u = v \circ \gr f$.

\begin{lemm}
The category $\overline{\Co}$ is equivalent to $\Co$ and we can identify 
the fiber $\Co(M_0)$ described above with the fiber $\overline{\Co}(M_0)$.
\end{lemm}
\begin{proof}
Let $F$ be the functor from $\Co$ to $\overline{\Co}$ defined by
$M \leadsto (M, \gr M, \Id_{\gr M})$ and $f \leadsto (f,\gr f)$ and
let $G$ be the forgetful functor from $\overline{\Co}$ to $\Co$.
Then $G \circ F$ is the identity functor of $\Co$ and $F \circ G$
is isomorphic to the identity functor of $\overline{\Co}$ by the
natural transformation which sends $X = (M,M_0,u)$ to the morphism
$(\Id_M,u)$ from $F \circ G(X) = (M,\gr M,\Id_{\gr M})$ to $X$.
\end{proof}

We now carry on this construction to the equivalent categories $\Repc(G)$ 
and $\Repc(G_0)$. The inclusion $\Co_0 \hookrightarrow \Co$ and its 
retraction $\gr: \Co \rightarrow \Co_0$ correspond dually to morphisms 
$G_0 \overset{\gr^*}{\rightarrow} G \overset{\pi}{\rightarrow} G_0$
such that $\pi \circ \gr^* = \Id_{G_0}$. We shall use $\gr^*$ to
identify $G_0$ to a (proalgebraic) subgroup of $G$. The Stokes
group is $\St = \Ker \pi$ and one deduces from these morphisms
the equality $G = \St \rtimes G_0$. Thus, $G_0$ acts upon $\St$
by inner automorphisms, which we shall denote $s \mapsto s^g := g^{-1} s g$.
We also shall denote $D \mapsto D^g$ the corresponding adjoint action
on the Lie algebra $\st$ of $\St$. \\

The functor $\Co_0 \hookrightarrow \Co$ is thereby identified with
the following functor from $\Repc(G_0)$ to $\Repc(G)$:
\begin{align*}
\left(\rho_0: G_0 \rightarrow \GL(V)\right) & \leadsto
\left(\rho_0 \circ \pi: G \rightarrow \GL(V)\right), \\
(\phi: V \rightarrow V', \rho_0 \rightarrow \rho_0') & \leadsto
(\phi: V \rightarrow V', \rho_0 \circ \pi \rightarrow \rho_0' \circ \pi).
\end{align*}
Similarly, the functor $\gr: \Co \rightarrow \Co_0$ is identified with
the following functor from $\Repc(G)$ to $\Repc(G_0)$:
\begin{align*}
\left(\rho: G \rightarrow \GL(V)\right) & \leadsto
\left(\rho_{|G_0}: G_0 \rightarrow \GL(V)\right), \\
(\phi: V \rightarrow V', \rho \rightarrow \rho') & \leadsto
(\phi: V \rightarrow V', \rho_{|G_0} \rightarrow \rho'_{|G_0}).
\end{align*}
Since $(\rho_0 \circ \pi)_{|G_0} = \rho_0$, the composition is the
identity of $\Repc(G_0)$ as it should. Then one checks that $\overline{\Co}$
is identified to the category of triples $(\rho,\rho_0,u)$, where
$\rho: G \rightarrow \GL(V)$ and $\rho_0: G_0 \rightarrow \GL(V_0)$ 
are rational representations and where $u: V \rightarrow V_0$ is an
isomorphism from $\rho_{|G_0}$ to $\rho_0$, with morphisms from
$(\rho,\rho_0,u)$ to $(\rho',\rho'_0,u')$ the pairs $(\phi,\phi')$
where $\phi: V \rightarrow V'$ and $\phi': v' \rightarrow V'_0$
yield morphisms $\rho \rightarrow \rho'$ and $\rho_0 \rightarrow \rho'_0$ 
of representations and where moreover $\phi_0 \circ u = u' \circ \phi$.
The equivalences of $\Co$ and $\overline{\Co}$ are easy to explicit. \\

Last, if $M_0$ ``is'' the representation $\rho_0: G_0 \rightarrow \GL(V)$,
the fiber $\overline{\Co}(M_0)$ is identified with the category with
objects the pairs $(\rho,u)$ of a rational representation
$\rho: G \rightarrow \GL(V)$ and a map $u: V \rightarrow V_0$ which is 
an isomorphism from $\rho_{|G_0}$ to $\rho_0$, with morphisms from
$(\rho,u)$ to $(\rho',u')$ the maps $\phi: V \rightarrow V'$ which
yield morphisms $\rho \rightarrow \rho'$ such that $u = u' \circ \phi$.

\begin{lemm}
The fiber $\overline{\Co}(M_0)$ can be identified with the set of
representations $\rho: G \rightarrow \GL(V_0)$ such that $\rho_{|G_0} = \rho_0$.
\end{lemm}
\begin{proof}
This set is considered as a category having only identity morphisms. The 
identification comes from the functor which sends the object $(\rho,u)$ 
to the representation $\rho_u: g \mapsto u \circ \rho(g) \circ u^{-1}$ and
every morphism to the corresponding identity morphism. This is a retraction
of the obvious inclusion, and an equivalence of categories.
\end{proof}

Now we return to our more concrete setting, with $G = \St \rtimes G_0$.
If $\rho_0: G_0 \rightarrow \GL(V)$ is fixed, to specify a representation
$\rho: G \rightarrow \GL(V_0)$ such that $\rho_{|G_0} = \rho_0$, we need 
only to give its restriction $\overline{\rho}$ to $\St$, and this is
subject to the necessary and sufficient condition:
$$
\forall s \in \St \;,\; \forall g \in G_0 \;,\;
\overline{\rho}(s^g) = (\overline{\rho}(s))^{\rho_0(g)}.
$$
Since $\St$ is connected and prounipotent, $\overline{\rho}$ is determined
by the corresponding representation of the Lie algebra $\st$. In the end,
we have proved:

\begin{prop}
The fiber $\overline{\Co}(M_0)$ can be identified with the set:
$$
\{\rho: \st \rightarrow \gl(V_0) \tq 
\forall D \in \st \;,\; \forall g \in G_0 \;,\; 
\rho(D^g) = (\rho(D))^{\rho_0(g)}\}.
$$
\end{prop}
\hfill $\Box$


\subsection{Linking representations of $\st$ with 
$H^{1}(\Eq,\lambda_{I}(M_{0}))$}
\label{subsection:linkinganalyticclassificationtorepresentations}

The bijection of $H^{1}(\Eq,\lambda_{I}(M_{0}))$ with $\Repa(\stt)$
resulting from the two descriptions of $\F(M_0)$ (see sections
\ref{subsection:firstlookrepresentations}, \ref{subsection:resultsonclassif}
and \ref{subsection:linkingisoformalclassestorepresentations}) is obtained
as follows: for any matrix $A$ corresponding to a class in $\F(M_0)$,
first compute the privileged cocycle
$(S_{\overline{c},\overline{d}} \hat{F}_{A}) \in Z^{1}_{pr}(\U,\Lambda_{I}(M_{0}))$.
Write temporarily $h(A)$ its class in $H^{1}(\Eq,\lambda_{I}(M_{0}))$. 
On the other hand, write $D_{\alpha}$ the residue at 
$\beta = \alpha$ of the meromorphic function 
$\beta \mapsto \log S_{\overline{c_0},\beta} \hat{F}_{A}(z_0) \in \g$,
and $\Delta_{\alpha}^{(\delta,\overline{c})}(A)$ the components of 
$D_{\alpha}$ for $\alpha^{\delta} = \overline{c}$. Then call
$\rho(A) \in \Repa(\stt)$ the unique representation which sends the
``$q$-alien derivations'' $\Delta_{\alpha}^{(\delta,\overline{c})}$ to the
$\Delta_{\alpha}^{(\delta,\overline{c})}(A)$ and $\nu$ to $U$ (the block-diagonal
matrix of unipotent components of $A$). The bijection puts in correspondance
the class $h(A)$ with the representation $\rho(A)$. \\

Using $Z^{1}_{pr}(\U,(\Lambda_{I}/\Lambda_{I}^{\geq \delta+1})(M_{0}))$,
we get correspondingly a bijection of 
$H^{1}(\Eq,(\Lambda_{I}/\Lambda_{I}^{\geq \delta+1})(M_{0}))$ with
$\Repa(\stt^{\leq \delta})$. This gives a commutative diagram with
surjective horizontal arrows and bijective vertical arrows:
$$
\xymatrix{
H^{1}(\Eq,\Lambda_{I}(M_{0})) \ar@<0ex>[rr] \ar@<0ex>[d] & &
\Repa(\stt) \ar@<0ex>[d] \\
H^{1}(\Eq,(\Lambda_{I}/\Lambda_{I}^{\geq \delta+1})(M_{0})) \ar@<0ex>[rr] & &
\Repa(\stt^{\leq \delta})
}
$$

Just with this information, we shall now start to get structural
information about $\stt$. Let $c \in \C^*$, $\delta \in \N^*$ and
set $A_0 := \begin{pmatrix} 1 & 0 \\ 0 & c z^{\delta} \end{pmatrix}$.
Then:
$$
\g = \g^{(\delta,\overline{c})} = \begin{pmatrix} 0 & \C \\ 0 & 0 \end{pmatrix}.
$$
Since $\g$ is abelian, $\Repa(\stt)$ can be identified with the dual
space of $\left(\frac{\stt}{[\stt,\stt]}\right)^{(\delta,\overline{c})}$.
But since $H^{1}(\Eq,\Lambda_{I}(M_{0})) = H^{1}(\Eq,\F_{c z^{\delta}})$ 
has dimension $\delta$ \cite{RSZ,RS1}, we conclude:
$$
\dim_\C \left(\frac{\stt}{[\stt,\stt]}\right)^{(\delta,\overline{c})} = \delta.
$$
Applying theorem \ref{theo:generateursdetilde(st)}, we see that (the 
images of) the $\Delta_{\alpha}^{(\delta,\overline{c})}$ for arbitrary
$\alpha \in \Eq$ generate the complex vector space
$\left(\frac{\stt}{[\stt,\stt]}\right)^{(\delta,\overline{c})}$. However,
when computing the residues in the case of a matrix 
$A_u := \begin{pmatrix} 1 & u \\ 0 & c z^{\delta} \end{pmatrix}$, we only
find poles at points such that $\alpha^{\delta} = \overline{c}$. Thus, if
all the $\Delta_{\alpha}^{(\delta,\overline{c})}$ such that 
$\alpha^{\delta} = \overline{c}$ vanish on the class of $A_u$, this class
is trivial. By duality, this means that (the images of) those 
$\Delta_{\alpha}^{(\delta,\overline{c})}$ such that $\alpha^{\delta} = \overline{c}$
generate $\left(\frac{\stt}{[\stt,\stt]}\right)^{(\delta,\overline{c})}$.

\begin{rema}
This is in essence the content of \cite[theorem 4.9]{RS1} and represent
the basic step in the ``abelian (two slopes) case''. In \cite{RS2} we
attempted to obtain the general case by devissage of the filtration.
Here we achieve this devissage in the next two subsections.
\end{rema}

For each pair $(\delta,\overline{c}) \in \N^* \times \Eq$, 
we choose $\delta$ among the $\delta^2$ points $\alpha \in \Eq$ such that 
$\alpha^\delta = \overline{c}$ in such a way that the images of the
corresponding $\Delta_{\alpha}^{(\delta,\overline{c})}$ form a basis
of $\frac{\stt}{[\stt,\stt]}$. We write $\Delta_i^{(\delta,\overline{c})}$,
$1 \leq i \leq \delta$, the corresponding $q$-alien derivations. 
In accordance with the analogy explained in the introduction, we
see them as \emph{``pointed''} and from now on they will be denoted
$\Der_i^{(\delta,\overline{c})}$, and $\nu \in \st^{(0)}$ will be denoted
$\Der^{(0)}$.

\begin{prop}
The family of all $\Der_i^{(\delta,\overline{c})}$ together with 
$\Der^{(0)} := \nu$ topologically generate $\stt$.
\end{prop}
\begin{proof}
Call $S$ the sub-Lie algebra generated by this family. It is naturally
$\N$-graded and one has, for all $\delta \in \N$:
$$
\stt^{(\delta)} = S^{(\delta)} + [\stt,\stt]^{(\delta)} =
S^{(\delta)} + \sum_{i+j = \delta}[\stt^{(i)},\stt^{(j)}].
$$
We shall prove inductively that $S^{(\delta)} = \stt^{(\delta)}$
for all $\delta \in \N$, which will imply the conclusion. For
$\delta = 0$, both sides are equal to $\C \nu$. Assuming it to
be true for all degrees $< \delta$, we calculate:
\begin{align*}
\stt^{(\delta)} &= S^{(\delta)} + \sum_{i+j = \delta}[\stt^{(i)},\stt^{(j)}] \\
&=  S^{(\delta)} + [\C \nu,\stt^{(\delta)}] +
\sum_{i+j = \delta \atop i,j < \delta}[\stt^{(i)},\stt^{(j)}] \\
&=  S^{(\delta)} + [\C \nu,\stt^{(\delta)}] +
\sum_{i+j = \delta \atop i,j < \delta}[S^{(i)},S^{(j)}] \\
&=  S^{(\delta)} + [\C \nu,\stt^{(\delta)}]
\end{align*}
since $[S^{(i)},S^{(j)}] \subset S^{(\delta)}$ when $i + j = \delta$.
Iterating, we get $\stt^{(\delta)} = S^{(\delta)} + \Phi^k(\stt^{(\delta)})$
for all $k$, where $\Phi(E) := [\C \nu,E]$. By the topological 
description of $\stt$ at the end of \ref{subsection:tannakianfiltration},
this ends the proof.
\end{proof}

We will show in \ref{subsection:freeingalienderivations} that this family 
is in some sense free.


\subsection{$q$-Gevrey devissage of the space of representations}

From the $q$-Gevrey d\'evissage of $H^{1}(\Eq,\Lambda_{I}(M_{0}))$
and of $Z_{pr}^{1}(\U,\Lambda_{I}(M_{0}))$, and from the identifications
with $\Repa(\stt)$, we get the following commutative diagram of
exact sequences (for concision, we do not indicate the dependency
on $M_0$):
$$
\xymatrix{
0 \ar@<0ex>[r] &
H^{1}(\Eq,\lambda_{I}^{(\delta)}) \ar@<0ex>[r] &
H^{1}(\Eq,(\Lambda_{I}/\Lambda_{I}^{\geq \delta+1})) \ar@<0ex>[r] &
H^{1}(\Eq,(\Lambda_{I}/\Lambda_{I}^{\geq \delta})) \ar@<0ex>[r] &
0 \\ 
0 \ar@<0ex>[r] &
Z_{pr}^{1}(\U,\lambda_{I}^{(\delta)}) 
\ar@<0ex>[r] \ar@<0ex>[u] \ar@<0ex>[d] &
Z_{pr}^{1}(\U,(\Lambda_{I}/\Lambda_{I}^{\geq \delta+1})) 
\ar@<0ex>[r] \ar@<0ex>[u] \ar@<0ex>[d] &
Z_{pr}^{1}(\U,(\Lambda_{I}/\Lambda_{I}^{\geq \delta})) 
\ar@<0ex>[r] \ar@<0ex>[u] \ar@<0ex>[d] &
0 \\ 
0 \ar@<0ex>[r] &
W^{(\delta)} \ar@<0ex>[r] & 
\Repa(\stt^{\leq \delta}) \ar@<0ex>[r] & 
\Repa(\stt^{\leq \delta-1}) \ar@<0ex>[r] & 
0 \\ 
}
$$
In the last line, being an exact sequence means that $W^{(\delta)}$
is a vector space acting on the middle term with quotient the
rightmost term. We shall now describe this space and this action.
For this, we recall the description given in \cite[prop. 3.3.9]{RS2}
of the fibers of the surjection from $\Repa(\stt^{\leq \delta})$ to
$\Repa(\stt^{\leq \delta-1})$ (modulo the change of notation from
$\st(\delta)$ there to $\stt^{\leq \delta}$ here). Let $B$ an element
of $\Co^{\delta-1}$ in Birkhoff-Guenther normal form having graded part $A_0$.
Two elements $A,A'$ of $\Co^\delta$ lifting $B$ are related by a unique
formal gauge transform $\hat{F}_{A,A'} \in \G(\Kf)$. This matrix has
null blocks for $0 < \mu_j - \mu_i < \delta$; the blocks corresponding
to levels $\mu_j - \mu_i > \delta$ are irrelevant; and we call
$\hat{f}_{A,A'}$ the part of $\hat{F}_{A,A'}$ corresponding to level
$\mu_j - \mu_i = \delta$. The family of all the
$S_{\overline{c},\overline{d}}\hat{f}_{A,A'}$ is a cocycle in
$Z_{pr}^{1}(\U,\lambda_{I}^{(\delta)})$. Then:
$$
\Delta_{\alpha}^{(\delta)}(A,A') :=
\Delta_{\alpha}^{(\delta)}(A') - \Delta_{\alpha}^{(\delta)}(A) =
Res_{\beta = \alpha} 
S_{\overline{c_0},\beta}\hat{f}_{A,A'}(z_0)
$$
Moreover, all families $(\Delta_{\alpha}^{(\delta)})$ arising this
way correspond to a difference $\Delta(A') - \Delta(A)$. We thus define:
$$
W^{(\delta)} := 
\Bigl\{\bigl(
\Res_{\beta = \alpha} S_{\overline{c_0},\beta} f(z_0)\bigr)_{\alpha}^{(\delta,\overline{c})} 
\tq f \in Z_{pr}^{1}(\U,\lambda_{I}^{(\delta)}) \Bigr\}.
$$
If we encode a representation by the family of all the 
$\Der_i^{(\delta,\overline{c})}(A)$, we see that we do get
an action of $W^{(\delta)}$ on $\Repa(\stt^{\leq \delta})$ with quotient
$\Repa(\stt^{\leq \delta-1})$.


\subsection{Freeing the alien derivations}
\label{subsection:freeingalienderivations}

\begin{defi}
Let $\Der^{(0)}$ be a symbol corresponding to to the element $\nu$
introduced in section \ref{subsection:firstlookrepresentations}. 
For each $(\delta,\overline{c}) \in \N^* \times \Eq$ and 
$i = 1,\ldots,\delta$, let $\Der_i^{(\delta,\overline{c})}$ be a symbol 
corresponding to the actual alien derivation with the same notation.
We call $L$ the free Lie algebra generated by $\Der^{(0)}$ and all
the $\Der_i^{(\delta,\overline{c})}$. We graduate it by
the semi-group $\{0\} \cup \N^* \times \Eq$ by taking $\deg \Der^{(0)} = 0$
and $\deg \Der_i^{(\delta,\overline{c})} = (\delta,\overline{c})$.
We also endow $L$ with the following action of $T_1^{(0)} \times G_{f,1}^{(0)}$:
\begin{align*}
(\Der_i^{(\delta,\overline{c})})^{(t,\gamma)} &:= 
t^\delta \gamma(\overline{c}) \Der_i^{(\delta,\overline{c})}, \\
(\Der^{(0)})^{(t,\gamma)} &:= \Der^{(0)}.
\end{align*}
We write $\Repa(L)$ the set of representations from $L$ to $\g$
compatible with this action (and similarly for all stable sub-Lie
algebras of $L$). 
\end{defi}

Write $L^{\leq \delta}$ the sub-Lie algebra generated by $\Der^{(0)}$ and
all the $\Der_i^{(\delta',\overline{c})}$, $\delta' \leq \delta$;
and $L^{> \delta}$ the ideal generated by
all the $\Der_i^{(\delta',\overline{c})}$, $\delta' > \delta$.
After \cite[chap. 2, \S 2, no 9, prop. 10]{Blie}, $L^{\leq \delta}$ is free
with basis the stated system of generators, while $L^{> \delta}$ is free
with basis the family of all 
$(ad D_1) \cdots (ad D_k) \Der_i^{(\delta'',\overline{c})}$
where the $D_i$ are $\Der_i^{(\delta',\overline{c})}$ with
$\delta' \leq \delta$ and where $\delta'' > \delta$.
Moreover, $L = L^{\leq \delta} \oplus L^{> \delta}$. Likewise, we have
$L^{\leq \delta} = L^{\leq \delta-1} \oplus L^{(\delta)}$, where $L^{(\delta)}$ is
the ideal generated by all the $\Der_i^{(\delta,\overline{c})}$. \\

From the obvious dominant morphisms of Lie algebras $L \rightarrow \stt$
and $L^{\leq \delta} \rightarrow \stt^{\leq \delta}$, one deduces a commutative
diagram with surjective horizontal maps and injective vertical maps:
$$
\xymatrix{
\Repa(\stt) \ar@<0ex>[r] \ar@<0ex>[d] & \Repa(\stt^{\leq \delta}) \ar@<0ex>[d] \\
\Repa(L) \ar@<0ex>[r] & \Repa(L^{\leq \delta}) 
}
$$
On the other hand, we have identifications:
\begin{align*}
U := \Repa(L) & \approx 
\bigoplus_{(\delta,\overline{c}) \in \N^* \times \Eq} 
\left(\g^{(\delta,\overline{c})}\right)^\delta \\
U^{\leq \delta} := \Repa(L^{\leq \delta}) & \approx 
\bigoplus_{(\delta',\overline{c}) \in \N^* \times \Eq \atop 1 \leq \delta' \leq \delta} 
\left(\g^{(\delta',\overline{c})}\right)^{\delta'} \\
U^{(\delta)} := \Repa(L^{(\delta)}) & \approx 
\bigoplus_{\overline{c} \in \Eq} 
\left(\g^{(\delta,\overline{c})}\right)^\delta.
\end{align*}
Indeed, the value of the generator $\nu$ is imposed since we consider
representations in $\Repa$, \ie\ relative to the fixed $A_0$. Then,
we can enrich as follows the previous diagram of exact sequences:
$$
\xymatrix{
0 \ar@<0ex>[r] &
H^{1}(\Eq,\lambda_{I}^{(\delta)}) \ar@<0ex>[r] &
H^{1}(\Eq,(\Lambda_{I}/\Lambda_{I}^{\geq \delta+1})) \ar@<0ex>[r] &
H^{1}(\Eq,(\Lambda_{I}/\Lambda_{I}^{\geq \delta})) \ar@<0ex>[r] &
0 \\ 
0 \ar@<0ex>[r] &
Z_{pr}^{1}(\U,\lambda_{I}^{(\delta)}) 
\ar@<0ex>[r] \ar@<0ex>[u] \ar@<0ex>[d] &
Z_{pr}^{1}(\U,(\Lambda_{I}/\Lambda_{I}^{\geq \delta+1})) 
\ar@<0ex>[r] \ar@<0ex>[u] \ar@<0ex>[d] &
Z_{pr}^{1}(\U,(\Lambda_{I}/\Lambda_{I}^{\geq \delta})) 
\ar@<0ex>[r] \ar@<0ex>[u] \ar@<0ex>[d] &
0 \\ 
0 \ar@<0ex>[r] &
W^{(\delta)} \ar@<0ex>[r] \ar@<0ex>[d] & 
\Repa(\stt^{\leq \delta}) \ar@<0ex>[r]  \ar@<0ex>[d] & 
\Repa(\stt^{\leq \delta-1}) \ar@<0ex>[r]  \ar@<0ex>[d] & 
0 \\ 
0 \ar@<0ex>[r] &
U^{(\delta)} \ar@<0ex>[r] & 
\Repa(L^{\leq \delta}) \ar@<0ex>[r] & 
\Repa(L^{\leq \delta-1}) \ar@<0ex>[r] & 
0 \\ 
}
$$
The new vertical arrows are \emph{a priori} injections.

\begin{theo}[Freeness theorem]
\label{theo:freenesstheorem}
The map $\Repa(\stt) \rightarrow \Repa(L)$ is bijective.
\end{theo}
\begin{proof}
By induction, using the two last lines of the above diagram, it is enough 
to show that the leftmost vertical arrow is bijective. But it is linear and 
it sends injectively each $W^{(\delta,\overline{c})}$ to $U^{(\delta,\overline{c})}$,
which has the same dimension $\delta$.
\end{proof}


\subsection{First step in direction of the inverse problem}
\label{subsection:firststelinverseproblem}

Recall from section \ref {subsection:overallstructure} the description
of the pure (or formal) Galois group with integral slopes:
$$
G_{p,1}^{(0)} = \Gal(\EE_{p,1}^{(0)}) = 
\C^{*} \times \Hom_{gr}(\C^{*}/q^{\Z},\C^{*}) \times \C.
$$
In section \ref{subsection:firstlookrepresentations}, we took off its
unipotent component $\C$ and glued it with the Stokes group, so we now
introduce its semi-simple component:
$$
G_{p,1,s}^{(0)} := \C^{*} \times \Hom_{gr}(\C^{*}/q^{\Z},\C^{*}).
$$
It acts as follows on the free Lie algebra $L$: the action on $\C \nu$ 
is trivial; for each $(\delta,\overline{c}) \in \N^* \times \Eq$, the 
action on the component $L^{(\delta,\overline{c})}$ is multiplication by
$t^\delta \gamma(\overline{c})$.

\begin{defi}
The \emph{wild fundamental group} of $\EE_{1}^{(0)}$ is the semi-direct
product $L \rtimes G_{p,1,s}^{(0)}$. A \emph{representation} of the wild
fundamental group is the data of a rational linear representation of 
$G_{p,1,s}^{(0)}$ together with a representation of $L$, required to be
compatible with the corresponding adjoint actions.
\end{defi}

The Tannakian category $Rep\, (L \rtimes G^{(0)}_{p,1,s})$ is by definition 
built with such representations.

Let $\rho: L \rtimes G^{(0)}_{p,1,s} \rightarrow \GL(V)$ be a representation 
of the wild fundamental group in the sense of the above definition. It it 
is easy to check that the restriction $d\rho_2: L \rightarrow  \End(V)$ is 
nilpotent and vanishes on every alien derivation but perhaps a finite number. 
Therefore $d\rho_2$ factors by $L^{\dag}$, $L^{\dag}$ being the f-pronilpotent 
completion of the free Lie algebra $L$ (\cf\ the appendix 
\ref{section:pronil}), which is a pronilpotent  proalgebraic Lie algebra. 

\begin{rema}
The natural morphism $L\rightarrow \tilde{\mathfrak{st}}$ factors into 
$L \rightarrow L^{\dag} \rightarrow  \tilde{\mathfrak{st}}$, the first 
morphism being injective and dominant (\ie\ its image is dense)\footnote{We 
are going to prove below that the second morphism is an isomorphism}. 
We deduce morphisms:
\[
L\rtimes G^{(0)}_{p,1,s} \rightarrow L^{\dag} \rtimes G^{(0)}_{p,1,s} \rightarrow
\tilde{\mathfrak{st}}\,\rtimes G^{(0)}_{p,1,s},
\]
and then functors:
\[
\Rep (\tilde{\mathfrak{st}}\,\rtimes G^{(0)}_{p,1,s})\rightarrow 
\Rep (L^{\dag}\rtimes G^{(0)}_{p,1,s})\rightarrow \Rep(L\rtimes G^{(0)}_{p,1,s}).
\]
It follows from theorem \ref{theo:freenesstheorem} that these are 
equivalences of categories, therefore: 
\[
L^{\dag}\rtimes G^{(0)}_{p,1,s} \rightarrow
\tilde{\mathfrak{st}}\,\rtimes G^{(0)}_{p,1,s}
\]
is an isomorphism in the obvious proalgebraic sense (proposition 
\ref{prop:calcul}) and $L^{\dag} \rightarrow \tilde{\mathfrak{st}}$ is an 
isomorphism of pronilpotent  proalgebraic Lie algebras.

Finally we get an isomorphism of proalgebraic groups:
\[
\exp(L^{\dag})\rtimes G^{(0)}_{p,1,s}\rightarrow 
\exp(\tilde{\mathfrak{st}})\,\rtimes G^{(0)}_{p,1,s} =
\mathfrak{St}\, \rtimes G^{(0)}_{p,1}=G^{(0)}_1.
\]
This is an ``explicit description" of the tannakian group $G^{(0)}_1$.
\end{rema}

To summarize, we have proved:

\begin{theo}
\label{theo:bigth1}
There is a natural bijection between representations of the wild fundamental
group of $\EE_{1}^{(0)}$ and isomorphism classes of objects of $\EE_{1}^{(0)}$. 
All the Galois groups of such objects are images of such representations.
\end{theo}
\hfill $\Box$



\section{Structure of the global Galois group}

We consider here the Galois theory of equations with matrix in $\GL_n(\C(z))$.
We shall not develop the theory in such general terms as we did in the previous
sections, but just enough to be able to apply it to the inverse problem.

\subsection{The global fuchsian Galois group}

We recall here results from \cite{JSGAL}, mostly its subsection 3.2. 
Unhappily, some of the results that we need are not completely proven 
there: details can be found in the thesis ``Th\'eorie de Galois des 
\'equations aux q-diff\'erences fuchsiennes'', available at 
\verb?http://www.math.univ-toulouse.fr/~sauloy/PAPIERS/these.pdf?. 
We shall slightly adapt the notations of \emph{loc. cit.} so that they 
extend more easily to our case of interest in the next section. \\

Let $\EE_f$ be the category with objects the matrices $A \in \GL_n(\C(z))$
which are fuchsian\footnote{All definitions and constructions given at $0$
can be applied at $\infty$ by using the coordinate $w := 1/z$.} at $0$ and 
at $\infty$, and with morphisms $F: A \rightarrow B$ the matrices
$F \in \Mat_{p,n}(\C(z))$ such that $(\sq F) A = B F$. It is endowed with
a natural tensor structure\footnote{The conventions used to obtain a
matrix (and not a quadritensor) as the result of tensoring two matrices are 
detailed in \emph{loc. cit.}.} which makes it a neutral tannakian category.
Each object $A$ of $\EE_f$ can be written, non canonically:
$$
A = M^{(0)}[A^{(0)}] = M^{(\infty)}[A^{(\infty)}], 
$$
where:
$$
A^{(0)}, A^{(\infty)} \in \GL_n(\C) \text{~and~}
M^{(0)} \in \GL_n(\C(\{z\})), M^{(\infty)} \in \GL_n(\C(\{w\})).
$$
For the constant matrix $C \in \GL_n(\C)$, one builds a canonical
fundamental solution of $\sq X = C X$ in the following way. First,
special functions are built from theta functions that satisfy the
following elementary equations: $\sq l_q = l_q + 1$; and, for all
$c \in \C^*$: $\sq e_c = c e_c$. All these functions are meromorphic
over $\C^*$; moreover, we have $e_1 = 1$ and $e_{qc} = z e_c$. Then, 
from the Jordan decomposition $C = C_s C_u$, where 
$C_s = P \Diag(c_1,\ldots,c_n) P^{-1}$, one draws $e_{C_u} := C_u^{l_q}$ 
and $e_{C_s} := P \Diag(e_{c_1},\ldots,e_{c_n}) P^{-1}$. Last,
$e_C := e_{C_s} e_{C_u}$. Thus $\sq X = A X$ admits the following non 
canonical fundamental solutions:
$$
\X^{(0)} := M^{(0)} e_{A^{(0)}} \text{~and~} \X^{(\infty)} := M^{(\infty)} e_{A^{(\infty)}}.
$$
The \emph{Birkhoff connection matrix} is then defined as:
$$
P := (\X^{(\infty)})^{-1} \X^{(0)} \in \GL_n(\M(\Eq)).
$$
In order to give it a functorial and even galoisian meaning, we record 
two basic facts. First \cite[lemma 1.2.4.1, p. 935]{JSGAL}, if $F^{(0)}$
is a meromorphic (at $0$) morphism from $A^{(0)}$ to $B^{(0)}$, then of
course $F^{(0)} e_{A^{(0)}} = e_{B^{(0)}} R^{(0)}$ where $R^{(0)}$ is elliptic.
But more is true: from the special form of our solutions, one can deduce
that $R^{(0)} \in \Mat_{p,n}(\C)$. (Similarly at $\infty$.) This is used
in the context of the following commutative diagram:
$$
\xymatrix{
I_n \ar@<0ex>[r]^{e_{A^{(\infty)}}} \ar@<0ex>[d]^{R^{(\infty)}} &
A^{(\infty)} \ar@<0ex>[r]^{M^{(\infty)}} \ar@<0ex>[d]^{F^{(\infty)}} &
A \ar@<0ex>[d]^{F} &
A^{(0)} \ar@<0ex>[l]_{M^{(0)}} \ar@<0ex>[d]^{F^{(0)}} &
I_n \ar@<0ex>[l]_{e_{A^{(0)}}} \ar@<0ex>[d]^{R^{(0)}} \\
I_p \ar@<0ex>[r]^{e_{B^{(\infty)}}} &
B^{(\infty)} \ar@<0ex>[r]^{N^{(\infty)}} &
B &
B^{(0)} \ar@<0ex>[l]_{N^{(0)}} &
I_p \ar@<0ex>[l]_{e_{B^{(0)}}} 
}
$$
One can start from $F$ and complete it outwards, or start from $R^{(0)}$
and $R^{(\infty)}$ and complete it inwards. \\

As for the tensor properties, the basic fact is that it is impossible to
choose the family of functions $e_c$ so that $e_c e_d = e_{cd}$. Thus we
are led to introduce the cocycle of elliptic functions 
$\phi(c,d) := \dfrac{e_c e_d}{e_{cd}}$ and to extend it to matrices (through
their eigenvalues) so as to have the formula:
$$
e_{C_1} \otimes e_{C_2} = e_{C_1 \otimes C_2} \Phi(C_1,C_2).
$$
Note that for unipotent matrices there is no twisting since $e_1 = 1$.

\paragraph{The tensor category $\mathcal{C}_f$ of connection data}
\label{paragraph:descriptionCf}

Its objects are triples $(A^{(0)},P,A^{(\infty)})$, where 
$A^{(0)}, A^{(\infty)} \in \GL_n(\C)$ and $P \in \GL_n(\M(\Eq))$.
Morphisms from $(A^{(0)},P,A^{(\infty)})$ to $(B^{(0)},Q,B^{(\infty)})$
are pairs $(R^{(0)},R^{(\infty)}) \in \Mat_{p,n}(\C)^2$ such that:
\begin{align*}
R^{(\infty)} P & = Q R^{(0)}, \\
F^{(0)} := e_{B^{(0)}} R^{(0)} (e_{A^{(0)}})^{-1}
& \text{~is meromorphic at~} 0 \\
F^{(\infty)} := e_{B^{(\infty)}} R^{(\infty)} (e_{A^{(\infty)}})^{-1}
& \text{~is meromorphic at~} \infty.
\end{align*}
In \emph{loc. cit.} an explicit condition is given ensuring these
meromorphies, but we shall not need it. (It is used to guarantee
that the following constructions do work.) \\

Now the tensor product has to be twisted in order to get the theorem
we need. For morphisms, and for the left and right components of
objects, we use the usual tensor product. For the middle component,
we shall use the twisted tensor product, defined as follows:
$$
(A_1^{(0)},P_1,A_1^{(\infty)}) \otimes (A_2^{(0)},P_2,A_2^{(\infty)}) :=
(A_1^{(0)} \otimes A_2^{(0)},
P_1 \underline{\otimes} P_2,
A_1^{(\infty)} \otimes A_2^{(\infty)}),
$$
where:
$$
P_1 \underline{\otimes} P_2 := 
\Phi(A_1^{(\infty)},A_2^{(\infty)}) 
(P_1 \otimes P_2) 
\bigl(\Phi(A_1^{(0)},A_2^{(0)})\bigr)^{-1}.
$$

\begin{theo}
The tensor categories $\EE_f$ and $\mathcal{C}_f$ are equivalent.
\end{theo}
\begin{proof}
Because of the non canonical choice, one does not define a functor from
one of these categories to the other. Instead, one defines yet another
category $\mathcal{S}_f$ with objects $(A^{(0)},M^{(0)},A^{(\infty)},M^{(\infty)})$
and with morphisms $(R^{(0)},R^{(\infty)})$, all being subject to adequate 
conditions. The tensor structure on $\mathcal{S}_f$ is the natural one.
Then functors from $\mathcal{S}_f$ to $\EE_f$ and $\mathcal{C}_f$ are easily 
defined. Note that the proof of the essential surjectivity of the second 
functor is essentially due to Birkhoff (it rests on his theorem of 
factorisation of analytic matrices).
\end{proof}

\paragraph{The Galois group of $\EE_f$ and $\mathcal{C}_f$}

From the description of $\mathcal{C}_f$, it is clear how to define
fiber functors $\omega_f^{(0)}$ and $\omega_f^{(\infty)}$ on it. These
extend to the local categories obtained by keeping only the $0$ or
$\infty$ component, and by allowing meromorphic morphisms. One thus
obtains the local Galois groups $G_f^{(0)}$ and $G_f^{(\infty)}$ that
were described in section \ref{subsection:overallstructure}. We
want to use $P$ to connect them. More precisely, we should like
each value $P(a) \in \GL_n(\C)$ to behave like a ``connection formula''
in Riemann-Hilbert correspondance, and so be a galoisian isomorphism 
from $\omega^{(0)}(A)$ to $\omega^{(\infty)}(A)$. This does not work
because the formation of the Birkhoff matrix is not $\otimes$-compatible:
that is, $(P_1 \underline{\otimes} P_2)(a) \neq P_1(a) \otimes P_2(a)$. 
We shall therefore twist $P$ in order to obtain tensor-compatibility
and also functoriality. This is done as follows. \\

One can define explicitly a family of (abstract) group morphisms $g_a$ 
from $\C^*$ to itself such that\footnote{In \cite[3.2.2.2]{JSGAL}, the
stated condition is $g_a(q) =1$, but it is a typographical error.}
$g_a(q) = a$ for all $a \in \C^*$.
Then we set $\psi_a(c) := \dfrac{e_c(a)}{g_a(c)}$ and we extend each
function $\psi_a$ to a function $\Psi_a$ on matrices, through their
eigenvalues. Last, we define:
$$
\check{P}(a) := \bigl(\Psi_a(A^{(\infty)}\bigr)^{-1} P(a) \Psi_a(A^{(0)}),
$$
and can prove that, for each $a \in \C^*$, one has an isomorphism
of fiber functors $(A^{(0)},P,A^{(\infty)}) \leadsto \check{P}(a)$
from $\omega_f^{(0)}$ to $\omega_f^{(\infty)}$. Since $\check{P}(a)$ is 
not defined for all $a$, this actually applies to a smaller category 
than $\EE_f$, but any given object belongs to ``most'' of these
subcategories.

\begin{theo}
The group generated by $G_f^{(0)}$, one particular conjugate 
$\bigl(\check{P}(a)\bigr)^{-1} G_f^{(\infty)} \check{P}(a)$ and the set 
of all defined values $\bigl(\check{P}(b)\bigr)^{-1} \check{P}(a)$ is 
Zariski-dense in the global Galois group of $A$.
\end{theo}

The proof uses Chevalley criterion: any line in any tensor construction
that is fixed by the smaller group is fixed by the bigger one. It rests
on the following useful fact: if $x$ is an eigenvector for $G_f^{(0)}$,
then it is an eigenvector for $\Psi_a(A^{(0)})$. We shall sketch the proof
in our case of interest in the next subsection.

\subsection{The global Galois group with integral slopes}

We now extend the results above to the case of irregular equations with
integral slopes. As the extension involves no new idea, our presentation
will be concise. The category $\EE_1$ of interest has as objects systems
with matrix $A \in \GL_n(\C(z))$ such that their slopes at $0$ and at
$\infty$ are integral; and as morphisms $A \rightarrow B$ matrices
$F \in \Mat_{p,n}(\C(z))$ such that $(\sq F) A = B F$. The tensor product
is the natural one and makes it a neutral tannakian category. Each object 
$A$ of $\EE_1$ can be written, non canonically:
$$
A = M^{(0)}[A^{(0)}] = M^{(\infty)}[A^{(\infty)}], 
$$
where $M^{(0)} \in \GL_n(\C(\{z\}))$, $M^{(\infty)} \in \GL_n(\C(\{w\}))$
and $A^{(0)}$, $A^{(\infty)}$ are in Birkhoff-Guenther normal form. \\

To define solutions, we choose once and for all a function $\theta$ such
that $\sq \theta = z \theta$ and an arbitrary direction of summation in
$\Eq$. Because of this, the following constructions are only valid on a
subcategory of $\EE_1$, but each particular object of $\EE_1$ belongs to
``most'' of these subcategories. We shall call $A_p^{(0)}$, $A_p^{(\infty)}$
the pure systems associated to $A^{(0)}$, $A^{(\infty)}$ by the $\gr$ functor
(hence there block-diagonal parts). Let $S^{(0)}$ the meromorphic isomorphism 
from $A_p^{(0)}$ to $A^{(0)}$ obtained by summation along the selected 
direction mentioned above; and similarly at infinity. Then, calling
$\mu_1,\ldots,\mu_k$ the slopes of $A^{(0)}$ and $r_1,\ldots,r_k$ their
multiplicities, let 
$\Gamma^{(0)} := \Diag(\theta^{\mu_1} I_{r_1},\ldots,\theta^{\mu_k} I_{r_k})$.
We have $A_p^{(0)} = \Gamma^{(0)}[A_f^{(0)}]$ with $A_f^{(0)} \in \GL_n(\C)$.
In the end, using the similar notations at infinity, we put:
$$
e_{A^{(0)}} := S^{(0)} \Gamma^{(0)} e_{A_f^{(0)}} \text{~and~}
e_{A^{(\infty)}} := S^{(\infty)} \Gamma^{(\infty)} e_{A_f^{(\infty)}}.
$$
Thus $\sq X = A X$ admits the following non canonical fundamental solutions:
$$
\X^{(0)} := M^{(0)} e_{A^{(0)}} \text{~and~} \X^{(\infty)} := M^{(\infty)} e_{A^{(\infty)}}.
$$
The \emph{Birkhoff connection matrix} is then defined as:
$$
P := (\X^{(\infty)})^{-1} \X^{(0)} \in \GL_n(\M(\Eq)).
$$
Its tensor behaviour is exactly similar to that observed in the fuchsian case
and we shall set, in appropriate context:
$$
(A_1^{(0)},P_1,A_1^{(\infty)}) \otimes (A_2^{(0)},P_2,A_2^{(\infty)}) :=
(A_1^{(0)} \otimes A_2^{(0)},
P_1 \underline{\otimes} P_2,
A_1^{(\infty)} \otimes A_2^{(\infty)}),
$$
where:
$$
P_1 \underline{\otimes} P_2 := 
\Phi((A_1)_f^{(\infty)},(A_2)_f^{(\infty)}) 
(P_1 \otimes P_2) 
\bigl(\Phi((A_1)_f^{(0)},(A_2)_f^{(0)})\bigr)^{-1}.
$$
The functorial behaviour requires some more comments. Let $B$ an object
of rank $p$ in $\EE_1$ and $B^{(0)}$, $N^{(0)}$, $B_p^{(0)}$, $T^{(0)}$, 
$B_f^{(0)}$, $\Delta^{(0)}$, $\Y^{(0)}$, $\Y^{(\infty)}$, and $Q$ the associated 
data corresponding respectively to $A^{(0)}$, $M^{(0)}$, $A_p^{(0)}$, $S^{(0)}$, 
$A_f^{(0)}$, $\Gamma^{(0)}$, $\X^{(0)}$, $\X^{(\infty)}$ and $P$. Let $F$ be a 
morphism from $A$ to $B$. Then we have a commutative diagram:
$$
\xymatrix{
I_n \ar@<0ex>[r]^{e_{A_f^{(\infty)}}} \ar@<0ex>[d]^{R^{(\infty)}} &
A_f^{(\infty)} \ar@<0ex>[r]^{\Gamma^{(\infty)}} \ar@<0ex>[d]^{F_f^{(\infty)}} &
A_p^{(\infty)} \ar@<0ex>[r]^{S^{(\infty)}} \ar@<0ex>[d]^{F_p^{(\infty)}} &
A^{(\infty)} \ar@<0ex>[r]^{M^{(\infty)}} \ar@<0ex>[d]^{F^{(\infty)}} &
A \ar@<0ex>[d]^{F} &
A^{(0)} \ar@<0ex>[l]_{M^{(0)}} \ar@<0ex>[d]^{F^{(0)}} &
A_p^{(0)} \ar@<0ex>[l]_{S^{(0)}} \ar@<0ex>[d]^{F_p^{(0)}} &
A_f^{(0)} \ar@<0ex>[l]_{\Gamma^{(0)}} \ar@<0ex>[d]^{F_f^{(0)}} &
I_n \ar@<0ex>[l]_{e_{A_f^{(0)}}} \ar@<0ex>[d]^{R^{(0)}} \\
I_p \ar@<0ex>[r]_{e_{B_f^{(\infty)}}} &
B_f^{(\infty)} \ar@<0ex>[r]_{\Delta^{(\infty)}} &
B_p^{(\infty)} \ar@<0ex>[r]_{T^{(\infty)}} &
B^{(\infty)} \ar@<0ex>[r]_{N^{(\infty)}} &
B &
B^{(0)} \ar@<0ex>[l]^{N^{(0)}} &
B_p^{(0)} \ar@<0ex>[l]^{T^{(0)}} &
B_f^{(0)} \ar@<0ex>[l]^{\Delta^{(0)}} &
I_p \ar@<0ex>[l]^{e_{B_f^{(0)}}} 
}
$$
Of course, all vertical arrows can be defined from $F$. For instance,
$F^{(0)} := (N^{(0)})^{-1} \circ F \circ M^{(0)} \in \Mat_{p,n}(\C(\{z\}))$
is a morphism in $\EE_1^{(0)}$, and similarly at $\infty$. Then one can
see that $F_p^{(0)} := (T^{(0)})^{-1} \circ F^{(0)} \circ S^{(0)}$ is actually
$\gr F^{(0)}$ (and similarly at $\infty$); and, from the block-diagonal
structures of the involved matrices, one can see that
$F_f^{(0)} := (\Delta^{(0)})^{-1} \circ F_p^{(0)} \circ \Gamma^{(0)}$ is
actually equal to $F_p^{(0)} = \gr F^{(0)}$, the block-diagonal of $F^{(0)}$.
Then, from the lemma already quoted \cite[lemma 1.2.4.1, p. 935]{JSGAL},
we see that 
$R^{(0)} := (e_{B_f^{(0)}})^{-1} \circ F_f^{(0)} \circ e_{A_f^{(0)}} \in \Mat_{p,n}(\C)$
and similarly at $\infty$. \\

Conversely, if we are given the two lines and the most external vertical
arrows $R^{(0)}$, $R^{(\infty)}$, the condition to be able to go inwards
and fill in the other vertical arrows to get a commutative diagram is
that $Q R^{(0)} = R^{(\infty)} P$. The condition to get a rational $F$ is
that $F^{(0)} \in \Mat_{p,n}(\C(\{z\}))$ and similarly at $\infty$. Indeed,
from the functional equation $\sq F^{(0)} = B^{(0)} F^{(0)} (A^{(0)})^{-1}$
and the fact that $A^{(0)},B^{(0)}$ are in Birkhoff-Guenther normal form,
one deduces that $F^{(0)}$ is meromorphic on $\C$, and similarly at $\infty$,
so that $F$ is actually meromorphic on the Riemann sphere, thus rational.

\paragraph{The tensor category $\mathcal{C}_1$ of connection data}

Its objects are triples $(A^{(0)},P,A^{(\infty)})$, where 
$A^{(0)}, A^{(\infty)} \in \GL_n(\C(\{z\}))$ are in Birkhoff-Guenther
normal form and $P \in \GL_n(\M(\Eq))$.
Morphisms from $(A^{(0)},P,A^{(\infty)})$ to $(B^{(0)},Q,B^{(\infty)})$
are pairs $(R^{(0)},R^{(\infty)}) \in \Mat_{p,n}(\C)^2$ such that:
\begin{align*}
R^{(\infty)} P & = Q R^{(0)}, \\
F^{(0)} := e_{B^{(0)}} R^{(0)} (e_{A^{(0)}})^{-1}
& \text{~is meromorphic at~} 0 \\
F^{(\infty)} := e_{B^{(\infty)}} R^{(\infty)} (e_{A^{(\infty)}})^{-1}
& \text{~is meromorphic at~} \infty.
\end{align*}

\begin{rema}
We saw in \ref{paragraph:descriptionCf} that there was an explicit condition 
(although we did not state it) on $R^{(0)}$ for $F_p^{(0)} = F_f^{(0)}$ to be 
meromorphic at $0$. Here, we must add a new condition to ensure that $F^{(0)}$ 
is also meromorphic at $0$. This condition is obviously related to the 
summations $S^{(0)}$ and $T^{(0)}$. We have not so far an explicit criterion, 
but it could be related to the way $F_p^{(0)}$ links the classifying cohomology 
class in $H^1(\Eq,\F_{A_p^{(0)}})$ corresponding to $A^{(0)}$ to the classifying 
cohomology class in $H^1(\Eq,\F_{B_p^{(0)}})$ corresponding to $B^{(0)}$.
\end{rema}

The tensor structure is defined as follows. For morphisms, and for the 
left and right components of objects, we use the usual tensor product. 
For the middle component, we shall use the twisted tensor product, defined 
as follows:
$$
(A_1^{(0)},P_1,A_1^{(\infty)}) \otimes (A_2^{(0)},P_2,A_2^{(\infty)}) :=
(A_1^{(0)} \otimes A_2^{(0)},
P_1 \underline{\otimes} P_2,
A_1^{(\infty)} \otimes A_2^{(\infty)}),
$$
where:
$$
P_1 \underline{\otimes} P_2 := 
\Phi(A_1^{(\infty)},A_2^{(\infty)}) 
(P_1 \otimes P_2) 
\bigl(\Phi(A_1^{(0)},A_2^{(0)})\bigr)^{-1}.
$$
Recall that we have extended the definition of $\Phi$ to this setting.

\begin{theo}
The tensor categories $\EE_1$ and $\mathcal{C}_1$ are equivalent.
\end{theo}
\begin{proof}
The method and the proof are the same as in the fuchsian case: we use
an enriched category $\mathcal{S}_1$ with objects 
$(A^{(0)},M^{(0)},A^{(\infty)},M^{(\infty)})$ and with morphisms 
$(R^{(0)},R^{(\infty)})$, all being subject to obvious conditions. 
The tensor structure on $\mathcal{S}_1$ is the natural one.
Then functors from $\mathcal{S}_1$ to $\EE_1$ and $\mathcal{C}_1$ are 
defined and proved to be $\otimes$-equivalences exactly as in the
fuchsian case.
\end{proof}

\paragraph{The Galois group of $\EE_1$ and $\mathcal{C}_1$}

From the description of $\mathcal{C}_1$, it is clear how to define
fiber functors $\omega_1^{(0)}$ and $\omega_1^{(\infty)}$ on it and
that their extension to the local categories $\EE_1^{(0)}$ and 
$\EE_1^{(\infty)}$ yields local Galois groups which are precisely
the Galois group $G_1^{(0)}$ studied in this paper and its counterpart 
$G_1^{(\infty)}$ at $\infty$. \\

Also the formula:
$$
\check{P}(a) := \bigl(\Psi_a(A^{(\infty)}\bigr)^{-1} P(a) \Psi_a(A^{(0)}),
$$
extends here with the only adaptation that $\Psi_a(A^{(0)})$ means
$\Psi_a(A_f^{(0)})$, and similarly at $\infty$. Again, one finds that,
for each $a \in \C^*$, one has an isomorphism of fiber functors 
$(A^{(0)},P,A^{(\infty)}) \leadsto \check{P}(a)$ from $\omega_1^{(0)}$ 
to $\omega_1^{(\infty)}$ (again, on appropriate subcategories).

\begin{theo}
\label{theo:structureglobale}
The group generated by $G_1^{(0)}$, one particular conjugate 
$\bigl(\check{P}(a)\bigr)^{-1} G_1^{(\infty)} \check{P}(a)$ and the set 
of all defined values $\bigl(\check{P}(b)\bigr)^{-1} \check{P}(a)$ is 
Zariski-dense in the global Galois group of $A$.
\end{theo}
\begin{proof}
The proof uses again Chevalley criterion in a similar way to \emph{loc. cit.}. 
Suppose we have two lines $D^{(0)}$ and $D^{(\infty)}$ that are respectively
fixed by $G_1^{(0)}$ and $G_1^{(\infty)}$ and such that each $\check{P}(a)$
sends $D^{(0)}$ to $D^{(\infty)}$. Taking generators $x^{(0)},x^{(\infty)}$, we
se by tannakian duality that they define rank one subobjects
$x^{(0)}: a^{(0)} \rightarrow A^{(0)}$ and 
$x^{(\infty)}: a^{(\infty)} \rightarrow A^{(\infty)}$.
By the lemma quoted at the end of the previous subsection, the fact
that $x^{(0)}$, $x^{(\infty)}$ are respectively eigenvectors of $G_1^{(0)}$,
$G_1^{(\infty)}$ implies that the value $P(a)$ of the non twisted connection
matrix sends $D^{(0)}$ to $D^{(\infty)}$, so that
$P(a) x^{(0)} = p(a) x^{(\infty)}$ for some $p(a) \in \C$. But then $p$ is
a non trivial elliptic function, $(a^{(0)},p,a^{(\infty)})$ is a rank one
object of $\mathcal{C}_1$ and $(x^{(0)},x^{(\infty)})$ an embedding of this
object as a subobject of $(A^{(0)},P,A^{(\infty)})$. Then, by functoriality,
all elements of the global Galois group must fix this subobject, whence
the two lines.
\end{proof}

\begin{coro}
Topological generators of the Stokes Lie algebra at $0$ and $\infty$
together with topological generators of the local pure Galois groups
and the values of $\check{P}(a)$ are together topological generators
of the global Galois group. 
\end{coro}



\section{The inverse problem}


\subsection{Known results}
\label{subsection:knownresults}

To our knowledge there existed before almost no result on the \emph{local} 
inverse problem that we shall solve below (for the integral slope case). 
We will review the known results on the \emph{global} inverse problem.

As far as we know, the first significant result on the global inverse 
problem of the $q$-difference Galois theory is due to P. Etingof \cite{Et} 
(Proposition 3.4, page 7).

\begin{prop}
\label{Etingof} 
Let $G$ be any connected complex linear algebraic group, there exists 
$\rho>0$ (depending on $G$) such that, for all $0<\vert q\vert<\delta$, 
there exists a rational  \emph{regular} difference system $\sigma_q Y=AY$ 
whose $q$-difference Galois group is $G$.
\end{prop}

We recall that the system $\sigma_q Y=AY$ is regular if $A(0)=A(\infty)=I$.

The proof uses the following result (\cf\ \cite{TT}).

\begin{lemm}
Let $G$ be any complex linear algebraic goup, then there exists 
$g_1,\ldots ,g_m\in G$ such that the subgroup generated by 
$g_1,\ldots ,g_m$ is Zariski dense in $G$.
\end{lemm}

We recall that the Tretkoffs used this lemma (and the Riemann-Hilbert 
correspondance) to solve the inverse problem of the Galois differential 
theory with regular singular systems.

If $G$ is \emph{abelian}, it is possible to improve the proposition 
\ref{Etingof}. 

\begin{prop}
\label{classes} 
Let $G$ be any \emph{abelian} connected complex linear algebraic group, 
then, \emph{for all} $q\in\C^*$, $\vert q\vert\neq 1$, there exists a 
rational \emph{regular} difference system $\sigma_q Y=AY$ whose 
$q$-difference Galois group is $G$.
\end{prop}

This proposition follows from \cite{JSGAL}, using the following lemma.

\begin{lemm}
Let $G$ be any \emph{abelian}  connected complex linear algebraic goup, 
then there exists a rational \emph{dominant} map $f:\Eq\rightarrow G$. 
\end{lemm}

Here $\Eq$ is seen as a projective algebraic curve (an elliptic curve).
The proof of this lemma follows from the existence of an isomorphism 
$G \approx (G_m)^k\times (G_u)^l$. \\

From propositions \ref{Etingof} and \ref{classes} we can conjecture that 
for every connected complex linear algebraic group $G$ and for all  
$q \in \C^*$, $\vert q\vert \neq 1$, there exists a rational \emph{regular} 
difference system $\sigma_q Y=AY$ whose $q$-difference Galois group is $G$.

Such a system will have in general ``a lot of singularities". Below we will 
attack the global inverse problem in the opposite direction, searching a 
system with a \emph{minimal} number of singularities in the spirit of a 
$q$-analog of the Abhyankar conjecture.

Another source of solutions of the \emph{inverse} problem are of course 
the known solutions of the \emph{direct} problem, in particular from the 
computation of the $q$-difference Galois groups of the \emph{generalized 
$q$-hypergeometric equations} (regular-singular or not). One can find a 
complete solution of this last problem in a series of papers of J. Roques 
\cite{Roques1, Roques2, Roques3}. Limiting ourselves to the cases of 
\emph{simple} groups, the complete list obtained by J. Roques is: 
$\text{SL}(n,\C)$, $\text{SO}(n,\C)$, $\text{Sp}(2n,\C)$. 

It is interesting to compare with the differential case (\cf\ \cite{BeuBrHe}, 
\cite{DuvM}, \cite{K}, \cite{Mit}). The simple groups which are differential 
groups of \emph{generalized hypergeometric differential equations} 
(regular-singular or not) are: $\text{SL}(n,\C)$, $\text{SO}(n,\C)$, 
$\text{Sp}(2n,\C)$ and \dots the group $G_2$ ! Therefore the \emph{only} 
difference between the $q$-difference case and the differential case is 
the exceptional group $G_2$.


\subsection{Linear algebraic groups: reminders and complements}


\subsubsection{Notations and definitions. Levi decomposition}

In the following \emph{all} the algebraic groups are \emph{complex linear} 
algebraic groups. In general $G$ is a linear algebraic group, $\mathfrak{g}$ 
is its Lie algebra, $T \subset G$ is a torus, $\mathfrak{t}$ the Lie algebra 
of $T$, and $D \subset G$ is an abelian semi-simple group.

An algebraic group G contains a unique \emph{maximal normal solvable 
subgroup}, this subgroup is closed. Its identity component is called 
the \emph{radical} $R(G)$ of $G$.

We will denote $R_u(G)$ the \emph{unipotent radical} of $G$ (i.e. the set 
of unipotent elements of $R(G)$). A group $G$ is \emph{reductive} if and 
only if $R_u(G)=\{e\}$.

\begin{defi}
A \emph{Levi subgroup} of a linear algebraic group $G$ is a maximal 
reductive subgroup.
\end{defi}

We have an exact sequence:
\[
\{e\} \rightarrow R_u(G) \rightarrow G \rightarrow G/R_u(G) \rightarrow \{e\}
\]
and, if $H \subset G$ is a Levi subgroup, then the quotient map 
$G \rightarrow G/R_u(G)$ induces an isomorphism $H \rightarrow G/R_u(G)$. 
More precisely we have the following result (essentially due to Mostow).

\begin{prop}
Let $G$ be a linear algebraic group.
\begin{itemize}
\item[(i)]
If $H \subset G$ is a Levi subgroup, then $G$ is a semi-direct product:
$G = R_u(G) \rtimes H$.
\item[(ii)]
Any two Levi subgroups of $G$ are conjugate under an inner automorphism.
\item[(iii)]
If $H \subset G$ is a subgroup and if the quotient map 
$G \rightarrow G/R_u(G)$ induces an isomorphism 
$H \rightarrow G/R_u(G)$, then $H$ is a Levi subgroup.
\end{itemize}
\end{prop}

\begin{proof}
For (i) and (ii) \cf\ \cite{Mostow} (a subgroup is fully reducible if and 
only if it is reductive). 

Let $H \subset G$ be a subgroup such that the quotient map 
$G \rightarrow G/R_u(G)$ induces an isomorphism $H \rightarrow G/R_u(G)$, 
$H$ is reductive, therefore it is contained in a maximal reductive subgroup 
$H'$ and $H=H'$.
\end{proof}

A \emph{Levi decomposition} of a linear algebraic group $G$ is an 
isomorphism $G \approx U \rtimes S$, where $S$ is \emph{reductive} and 
$U$ is \emph{unipotent}.

 
\subsubsection{Diagonalisable and triangularisable groups}

We shall recall the notions of \emph{diagonalisable} and of  
\emph{triangularizable} algebraic group. The properties of the 
diagonalisable groups and of the triangularisable \emph{connected} groups 
are well known, but for the triangularisable \emph{non connected} groups 
we do not know good references hence, for sake of completeness, we shall 
detail the necessary results.
 
An algebraic group $G$ is \emph{diagonalizable} if and only if it is
abelian and semi-simple ($G=G_s$). If $G$ is diagonalizable, then every 
representation of $G$ is diagonalizable in the matrix sense. An algebraic 
group $G$ is  diagonalisable if and only if there exists a faithful 
representation of $G$ which is diagonalisable in the matrix sense.
 
We will say that a linear algebraic group $G$ is \emph{triangularizable}
if there exists a faithful triangular representation. A triangularizable 
group is \emph{solvable}.

A solvable \emph{connected} linear algebraic group is triangularizable
(Lie-Kolchin theorem). In particular a unipotent group is triangularizable.
 
\begin{prop}
\label{triangprop}
A linear algebraic group is triangularizable if and only if 
$G \approx U \rtimes D$, where $U$ is \emph{unipotent} and $D$ is 
abelian and semi-simple. Then $U=R_u(G)$ and $D\approx G/R_u(G)$. 

The Levi subgroups of a triangularizable algebraic group $G$ are the 
maximal abelian semi-simple subgroups. If $G$ is connected, the Levi 
subgroups are the maximal tori.
\end{prop}

\begin{proof}
\begin{enumerate}
\item 
If $G \approx U \rtimes D$, where $U$ is \emph{unipotent} and $D$ is abelian 
and semi-simple, $G$ is triangularizable by \cite[I.7, lemma, p 20]{Kolchin}.
\item
We suppose that $G$ is triangularizable, there exists a faithful representation
$\rho: G \rightarrow \GL_n(\C)$, such that $\rho(G)$ is an upper triangular 
subgroup of $\GL_n(\C)$, a subgroup of the upper triangular subgroup $T_n$. 
We denote $U_n$ (resp. $D_n$) the unipotent upper-triangular subgroup (resp. 
the diagonal subgroup) of $\GL_n(\C)$, then $T_n = U_n \rtimes D_n$, 
$T_n/U_n = D_n$.

There exists a Levi decomposition $G = U \rtimes D$, where $U$ is 
\emph{unipotent} and $D$ reductive. Then $\rho(U)$ is unipotent, therefore 
it is a subgroup of $U_n$ and $\rho$ induces an injective morphism
$D = G/U \rightarrow T_n/U_n=D_n$. The group $D_n$ is abelian semi-simple, 
and $D$ is isomorphic to a subgroup, therefore $D$ is abelian semi-simple.
\item
The Levi subgroups are abelian semi-simple and any abelian semi-simple 
subgroup is reductive, the result follows.
\end{enumerate}
\end{proof}


\subsection{$\Theta$-structures on linear algebraic groups}
\label{subection:Theta-structures}


\subsubsection{Weights and coweights}

\begin{defi}
Let $G$ be an abelian semi-simple group. The \emph{weight group} $G^{\bullet}$ 
of $G$ is the group of homomorphisms of algebraic groups $G\rightarrow \C^*$.
\end{defi}

A weight on $G$ is also called a \emph{character} on $G$.
 
The group $G^{\bullet}$  is an abelian finitely generated group.
 
The weight functor $G \leadsto G^{\bullet}$ is an antiequivalence of categories
between abelian semi-simple algebraic groups and finitely generated abelian 
groups. The quasi inverse of the weight functor is 
$\Hom_{gr}(.,\C^*)$.

Let $f: \C^* \rightarrow \C^*$ be a homomorphism of algebraic groups,  
then $f: z \mapsto z^n$, $n \in \Z$ is the \emph{degree} of $f$ and we 
denote $\deg f=n$.

\begin{defi}
Let $D$ be an abelian semi-simple group. A coweight on $D$ is a homomorphism 
of algebraic groups $\C^* \rightarrow D$. 
\end{defi}

A coweight on $D$ is also called a \emph{one parameter subgroup} of $D$.

If $\chi: \C^* \rightarrow D$ is a coweight on $D$, its image is contained 
in the maximal torus $T \subset D$, therefore it is also a coweight on $T$.

For every weight $\xi$ and every coweight $ \chi$ on an algebraic abelian 
semi-simple group $D$, we set 
\[
<\xi,\chi > :=\deg (\xi \circ \chi).
\]

\begin{defi}
Let $T$ be a complex algebraic torus. The \emph{weight lattice} $T^{\bullet}$ 
of $T$ is the group of weights $T \rightarrow \C^*$, and the \emph{coweight 
lattice} $T_{\bullet}$ of $T$ is the group of coweights $\C^* \rightarrow T$. 
\end{defi}
 
The groups $T^{\bullet}$ and $T_{\bullet}$ are both \emph{free abelian} groups 
whose rank is the dimension of $T$. The map 
$(\xi,\chi) \mapsto <\xi,\chi > := \deg (\xi\circ\chi)$ is a canonical 
\emph{non degenerate pairing} $T^{\bullet}\times T_{\bullet}\rightarrow \Z$ 
(\cf\ for example \cite[16.1]{H}). \\

The weight functor $T \leadsto T^{\bullet}$, resp. the coweight functor 
$T \leadsto T_{\bullet}$ is an antiequivalence, resp. an equivalence of 
categories between algebraic tori and finitely generated free abelian groups. 
The quasi inverse of the weight functor is $\Hom_{gr}(.\C^*)$.

An isomorphism of algebraic torus $\Phi: (\C^*)^{\mu} \rightarrow T$ gives 
a $\Z$-basis of $T_{\bullet}$ and the inverse isomorphism 
$\Phi^{-1}: T \rightarrow (\C^*)^{\mu}$ gives a $\Z$-basis of $T^{\bullet}$.

To a weight $\xi: T \rightarrow \C^*$ we associate its infinitesimal 
counterpart $L \xi: \mathfrak{t} \rightarrow \C$ (remember $\mathfrak{t}$
denotes the Lie algebra of $T$). We will sometimes ``identify" the group 
of weights and the group of infinitesimal weights and we will interpret 
$T^{\bullet}$ as a $\Z$-submodule of the complex dual space $\mathfrak{t}^*$ 
of $\mathfrak{t}$. According to the tradition, we will freely use the 
\emph{additive} notation for the weights. We will denote $\mathcal{W}_{\R}$ 
the real vector space $\R \otimes_{\Z} T^{\bullet}$.

For $\xi \in T^{\bullet}$, $\chi \in T_{\bullet}$, we define: 
\[
<L\xi,L\chi > := L\xi \circ L\chi \,(1) = <\xi,\chi >. 
\]

Let $G$ be a linear algebraic group and let be $D$ an \emph{abelian 
semi-simple group}. We recall that the \emph{roots} of $D$ are the 
\emph{non trivial} weights on $D$ for the adjoint action of $D$ on 
the Lie algebra $\mathfrak{g}$. We denote $\mathfrak{g}_{\xi}$ the 
root space associated to the root $\xi$:
$$
\mathfrak{g}_{\xi} :=
\{x\in\mathfrak{g} \tq \forall \lambda\in D \;,\;
(\Ad \lambda)(x) = \xi(\lambda)\,x \}.
$$

We have 
$\mathfrak{g} =
\mathfrak{g}_0 \oplus \bigoplus_{\xi\in\mathcal{R}} \mathfrak{g}_{\xi}$, 
the sum being on the set of roots $\mathcal{R}$ and 
$\mathfrak{g}_0$ being the space of elements \emph{invariant} by $D$.

If $D = T$ is connected (a torus), then 
$$
\mathfrak{g}_{\xi} =
\{ x\in\mathfrak{g} \tq \forall \tau \in \mathfrak{t} \;,\; 
(\ad \tau) (x)=L\xi(\tau)\, x\}.
$$

Let $\chi$ be a \emph{non trivial} coweight on a torus $T\subset G$, 
it induces a \emph{grading} of Lie algebras on $\mathfrak{g}$:
\[
\forall\, k \in \Z \;,\; 
gr_{\chi}^k := \{x \in \mathfrak{g} \tq 
\forall t \in \C^* \;,\; (\Ad \chi(t))(x) = t^k x\} =
\{x \in \mathfrak{g} \tq [L\chi(1),x] = k x\}.
\]
If $\xi$ is a weight on $T$, then there exists a unique $k\in\Z$ such that 
$\mathfrak{g}_{\xi} \subset gr_{\chi}^k$ and we have $k = <\xi,\chi >$. 
In particular
$\mathfrak{g}_0\subset gr_{\chi}^0$. We have:
\[
gr_{\chi}^k = \bigoplus_{<\xi,\chi>=k} \mathfrak{g}_{\xi}.
\]

 
\subsubsection{$\Theta$-coweights and $\Theta$-structures}

\begin{defi}
Let $D \subset G$ be an abelian semi-simple subgroup and 
$\mathcal{P} := \{\xi_i\}_{i\in I}$ a finite family of non trivial weights 
on $D$. We will say that a coweight $\chi$ on $D$ is \emph{positive} 
(resp. \emph{negative}) on $\mathcal{P}$ if, for every $i\in I$, 
$<\xi_i,\chi >\, >0$ (resp. $<\xi_i,\chi >\, <0$).
\end{defi}

\begin{defi}
Let $G$ be a \emph{triangularizable} linear algebraic group. Let $D$ be 
a \emph{Levi subgroup of $G$}. By proposition \ref{triangprop}, $D$
is abelian semi-simple. A $\Theta$-\emph{coweight} on $D$ is a coweight 
on $D$ which is \emph{negative} on the family of \emph{roots} for the 
adjoint action of $D$ on the Lie algebra $\mathfrak{g}$ of $G$.
\end{defi}

If $\chi$ is a $\Theta$-\emph{coweight} on $D$, then 
$\mathfrak{g}_0=gr_{\chi}^0$ and
$\mathfrak{g}=\bigoplus_{k\in -\N} gr_{\chi}^k$.
\bigskip

\paragraph{$\Theta$-structures}

\begin{defi}
\label{deftheta}
We will say that a linear algebraic group $G$ admits a $\Theta$-structure 
if it is triangularizable and if there exists a $\Theta$-\emph{coweight} 
on a Levi subgroup of $G$.
\end{defi}

Then, by conjugation, there exists $\Theta$-\emph{coweight} on any Levi 
subgroup of $G$.

\begin{defi}
Let $G$ be a linear algebraic group, let $D \subset G$ be an abelian 
semi-simple subgroup and $\chi$ a $\Theta$-coweight on $D$. We will 
say that $\chi$ is \emph{dominant} if, for every root $\zeta$ on $D$, 
we have $\dim \mathfrak{g}_{\zeta} \leq -<\zeta,\chi>$.
\end{defi}

\begin{lemm}
\label{domlemm}
Let $G$ be a \emph{triangularizable} complex linear algebraic group, 
let $D \subset G$ be an abelian semi-simple subgroup. We suppose that 
$\chi$ is a $\Theta$-coweight on $D$. Then there exists a $\Theta$-coweight 
on $D$ which is \emph{dominant}.
\end{lemm}

\begin{proof}
Let $m \in \N^*$ and $\varphi_m:\C^* \rightarrow \C^*$ defined by 
$\varphi_m: t \mapsto t^m$. Then $\chi_m:=\chi \circ \varphi_m$ is a coweight 
on $T$, and for every root $\xi$, we have $<\xi,\chi> < 0$, that is
$<\xi,\chi> \leq -1$, whence $<\xi,\chi_m> = m <\xi,\chi>\, \leq - m$, 
and $\chi_m$ is a $\Theta$-coweight. Then, for a sufficiently big $m$ 
($m \geq \max_{\xi\in\mathcal{R}}\, (\dim \mathfrak{g}_{\xi})$), $\chi_m$ 
is dominant.
\end{proof}


\subsubsection{Existence of a $\Theta$-structure}
 
\begin{rema}
\label{TL:contrex}
\begin{enumerate}
\item
If $G = U \rtimes D$ is a Levi decomposition \emph{such that the semidirect
product is not direct} and if $D$ is an abelian \emph{finite} group, then 
there exists no $\theta$-structure on $G$.
\item 
We suppose that there exists a $\Theta$-structure on a linear algebraic 
group $G$. If $\xi$ is a root, then $\xi^{-1}$ is \emph{not} a root.
\item
\label{TL:contrex0}
There exists a triangularizable connected linear algebraic group $G$ 
such that there exists no $\Theta$-structure on $G$. Let :
\[
G:=
\Bigl\{ \begin{pmatrix}
1 & \alpha & \beta \\
0 & t & \gamma \\
0 & 0 & 1
\end{pmatrix}~\Big\vert ~ t\in\C^*, \alpha, \beta, \gamma \in \C\Bigr\},
\]
it is triangular and it admits the infinitesimal roots $1$ and $-1$, 
therefore there exists no $\Theta$-structure on $G$. 
\end{enumerate}
\end{rema}

A triangularizable group being given it seems difficult to find a practical 
criterium to decide if it admits a $\Theta$-structure. We shall give now a 
\emph{sufficient condition} (we will use it below for the case of Borel 
subgroups of reductive groups).

\begin{defi}
Let $G$ a linear algebraic group and $T \subset G$ a \emph{torus}
A \emph{good system} of roots for the adjoint action of $T$ on 
$\mathfrak{g}$ is a set $\Sigma$ of roots
such that 
\begin{itemize}
\item[(i)]
$\Sigma$ is a $\R$-\emph{free} subset of $W_{\R}$;
\item[(ii)]
every root $\xi\in\mathcal{R}$ can be written 
$\displaystyle{\xi = \sum_{k\in I} a_i \xi_i}$, with, for every 
$i\in I$, $a_i\in\R_+$ and $\xi_i \in \Sigma$.
 \end{itemize}
\end{defi}

\begin{prop}
\label{TL:propgoodsyst}
Let $G$ a connected triangularisable group and $T\subset G$ a 
\emph{maximal torus}. If there exists a good system of roots for 
the adjoint action of $T$ on $\mathfrak{g}$, then there exists a 
$\Theta$-structure on $G$.
\end{prop}

\begin{proof}
We prove firstly a preliminary lemma (part (ii) of this lemma will 
be used later).

\begin{lemm}
\label{TL:lem}
\begin{itemize}
\item[(i)]
Let $f_1,\ldots ,f_{\mu'}$ \emph{independent} $\R$-linear forms on $\R^{\mu}$, 
there exists $p=(p_1,\ldots ,p_{\mu})\in\Z^{\mu}$ such that $f_i(p)< 0$ for 
all $i=1,\ldots ,\mu'$.
\item[(ii)]
Let $f_1,\ldots ,f_{\mu'}$ non trivial $\R$-linear forms on $\R^{\mu}$, 
there exists $p=(p_1,\ldots ,p_{\mu})\in\Z^{\mu}$ such that $f_i(p)\neq 0$ 
for all $i=1,\ldots ,\mu'$.
\end{itemize}
\end{lemm}

\begin{proof}
\begin{itemize}
\item[(i)]
The set $U := \{y\in\R^{\mu} \tq f_i(y)<0,~i=1,\ldots ,\mu'\}$ is a 
\emph{non void} open subset of $\R^{\mu}$, therefore there exists  
$p' = (p'_1,\ldots ,p'_{\mu}) \in \Q^{\mu} \cap U$. If $y \in U$ and $a \in\N^*$, 
then $a y \in U$, the result follows.
\item[(ii)]
The set $U := \{y\in\R^{\mu} \tq f_i(y) \neq 0,~i=1,\ldots ,\mu'\}$ is a 
\emph{non void} open subset of $\R^{\mu}$, therefore there exists  
$p' = (p'_1,\ldots ,p'_{\mu}) \in \Q^{\mu} \cap U$. If $y \in U$ and $a \in\N^*$, 
then $a y \in U$, the result follows.
\end{itemize}
\end{proof}

We can now prove the proposition.

Let $\Phi:(\C^*)^{\mu} \rightarrow T$ be an isomorphism of tori.

Let $\Sigma = \{\xi_1,\ldots ,\xi_{\mu'}\}$ be a \emph{good system} of 
roots of $G$. For $i = 1,\ldots ,\mu'$, we set $f_i: = L\xi_i\circ L\Phi$.
We interpret $f_1,\ldots ,f_{\mu'}$ as linear forms on $\R^{\mu}$, 
by hypothesis they are independent, therefore we can apply the lemma 
\ref{TL:lem} above. There exists $p = (p_1,\ldots ,p_{\mu})\in\Z^{\mu}$ 
such that $f_i(p)< 0$ for all $i = 1,\ldots ,\mu'$. We define a morphism
$\chi: \C^* \rightarrow T$ by 
$\Phi^{-1}\circ\chi:t \rightarrow (t_1:=t^{p_1},\ldots ,t_{\mu}:=t^{p_{\mu}})$,
for $i=1,\ldots ,\mu'$, then we set 
 $v_i := f_i \circ L(\Phi^{-1} \circ \chi ):= L\xi_i \circ L\chi$. 
We have  $v_i(1)=f_i(p)<0$. If 
$\xi$ is a root, then 
$\displaystyle <L\xi ,L\chi>\, = \sum_{i=1,\ldots,\mu'} a_iv_i(1)$ with 
$a_i \geq 0$, $a_1 + \cdots + a_{\mu'} > 0$ and therefore 
$<L\xi ,L\chi>\ =\, <\xi,\chi>\, <0$.
\end{proof}

In the following proposition, $T$ is a maximal torus of $G$. One implication
is proposition \ref{TL:propgoodsyst}.

\begin{prop}
If the dimension of $T$ is one, then $G$ admits admits a $\Theta$-structure 
if and only if there exists a good system of roots. 
\end{prop}

\begin{proof}
If $G$ admits admits a $\Theta$-structure, then there exists a surjective 
morphism $\eta:\C^*\rightarrow T$ such that, for every root $\xi$, 
$<\xi,\eta>$ is negative. Let $\xi_1$ be a root, then, for every root
$\xi$, we have $L\xi =a L\xi_1$ with $a > 0$, therefore $\{\xi_1\}$ is a good 
system of roots.
\end{proof}

For basic definitions on Borel subgroups, positive systems of roots...
\cf \cite{H}.

\begin{prop}
\label{TL:borelth}
If $G^+$ is a \emph{Borel subgroup} of a \emph{connected reductive} 
algebraic group, then there exists a $\Theta$-structure on $G^+$.
\end{prop}
  
\begin{proof}
Let $G^+$ be a Borel subgroup of the connected reductive algebraic group $G$. 
Let $T$ be a maximal torus of $G$ contained in $G^+$, then $G^+$ corresponds 
to a \emph{positive} system of roots $\mathcal{R}^+$ of $G'$ 
($\displaystyle\mathfrak{g} =
\mathfrak{t} + \bigoplus_{\xi\in\mathcal{R}^+} \mathfrak{g}_{\rho}$). 
We denote by $\mathcal{B}:=(\xi_1,\ldots ,\xi_{\mu})$ a \emph{basis} 
(or system of simple roots) of this system $\mathcal{R}^+$ (such a basis 
exists). Then every root in $\mathcal{R}^+$ is a linear combination of 
the roots of this basis with positive coefficients (they are integers) 
and therefore $\mathcal{B}$ is a \emph{good system} of roots. Then the 
result follows from the proposition \ref{TL:propgoodsyst}.
\end{proof}


\subsection{Some complements on linear algebraic groups}
\label{subsection:complementsLAG}

We shall use later this part for the solution of the local inverse 
problem and in our study of the global inverse problem. Similar tools 
were introduced by the first author in order to solve inverse problems 
in the \emph{differential case}. For the missing proofs \cf\ 
\cite[11.3,11.4]{vdPS2}. \\

We denote by $L(G)$ the subgroup of an algebraic group $G$ generated by 
all the maximal tori of $G$, it is a connected algebraic normal subgroup 
and the maximal torus of the algebraic group $V(G):=G/L(G)$ is reduced to 
the identity.

\begin{lemm}
\label{generators}
The Lie algebra $\mathfrak{L}$ of $L(G)$ is generated by $\mathfrak{t}$ 
(the Lie algebra of a maximal torus) and the root-spaces $\mathfrak{g}_{\xi}$. 
\end{lemm}

The group $R_u(G)/\bigl(G^0,R_u(G)\bigr)$ is a \emph{commutative unipotent} 
group therefore it can be identified with a finite dimensional complex 
vector space. The finite group $G/G^0$ acts naturally on 
$R_u(G)/\bigl(G^0,R_u(G)\bigr)$.
\smallskip

We set $S(G):=R_u(G)/\bigl(G^0,R_u(G)\bigr)\rtimes G/G^0$. Due to a result 
of the first author \cite[Proposition 1.8, page 276]{vdPS2} , there is an 
isomorphim of algebraic groups:
\[
S(G)\rightarrow V(G)/\bigl(V(G)^0,V(G)^0\bigr). 
\]
\begin{lemm}
\label{nbgenerators}
The linear algebraic groups $S(G)$, $V(G)$ and $V(G)/(V(G)^0,V(G)^0)$ have 
the same number $m$ of topological generators. 

We have $\dim R_u(G)/\bigl(G^0,R_u(G)\bigr)\leq m$ and it is an equality 
if $G$ is connected.

If $G$ is topologically generated by $s$ elements, then $m\leq s$.
\end{lemm}

\begin{lemm}
\label{generators2}
Let $G$ be an algebraic group endowed with a $\Theta$-structure defined by 
a $\Theta$-coweight $\chi$ on a Levi subgroup $D\subset G$. Let 
$T\subset D$ be the maximal torus of $D$. We set $U:=R_u(G)$ and denote 
$\mathfrak{u}$ its Lie algebra. Then:
\begin{itemize}
\item[(i)]
$\mathfrak{u} =
\mathfrak{u}_0\oplus \bigoplus_{\xi\in\mathcal{R}} \mathfrak{g}_{\xi}$ and
$\bigoplus_{\xi\in\mathcal{R}} \mathfrak{g}_{\xi}$ is a sub-Lie algebra of
$\mathfrak{u}$;
\item[(ii)]
$\mathfrak{L}=\mathfrak{t}\oplus\bigoplus_{\xi\in\mathcal{R}} \mathfrak{g}_{\xi}$,
$\mathfrak{g} =
\mathfrak{g}_0\oplus \bigoplus_{\xi\in\mathcal{R}} \mathfrak{g}_{\xi}$,
$\mathfrak{g}_0=\mathfrak{u}_0\oplus\mathfrak{t}$,
$\mathfrak{g}=\mathfrak{u}_0\oplus\mathfrak{L}$;
\end{itemize}
\end{lemm}

\begin{proof}

\begin{itemize}
\item[(i)]
For any weight $\xi$ on $D$, $\mathfrak{u}_{\xi}\subset \mathfrak{g}_{\xi}$ 
and if $\xi$ is a root, $\mathfrak{g}_{\xi}=\mathfrak{u}_{\xi}$.

Let $\alpha,\beta\in\mathcal{R}$, we have $\mathfrak{g}_{\alpha+\beta}=(0)$ or
$[\mathfrak{g}_{\alpha},\mathfrak{g}_{\beta}]\subset \mathfrak{g}_{\alpha+\beta}$. 
As $\alpha+\beta\neq 0$, in the second case $\alpha+\beta$ is a root. 
Hence $\bigoplus_{\xi\in\mathcal{R}}\mathfrak{g}_{\xi}$ is a sub-Lie algebra 
of $\mathfrak{u}$.
 
\item[(ii)] For all $\xi\in\mathcal{R}$, 
$[\mathfrak{t},\mathfrak{g}_{\xi}]=\mathfrak{g}_{\xi}$, 
therefore, using (i) and the lemma \ref{generators}, we get 
$\mathfrak{L}=\mathfrak{t}\oplus\bigoplus_{\xi\in\mathcal{R}} \mathfrak{g}_{\xi}$.
\end{itemize}
\end{proof}



\section{The local inverse problem}
\label{section:LIP}


\subsection{The regular singular case}


\subsubsection{Universal groups and representations. Necessary conditions}

The universal group for the local regular singular case (at $0$) is the  
\emph{commutative} proalgebraic group:
\[
G_{f}^{(0)} = G_{f,s}^{(0)} \times G_{f,u}^{(0)},
\]
with:
\[
G_{f,s}^{(0)} = \Hom_{gr}(\Eq,\C^*) ~~\text{and}~ G_{f,s}^{(0)}= \C.
\]

To a germ (at the origin) of meromorphic $q$-difference system 
$\Delta: \sigma_q Y = AY$, up to meromorphic equivalence, corresponds 
a rational representation:
\[
\rho_{f} :G_{f}^{(0)} \rightarrow \GL_n(\C)
\]
and conversely. The $q$-difference Galois group of $\Delta$ is 
$G = \Im \rho_{f}$. It is \emph{abelian}.

The knowledge of the representation $\rho$ is equivalent to the knowledge 
of a pair of \emph{commuting} representations:
\[
\rho_{f,s} :G_{f,s}^{(0)} \rightarrow \GL_n(\C) \quad
\rho_{f,u} :G_{f,u}^{(0)} \rightarrow \GL_n(\C).
\]
We have $G_s = \Im\rho_{f,s}$ and  $G_u = \Im \rho_{f,u}$ and our commutation
condition means that each element of $G_s$ commutes with each element of $G_u$.

The commutative unipotent group $G_u$ being the image of $\C$ by $\rho_{f,u}$ 
its dimension is \emph{at most one}.

The group $\Hom_{gr}(\Eq,\C^*)$ is \emph{topologically generated} by 
(exactly) two elements \cite{JSGAL} and $\C$ is topologically generated 
by one element. Therefore the groups $G_s$ and $G$ are \emph{topologically 
generated} by at most \emph{two} elements, the group $G_u$ is topologically 
generated by one element. The \emph{finite} group $G/G^0$ is 
\emph{algebraically generated} by at most \emph{two} elements.  
 
\begin{prop}
\label{UGR:propdir}
Let $G$ be the Galois group of a local \emph{regular-singular} 
$q$-difference system, then:
\begin{itemize}
\item $G$ is \emph{abelian} and \emph{topologically generated} by two elements;
\item $G/G^0$ is \emph{algebraically generated} by at most two elements;
\item $\dim_\C G_u\leq 1$.
\end{itemize}
\end{prop}
 
\paragraph{A description of $\Hom_{gr}(\Eq,\C^*)$}

We recall the description of $\Hom_{gr}(\Eq,\C^*)$. We choose $\tau \in \C$ 
such that $e^{-2i\pi\tau} = q$ ($\Im \tau > 0$). The map 
$w \mapsto z := e^{2i\pi w}$ induces an isomorphism of 
$\C/(\Z \oplus \Z \tau)$ on $\Eq$. We consider $\C$ as a $\Q$-vector space, 
we can write it as a direct sum of $\Q$-vector spaces 
$\C = \Q \oplus \Q \tau \oplus \L$, then we have a product of $\Z$-modules
$\C/(\Z \oplus \Z \tau)\approx (\Q/\Z) \times (\Q\tau/\Z\tau) \times \L$ and
 the corresponding image is the product of $\Z$-modules:
 \[
 \Eq = \mmu \times \mmu_q \times \L,
 \]
where $\mmu := e^{2i\pi\Q}$ is the group of the roots of the unity,  
$\mmu_q =q^{\Q}/q^{\Z}$ is the image in $\Eq$ of the subgroup 
$q^{\Q} \subset \C^*$ ($\mmu_q \approx \mmu \approx \Q/\Z$)
and $\L$ is a torsion free subgroup (the ``universal lattice", defined up 
to isomorphism).

We shall consider each abelian group as the \emph{inductive limit} of its 
finitely generated subgroups.

Recall that the groups written $G^{(0)}$ are (universal) local Galois
groups at $0$, while $G^0$ means the neutral component of any proalgebraic
group $G$. We have a short exact sequence of proalgebraic groups:
\[
\bigl(G_{f,s}^{(0)}\bigr)^0 \rightarrow G_{f,s}^{(0)} \rightarrow 
G_{f,s}^{(0)}/\bigl(G_{f,s}^{(0)}\bigr)^0,
\]
we get it applying the exact contravariant functor $\Hom_{gr} (.,\C^*)$ 
to the short exact sequence of groups:
\[
\mmu \times \mmu_q \rightarrow 
\Eq \rightarrow \Eq/(\mmu \times \mmu_q).
\]
We have $\Eq\rightarrow \Eq/(\mmu \times \mmu_q) \approx \L$, therefore:
\[
\Hom_{gr}\bigl(\Eq/(\mmu\times\mmu_q),\C^*\bigr) \approx 
\Hom_{gr}\bigl(\L,\C^*\bigr).
\]
Hence 
$\bigl(G_{f,s}^{(0)}\bigr)^0 \approx 
\Hom_{gr}\bigl(\Eq/(\mmu\times\mmu_q),\C^*\bigr)$ is a \emph{protorus}, 
we will call it \emph{the fuchsian protorus} and we will denote it 
$\mathbf{T_f}$.
 
We recall that $\Hom_{gr}(\Q/\Z,\C^*)=\hat \Z$. Then 
$\Hom(\mmu,\C^*)=\hat \Z(1)$ ($\hat \Z(1)$ is the multiplicative 
notation for $\hat \Z$).
 
We have $G_{f,s}^{(0)}/\mathbf{T_f}\approx \hat \Z(1)\times \hat \Z(1)$.

Considering $G_{f,s}^{(0)}$ as a proalgebraic group, we get 
$\Eq = \Hom(G_{f,s}^{(0)},\C^*)$ (here $\Hom$ is for \emph{morphisms of 
proalgebraic groups} \ie\ rational homomorphisms), as the inductive 
limit of its \emph{finitely generated} subgroups. Then we can consider 
$\Eq$ as the \emph{group of weights} of $G_{f,s}^{(0)}$. More precisely, 
if $\bar c \in \Eq$, then the map 
$\psi_{\bar c}: G_{f,s}^{(0)} = \Hom_{gr}(\Eq,\C^*) \rightarrow \C^*$ defined by 
$f \in \Hom_{gr}(\Eq,\C^*) \mapsto f(\bar c)$ is a weight on $G_{f,s}^{(0)}$ 
and conversely if $\varphi$ is a weight on $G_{f,s}^{(0)}$, there exists 
a unique $\bar c\in\Eq$ such that $\varphi = \psi_{\bar c}$.

Applying the functor $\Hom(.,\C^*)$ to the (non-canonical) decomposition 
$G_{f,s}^{(0)} =\hat \Z(1) \times \hat \Z(1) \times \mathbf{T_f}$, we get the 
(non-canonical) decomposition $\Eq = \mmu \times \mmu_q \times \L$.


\subsubsection{The inverse problem for the regular-singular case, 
a tannakian solution}

We solve the inverse problem for the regular-singular case using the 
tannakian mechanism, proving that the conditions of the proposition 
\ref{UGR:propdir} are sufficient. Afterwards we will give an elementary proof.

\begin{prop}
\label{UGR:propinv}
Let $G$ be an \emph{abelian} complex linear algebraic group such that:
\begin{itemize}
\item[(i)]
$G$ is \emph{topologically} generated by at most \emph{two} elements;
\item[(ii)]
$\dim_\C G_u\leq 1$.
\end{itemize}
Then $G$ is the local Galois group of a local \emph{regular singular} 
meromorphic linear $q$-difference system.
\smallskip

The condition (i) can be replaced by the following (\emph{a priori} weaker) 
condition:
\begin{itemize}
\item[(iii)]
$G/G^0$ is generated by at most \emph{two} elements.
\end{itemize}
\end{prop}

\begin{proof}
We will give a tannakian proof, defining a \emph{surjective} morphism
$\rho: G_{f,s}^{(0)} \rightarrow G$. Then, if $r: G \rightarrow \GL_n(\C)$ 
is a faithful representation, the morphism 
$r \circ \rho: G_{f,s}^{(0)} \rightarrow \GL_n(\C)$ defines a system of 
rank $n$ whose Galois group is $r(G)$.
\smallskip

Let $G$ be an \emph{abelian} linear algebraic group satisfying the above
conditions, then $G = G_u \times G_s$. The natural map 
$G_s/G_s^0 \rightarrow G/G^0$ is an isomorphism, therefore there exists 
an isomorphism $G_s \approx \Z/p_1\Z \times \Z/p_2\Z\times (\C^*)^{\nu}$ 
($p_1,p_2\in\N^*$). We will suppose $p_1,p_2\geq 2$, leaving the other 
cases to the reader.

Using a sub-lattice of rank $n$ of $\L$, we get a surjective morphism 
$\mathbf{T_f} \rightarrow (\C^*)^{\nu}$. There exists also a surjective morphism
$\mmu \times \mmu_q \rightarrow \Z/p_1\Z \times \Z/p_2\Z$. Hence we get
a surjective morphism 
\[
\rho_s: G_{f,s}^{(0)} \approx \mmu \times \mmu_q \times \mathbf{T_f}
\rightarrow \Z/p_1\Z \times \Z/p_2\Z \times (\C^*)^{\nu}.
\]

The Lie algebra $\mathfrak{u}$ of $G_u$ is of dimension at most one. 
Therefore there exists $N \in \mathfrak{u}$ such that 
$G_u = \{\exp tN \tq t \in \C\}$. If $N = 0$, the end of the proof is trivial.
Otherwise, the map $\rho_u: G_{f,u}^{(0)} \approx \C \rightarrow G_u$ defined 
by $t \mapsto \exp tN$ is an isomorphism (of algebraic groups).

The representations $\rho_s$ and $\rho_u$ clearly commute and the morphism 
$\rho := (\rho_s,\rho_u): G_f^{(0)} = G_{f,s}\times G_{f,u} \rightarrow G$ 
is onto. That ends the proof.
\end{proof}


\subsubsection{Explicit descriptions and elementary proof}

We shall recall how to compute the Galois group of a local 
regular-singular $q$-difference system and shall deduce a (elementary) 
proof of proposition \ref{UGR:propinv} from this computation.

Up to a meromorphic gauge transformation, it is sufficient to consider 
the case of a \emph{constant coefficient} system $\Delta: \sigma_q Y=AY$, 
$A \in \GL_n(\C)$.

We suppose that the matrix $A$ is in upper triangular Jordan form.
The representation $\rho$ of the universal group 
$G_{f}^{(0)} = G_{f,s}^{(0)} \times G_{f,u}^{(0)}$
associated to the system $\Delta$ is:

\begin{equation}
\label{reprs}
\rho =(\rho_{s},\rho_{u}): \; (\gamma,\lambda) \mapsto 
\gamma(A_{s})A_{u}^{\lambda}.
\end{equation}

We have $A_{s} = \Diag~(a_{1},\ldots ,a_{n})$, then
$\gamma(A_{i,s}) =
\Diag~\bigr(\gamma (\overline a_{1}),\ldots ,\gamma (\overline a_{n})\bigl)$, 
$\overline a_{i}$ being the image of $a_{i}$ in $\Eq$.
 
Let $H$ be the subgroup of $\Eq$ generated by the image of 
$\text{Spec} A$. Using the decomposition
$\Eq = \mmu  \times\mmu_ \times \L$, we get (up to the isomorphism 
$\mmu_q \approx \mmu$)
$H=\Z/p_1\Z \times \Z/p_2\Z \times \Lambda$, where $\Lambda$ is a lattice 
of rank $\mu$.

The algebraic group $\Hom_{gr}(H,\C^*)$ is an algebraic quotient 
of $\Hom_{gr}(\Eq,\C^*)$ (using the canonical injection $H \rightarrow \Eq$) 
and the semi-simple component $G_s$ of the Galois group $G$ of
$\Delta$ is the image of the quotient map, that is $\Hom_{gr}(H,\C^*)$, then:
\[
G_s =\Z/p_1\Z \times\Z/p_2\Z \times \Hom_{gr}(H,\C^*) \approx  
\Z/p_1\Z \times \Z/p_2\Z \times (\C^*)^{\mu}.
\]

More precisely we get the representation of $G_s$ in $\GL_n(\C)$ 
corresponding to (\ref{reprs}) using the interpretation of $H$ as 
the group of weights of $G_s$. This representation is given by the 
diagonal weights  $(\overline a_{i})_{i=1,\ldots,n}$ ($\overline a_{i}\in H$). \\

We can now solve explicitly the inverse problem.

Let $G$ be an \emph{abelian} complex linear algebraic group satisfying 
the conditions of the proposition \ref{UGR:propinv}, we will compute 
a matrix $A\in\GL_n(\C)$ such that the system $\Delta: \sigma_q Y=AY$ 
admits $G$ as Galois group.

More precisely, we start from a faithful \emph{representation} of 
the \emph{abelian} group $G$ in $\GL_n$ in upper triangular form. 
Then $G_s$ is diagonal and (due to condition \emph{(ii)}) there exists 
a \emph{unipotent} matrix $N \in \M_n(\C)$ such that 
$G_u = \{N^{\lambda} \tq \lambda \in \C\}$.

The abelian linear algebraic group $G_s$ is isomorphic to the product 
of a finite group (the quotient $G/G^0$) by a torus of dimension $\mu$, 
and the finite component is generated by at most two elements. Then
$G_s \approx \Z/p_1\Z \times \Z/p_2\Z \times (\C^*)^{\mu}$. The dual group 
(group of weights) of 
$\Z/p_1\Z \times \Z/p_2\Z \times (\C^*)^{\mu}$ is
 $\Z/p_1\Z \times \Z/p_2\Z \times \Z^{\mu}$

Using the decomposition $\Eq = \mmu \times \mmu_q \times \L$, we get 
an isomorphism between $\Z/p_1\Z \times \Z/p_2\Z \times \Z^{\mu}$ and
a subgroup $H$ of $\Eq$. We can therefore interpret $H$ as the group 
of weights on the diagonal group $G_s$. 

We denote $\varpi_1,\ldots ,\varpi_n$ the \emph{diagonal weights} of the 
diagonal group $G_s$, they are elements of $H$. Let $a_1,\ldots ,a_n\in\C^*$ 
such that their natural images in $\Eq$ are $\varpi_1,\ldots ,\varpi_n$.
We moreover require these choices to be consistent in the following sense:
each time $\varpi_i = \varpi_j$, we take $a_i = a_j$. Then $H$ is generated by 
$\overline a_1 = \varpi_1,\ldots ,\overline a_n = \varpi_n$.

We can now define $A \in \GL_n(\C)$:
\[
A_s:=\Diag~(a_1,\ldots ,a_n) \quad  \text{and} \quad A_u:=N.
\]
Indded, because of our consistent choices above, $A_s$ and $A_u$ do commute.
Then the Galois group of $\Delta: \sigma_q Y=AY$ is $G$.


\subsection{The pure case with integral slopes}


\subsubsection{Universal groups and representations. Necessary conditions}

The universal group for the pure case with integral slopes (at $0$) is 
the  \emph{commutative} proalgebraic group:
\[
G_{p,1}^{(0)} =G_{f,s}^{(0)} \times G_{f,u}^{(0)} \times T_1^{(0)},
\]
with:
\[
G_{f,s}^{(0)} = \Hom_{gr}(\Eq,\C^*), \quad G_{f,s}^{(0)} = \C  
~~\text{and}~~T_1^{(0)} = \C^*.
\]

To a germ (at the origin) of meromorphic $q$-difference system, pure with 
integral slopes, up to meromorphic equivalence, corresponds a morphism:
\[
\rho :G_{p,1}^{(0)} \rightarrow \GL_n(\C);
\]
$G =\Im \rho$ is the Galois group of the system, it is \emph{commutative}.

The knowledge of the representation $\rho$ is equivalent to the knowledge 
of a triple of pairwise commuting representations:
\[
\rho_{f,s} :G_{f,s}^{(0)}\rightarrow \GL_n(\C), \quad
\rho_{f,u} :G_{f,u}^{(0)}\rightarrow \GL_n(\C), \quad
\rho_{\theta}:T_1^{(0)}\rightarrow \GL_n(\C)
\]
We have (up to the obvious reordering of the factors) 
$G_s =\Im(\rho_{f,s},\rho_{\theta})$, $G_u=\Im\rho_{f,u}$.

As in the regular singular case, we get the following result.

\begin{prop}
\label{UGR:propdirpure}
Let $G$ be the Galois group of a local \emph{pure} $q$-difference system, then:
\begin{itemize}
\item $G$ is \emph{abelian} and \emph{topologically generated} by two elements;
\item $G/G^0$ is \emph{algebraically generated} by at most two elements;
\item $\dim_\C G_u\leq 1$.
\end{itemize}
\end{prop}

We recall that we have a (non-canonical) decomposition 
$G_{f,s}^{(0)}=\hat \Z(1) \times \hat \Z(1) \times \mathbf{T_f}$, where
the fuchsian universal protorus is a \emph{subgroup} of $G_{f,s}^{(0)}$.

We will denote $T_{f}$ the image of $\mathbf{T_f}$ by $\rho_{f,s}$ and 
we will call it the \emph{fuchsian torus} of $G$. We will denote $T_{\theta}$ 
the image of $\mathbf{T_1}^{(0)}$ by $\rho_{f,s}$ and we will call it 
the $\theta$-torus of $G$. The $\theta$-torus and the fuchsian torus 
of $G$ generate the \emph{maximal torus} of $G$.
\bigskip


\subsubsection{Sufficient conditions}

\begin{prop}
\label{UGR:propinvpure}
Let $G$ be an \emph{abelian} complex linear algebraic group and a 
non trivial coweight $\chi:\C^*\rightarrow G_s$. We suppose that:
\begin{itemize}
\item[(i)]
$G$ is \emph{topologically} generated by at most \emph{two} elements;
\item[(ii)]
$\dim_\C G_u\leq 1$.
\end{itemize}
Then $G$ is the local Galois group of a local \emph{pure} meromorphic 
linear $q$-difference system with integral slopes such that 
$\chi = \rho_{\theta}$ (where $\rho =(\rho_f,\rho_{\theta})$ is 
the representation defining the system). The condition (i) can be replaced 
by the following (a priori weaker) condition:
\begin{itemize}
\item[(iii)]
$G/G^0$ is generated by at most \emph{two} elements.
\end{itemize}
\end{prop}

\begin{proof}
We will prove the existence of a system such that its fuchsian torus
$T_{f}$ is a maximal torus, or equivalently such that $T_{\theta}\subset T_{f}$.
The proof is tannakian and it is only a slight modification of the proof 
of the proposition \ref{UGR:propinv}.

Let $T$ be the maximal torus of $G$, it contains the image of $\chi$. 

We build as above a \emph{surjective} representation:
\[
\rho_{f} = (\rho_{f,s},\rho_{f,u}): \;
G_{f}^{(0)} = G_{f,s}^{(0)} \times G_{f,u}^{(0)} \rightarrow G.
\]
Then using 
$\mathbf{T}_1^{(0)}=\C^*$, we define a representation
\[
\rho = (\rho_f,\rho_{\theta}): \;
G_{p,1}^{(0)} = G_f^{(0)}\times\mathbf{T}_1^{(0)} \rightarrow G,
\]
by $\rho_{\theta} = \chi$. (The component representations automatically commute.)
It is a surjective morphism and it answers the question.

\end{proof}


\subsubsection{Explicit descriptions}

We recall how to compute the Galois group of a local pure $q$-difference 
system with integral slopes and deduce a new (elementary) proof of 
proposition \ref{UGR:propinvpure} from this computation.

Up to a meromorphic gauge transformation, it is sufficient to consider 
the case of a  system $\Delta: \sigma_q Y = AY$, 
such that the matrix $A$ is in upper triangular normal form:

\begin{equation}
A:= \begin{pmatrix}
z^{\mu_1}A_{1}  & \ldots & \ldots & \ldots & \ldots \\
\ldots & \ldots & \ldots  & 0 & \ldots \\
0      & \ldots & \ldots   & \ldots & \ldots \\
\ldots & 0 & \ldots  & \ldots & \ldots \\
0      & \ldots & 0       & \ldots & z^{\mu_k}A_{k}
\end{pmatrix},
\end{equation}
where, for $1 \leq i  \leq k$, $A_{i}\in \GL_{r_i}(\C)$ is in Jordan form, and
$\mu_1,\ldots,\mu_k \in \Z$. (Usually we take $\mu_1 < \cdots < \mu_k$, although
this has no consequence in the \emph{formal} case.)

The representation $\rho$ of the universal group 
$G_{p,1}^{(0)} = G_{f,s}^{(0)} \times G_{f,u}^{(0)} \times T_1^{(0)} $
associated to the system $\Delta$ is:

\begin{equation}
\label{reprspure}
\rho =(\rho_{s},\rho_{u},\rho_{\theta}): \; (\gamma,\lambda,t)\mapsto 
\begin{pmatrix}
t^{\mu_1}\gamma(A_{1,s})A_{1,u}^{\lambda}  & \ldots & \ldots & \ldots & \ldots \\
\ldots & \ldots & \ldots  & 0 & \ldots \\
0      & \ldots & \ldots   & \ldots & \ldots \\
\ldots & 0 & \ldots  & \ldots & \ldots \\
0      & \ldots & 0       & \ldots & t^{\mu_k}\gamma(A_{k,s})A_{k,u}^{\lambda}  
\end{pmatrix}.
\end{equation}

We can now give a new proof of proposition \ref{UGR:propinvpure}.
We start from the abelian group $G$ and the one-parameter subgroup $\chi$. 
We can assume that it is diagonalized:
$$
\forall t \in \C^* \;,\; \chi(t) =
\begin{pmatrix}
 t^{\mu_1} I_{r_1}  & \ldots & \ldots & \ldots & \ldots \\
\ldots & \ldots & \ldots  & 0 & \ldots \\
0      & \ldots & \ldots   & \ldots & \ldots \\
\ldots & 0 & \ldots  & \ldots & \ldots \\
0      & \ldots & 0       & \ldots &  t^{\mu_k} I_{r_k}
\end{pmatrix},
$$
and we apply the explicit proof of proposition \ref{UGR:propinv} to each
of the regular-singular blocks of ranks $r_i$, yielding matrices $A_i$
with constant coefficients. Then we set:
\[
A:= \begin{pmatrix}
z^{\mu_1}A_{1}  & \ldots & \ldots & \ldots & \ldots \\
\ldots & \ldots & \ldots  & 0 & \ldots \\
0      & \ldots & \ldots   & \ldots & \ldots \\
\ldots & 0 & \ldots  & \ldots & \ldots \\
0      & \ldots & 0       & \ldots & z^{\mu_k}A_{k}
\end{pmatrix},
\]
The image of $\chi$ is contained in $G$, therefore the Galois group of 
the system $\sigma_q Y=AY$ is $G$ and we have $\chi = \rho_{\theta}$.


\subsection{The local inverse problem: the general case with integral slopes}
\label{subsection:NC}


\subsubsection{Necessary conditions}
 
\begin{theo}
\label{NC:th0}
Let $G$ be a  complex linear algebraic subgroup. If $G$ is the local 
Galois group of a meromorphic linear $q$-difference equation, then:

\begin{itemize}
\item[(i)] $G$ is triangularizable;
\item[(ii)]
$G/L(G)$ is abelian  and \emph{topologically} generated by at most 
\emph{two} elements;
\item[(iii)]
the finite group $G/G^0$ is abelian  and generated by at most \emph{two} 
elements;
\item[(iv)]
the group $G/G^0$ acts \emph{trivially} on $R_u(G)/\bigl(G^0,R_u(G)\bigr)$ 
and the dimension of the vector space $R_u(G)/\bigl(G^0,R_u(G)\bigr)$ is 
at most \emph{one}.
 \end{itemize}
\end{theo}

\begin{proof}
\begin{itemize}
 \item[] \quad
\item[(i)] Trivial.
\item[(ii)]
We will use a Tannakian argument which is a variant of an idea due 
to O. Gabber in the differential case \cite{K}.

Let $G^{(0)}$ be the Tannakian group of the Tannakian category 
$\mathcal{E}^{(0)}$. To a $q$-difference system of rank $n$, meromorphic 
at the origin, corresponds a (rational) representation 
$\rho: G^{(0)} \rightarrow \GL_n(\C)$ and conversely. 
If $G = Gal_{\Ka}(\Delta)$ is the Galois group of $\Delta$, 
then $G = \Im \rho$.

Let $\pi:G \rightarrow G/L(G)$ be the canonical map, let 
$\iota: G/L(G)\rightarrow \GL_{n'}(\C)$ be a faithful linear representation 
of $G/L(G)$, then we get a continuous linear representation 
$\rho': \iota \circ \pi \circ \rho: G^{(0)} \rightarrow \GL_{n'}(\C)$. 

To the representation $\rho'$ corresponds a $q$-difference system 
$\Delta'$ of rank $n'$ and and 
$G':= \iota(G/V(G)) = Gal_{\Ka}(\Delta') = \Im \rho'$. 

The maximal torus of $G'$ is reduced to the identity, therefore the 
$\theta$-torus of $\Delta'$ is trivial and $\Delta'$ is \emph{regular 
singular}.

Hence the Galois group $G'$ of $\Delta'$ is \emph{abelian} and 
\emph{topologically generated} by at most \emph{two} elements \cite{JSGAL}. 
Moreover $G'=G'_sG'_u$, where the unipotent group $G'_u$ is 
\emph{topologically generated} by at most \emph{one} element \cite{JSGAL}.
\item[(iii)]
We have $G/L(G) = V(G) \approx G'$. The group $G/G^0$ is a quotient 
of $V(G)$ therefore it is \emph{abelian} and topologically generated 
by at most two elements, as it is \emph{finite} it is \emph{algebraically} 
generated by at most \emph{two} elements. 
\item(iv) 
We set as in section \ref{subsection:complementsLAG}
$S(G) := R_u(G)/\bigl(G^0,R_u(G)\bigr) \rtimes G/G^0$, 
we recall that there is an isomorphim of algebraic groups
$S(G) \rightarrow V(G)/\bigl(V(G)^0,V(G)^0\bigr)$. The group $V(G)$ being 
commutative, we get an isomorphism $S(G)\rightarrow V(G)$, $S(G)$ is 
commutative and the action of $G/G^0$ on $R_u(G)/\bigl(G^0,R_u(G)\bigr)$ 
is \emph{trivial}.

We have an isomorphism 
$S(G)_u = R_u(G)/\bigl(G^0,R_u(G)\bigr) \rightarrow V(G)_u$. As $V(G)_u$,  
$S(G)_u$ is topologically generated by at most \emph{one} generator. Then 
$\dim_\C R_u(G)/\bigl(G^0,R_u(G)\bigr) \leq 1$.  
\end{itemize}
\end{proof}

We think that the four \emph{necessary} conditions of the above theorem 
are \emph{not sufficient}. Anyway if we want to realize $G$ as the Galois 
group of a meromorphic linear $q$-difference system whose Newton polygon 
has \emph{integral} slopes, then there is a \emph{new} necessary condition 
(\emph{(vi)} of the following theorem). This condition is not trivial: 
there exists a solvable linear algebraic group satisfying the conditions 
\emph{(ii)}, \emph{(iii)}, \emph{(iv)} of theorem \ref{NC:th0} which does 
not satisfies the condition \emph{(vi)} of theorem \ref{NC:th} below (\cf\
\ref{TL:contrex0} of remark  \ref{TL:contrex}, page \pageref{TL:contrex}).
 
\begin{theo}
\label{NC:th}
Let $G$ be a  complex linear algebraic subgroup. If $G$ is the local Galois 
group of a meromorphic linear $q$-difference system whose Newton polygon 
has \emph{integral slopes}, then:

\begin{itemize}
\item[(i)] $G$ is triangularizable;
\item[(ii)] $G/L(G)$ is is abelian and \emph{topologically} generated 
by at most \emph{two} elements;
\item[(iii)]
$G/G^0$ is abelian  and generated by at most \emph{two} elements;
\item[(iv)]
the dimension of the unipotent component ot the abelian group $G/L(G)$ 
is \emph{at most one};
\item[(v)]
the dimension of $R_u(G)/\bigl(G^0,R_u(G)\bigr)$ is at most \emph{one};
\item[(vi)]
there exists a $\Theta$-structure on $G$.
\end{itemize}
\end{theo}

\begin{proof}
Assertions (i) to (v) follow from the proposition \ref{UGR:propdir} 
and theorem \ref{NC:th0}.

It remains to prove (vi). 

Every system with integral slopes admits, up to meromorphic equivalence, 
a Birkhoff-Guenther normal form, therefore it is sufficient to prove 
the result for a system $\sigma_q Y = AY$ in Birkhoff-Guenther normal form:

\begin{equation}
\label{ABnm}
A = A_{U} := \begin{pmatrix}
B_{1}  & \ldots & \ldots & \ldots & \ldots \\
\ldots & \ldots & \ldots  & U_{i,j} & \ldots \\
0      & \ldots & \ldots   & \ldots & \ldots \\
\ldots & 0 & \ldots  & \ldots & \ldots \\
0      & \ldots & 0       & \ldots & B_{k}
\end{pmatrix},
\end{equation}
 where, for $1 \leq i < j \leq k$,
$U_{i,j} \in \Mat_{r_{i},r_{j}}(\Ka)$. Here, $U$ stands short for
$(U_{i,j})_{1 \leq i < j \leq k} \in \prod\limits_{1 \leq i < j \leq k}
\Mat_{r_{i},r_{j}}(\Ka)$. (This requirement is actually weaker than the 
true Birkhoff-Guenther normal form, where the $U_{i,j}$ would have
polynomial coefficients, \cf\ section \ref{subsection:overallstructure}.)

We suppose that:
\[
 B_i = z^{\mu_i}A_i,~ A_i\in\GL_{r_i}(\C), ~
\mu_1<\cdots  <\mu_i< \cdots <\mu_k,
\]
and we set (\eqref{eqn:formesstandards}):
\[
A_{0} := \begin{pmatrix}
B_{1}  & \ldots & \ldots & \ldots & \ldots \\
\ldots & \ldots & \ldots  & 0 & \ldots \\
0      & \ldots & \ldots   & \ldots & \ldots \\
\ldots & 0 & \ldots  & \ldots & \ldots \\
0      & \ldots & 0       & \ldots & B_{k}
\end{pmatrix} 
\]

We firstly consider the differential Galois group $G_0$ of the pure system
$\sigma_q Y = A_0Y$. This group is abelian, in upper triangular form, and 
its semi-simple component $D:=(G_0)_s$ is \emph{diagonal}.  

We define a coweight of $D$ by
$\chi: t \mapsto (t^{\mu_{1}}I_{r_1},\ldots,t^{\mu_{k}}I_{r_k})$ and we denote 
its image by $T_{\theta}$ (the theta-torus). Then the maximal torus $T$ 
of $D$ is generated by the ``fuchsian torus" $T_{f}$ and $T_{\theta}$.

We consider now the differential Galois group $G$ of the system
$\sigma_q Y=AY$. It is on upper triangular form, it contains $G_0$ 
as a subgroup, moreover $D$ is a Levi subgroup of $G$.

We denote $\varpi_1,\ldots ,\varpi_n$ the \emph{diagonal weights} of $D$. 
The root for the adjoint action of $D$ on $\mathfrak{g}$ are elements of 
the set $\{\varpi_i\varpi_j^{-1} \tq i<j\}$. If $\xi$ is a root, 
the corresponding root space is not trivial: there exists 
$x \in \mathfrak{g}_{\xi}$ such that $x \neq 0$. Therefore there exists 
$i,j$, with $i<j$, such that $x_{ij} \neq 0$, then
$<\xi,\chi> = \mu_i - \mu_j < 0$ ($\varpi_i \circ \chi(t) = t^{\mu_i}$). 
Hence $\chi$ is a $\Theta$-structure on $G$.
\end{proof}


\subsubsection{Sufficient conditions}
\label{subsubsection:SC}
  
We will prove in this part that the conditions of the theorem \ref{NC:th} 
are sufficient. 

From lemma \ref{generators2} we deduce the following preliminary result.

\begin{lemm}
\label{lemmdim}
Let $G$ be a complex linear algebraic group admitting a $\Theta$-structure. 
Then the following conditions are equivalent:
\item[(i)]
the dimension of $R_u(G)/\bigl(G^0,R_u(G)\bigr)$ is at most one;
\item[(ii)]
if $G = U \rtimes D$ is a Levi decomposition, then the dimension of 
$\mathfrak{u}_0$ (which was defined in lemma \ref{generators2}) is at most one;
\item[(iii)]
if $D$ is a Levi subgroup of $G$, then $\dim~C_G(D) \leq \dim~D+1$
(we write $C_G(D)$ the centralizer of $D$).

If these conditions are satisfied, then: 
\[
\dim~R_u(G)/\bigl(G^0,R_u(G)\bigr) = \dim~\mathfrak{u}_0 = 
\dim~C_G(D) - \dim~D.
\]
\end{lemm}

\begin{theo}
\label{NC:thsuff}
Let $G$ be a complex linear algebraic group such that:
\begin{itemize}
\item[(i)]
$G/L(G)$ is abelian  and \emph{topologically} generated by at most 
\emph{two} elements;
\item[(ii)]
the dimension of $R_u(G)/\bigl(G^0,R_u(G)\bigr)$ is at most one;
\item[(iii)]
$G$ admits a $\Theta$-structure, 
\end{itemize}
then $G$ is the local Galois group of a meromorphic linear $q$-difference 
system whose Newton polygon has \emph{integral slopes}.

More precisely, if $r: G \rightarrow \GL_n(\C)$ is a faithful representation, 
it is possible to find a meromorphic linear $q$-difference system whose 
Newton polygon has \emph{integral slopes} and whose Galois group is $r(G)$.
\end{theo}

It is possible to replace the condition \emph{(ii)} by the following:
\emph{\begin{itemize}
\item[(ii')]
the dimension of the unipotent component of the abelian group
$G/L(G)$ is at most one.
\end{itemize}}

\begin{proof}

The proof is tannakian, starting from an algebraic group $G$, we will 
obtain the system as a rational representation of the total Galois 
group with integral slopes:
\[
\rho = (\rho_w,\rho_{p,1}): 
G_1^{(0)} = \mathfrak{St}\, \rtimes G_{p,1}^{(0)} \rightarrow G
\] 
whose image is $G$.

We will built this representation using our main result on the description
of the representations of the Tannakian group $G_1^{(0)}$ via the 
representations of the \emph{wild fundamental group} $L \rtimes G_{p,1,s}^{(0)}$. 
We recall (\cf\ sections \ref{section:previousresults} and
\ref{section:structStokes}) that the knowledge of $\rho_w$ is equivalent 
to the knowledge of its infinitesimal counterpart $L\rho_w$ and that the 
knowledge of $L\rho_w$ is equivalent to the knowledge of a representation 
of $L$: $\lambda:L\rightarrow \mathfrak{g}$, compatible with the 
corresponding adjoint actions of $G_{p,1,s}^{(0)}$ and 
$\rho_{p,1}\bigl(G_{p,1,s}^{(0)}\bigr)$ (\cf\ section 
\ref{subsection:firstlookrepresentations}).
Moreover we have $\Im L\rho_w = \Im \lambda$.
\bigskip

Let $G$ be a \emph{triangularizable} complex linear algebraic goup. 
Let $\chi$ be a $\Theta$-coweight on a Levi subgroup $D$ of $G$. 
Using lemma \ref{domlemm}, we can suppose that $\chi$ is \emph{dominant}.
\smallskip

We will build the representation $\rho$ in three steps:

\begin{itemize}
\item[--]
we will define a rational representation $\rho_f: G_f^{(0)} \rightarrow G$, 
whose image is the centralizer $C_G(D)$;
\item[--]
using the coweight $\chi:\C^*\rightarrow D$ and the canonical injection 
$D\rightarrow G$, we get a morphism $\rho_{\theta}: \C^* \rightarrow G$ and 
we define a rational representation 
$\rho_{p,1} = (\rho_{\theta},\rho_f): G_{p,1}^{(0)} \rightarrow G$;
\item[--]
we will define a representation $\lambda: L \rightarrow \mathfrak{g}$ 
such that, if $ L\rho_w: \mathfrak{st} \rightarrow \mathfrak{g}$ is 
the associated representation, then
$\rho := (\rho_w,\rho_{p,1}): G_1^{(0)} \rightarrow G$ is onto.
\end{itemize}

\paragraph{Definition of $\rho_f$ and $\rho_{p,1}$}

We consider the centralizer $C_G(D)$, its Lie algebra is 
$\mathfrak{u}^0 \oplus \mathfrak{t}$ and, according to the hypothesis 
and to lemma \ref{lemmdim}, $\dim \mathfrak{u}^0 \leq 1$. We choose 
a generator $N$ of the vector space $\mathfrak{u}^0$, then
$C_G(D) = U_0 \times D$, where 
$U_0 = \exp \mathfrak{u}^0 = \{\exp tN \tq\ t \in \C\}$, in particular 
$C_G(D)$ is abelian. 

We consider the abelian algebraic group $C_G(D)$ and the coweight $\chi$ on
$D \subset C_G(D)$.  They satisfy the conditions of proposition 
\ref{UGR:propdirpure} ($C_G(D)/(C_G(D))^0 \approx G/G^0$), therefore 
there exists a representation
\[
\rho'_{p,1} : G_{p,1}^{(0)} \rightarrow C_G(D)
\]
such that $\Im \rho'_{p,1} = C_G(D)$ and such that the corestriction of
$\rho'_{\theta}: \mathbf{T}_1^{(0)} \rightarrow C_G(D)$ to $D$ is equal to 
the coweight $\chi$.
 
By composition of $\rho'_{p,1}$ by the canonical injection  
$C_G(D) \rightarrow G$, we get a representation:
\[
\rho_{p,1} = \rho'_{p,1} \circ r :G_{p,1}^{(0)} \rightarrow G.
\]
Its image is topologically generated by $D$ and $\exp N$.

\paragraph{Definition of $\rho_w$ and $\rho$}

We want to extend the representation $\rho_{p,1}$ into a \emph{surjective} 
representation
\[
\rho = (\rho_{w},\rho_{p,1}):
G_1^{(0)} = \mathfrak{S}\, \rtimes G_{p,1}^{(0)} \rightarrow G.  
\]
  
As we recalled above, the knowledge of $\rho_{w}: \mathfrak{St} \rightarrow G$ 
is equivalent to the knowledge of a representation:
\[
\lambda: L \rightarrow \mathfrak{g}, 
\]
the images of $\lambda$ and $L\rho_{w}$ being equal. \\

We will define $\lambda$ such that its image contains all the root spaces 
$\mathfrak{g}_{\xi}$, $\xi \in \mathcal{R}$. \\

We have a surjective map 
\[
\rho_{f,s}: G_{f,s}^{(0)} = \Hom_{gr}(\Eq,\C^*) \rightarrow D.
\]
Let $\xi$ be a root on $D$, then $\xi \circ \rho_{f,s}$ defines a weight on 
$\Hom_{gr}(\Eq,\C^*)$, that is an element $\bar c \in \Eq$. We set 
$<\xi,\chi> =: -\delta$, $\delta \in \N^*$. Therefore to each root $\xi$ 
we associate a label $(\delta,\overline c) \in \N^* \times \Eq$. We denote 
by $\Sigma \subset \N^* \times \Eq$ the \emph{finite} subset of labels 
obtained from the roots by this procedure.
 
If $(\delta,\overline c) \notin \Sigma$, for all $i=1,\ldots ,\delta$, 
we set $\lambda(\Der_i^{(\delta,\overline c)}) := 0$. It remains to define
$\lambda(\Der_i^{(\delta,\overline c)})$ for $(\delta,\overline c) \in \Sigma$ 
and for all $i= 1,\ldots ,\delta$.

We set $d_{\xi} := \dim~\mathfrak{g}_{\xi}$. The $\Theta$-coweight $\chi$ 
is \emph{dominant}, therefore $d_{\xi} \leq \delta$. We choose a \emph{basis} 
$(e_{\xi,1},\ldots ,e_{\xi,d_{\xi}})$ of the vector space $\mathfrak{g}_{\xi}$ 
and we set $\lambda(\Der^{(\delta,\overline c)}) := e_{\xi,i}$
if $i=1,\ldots,d_{\xi}$ and $\Der_i^{(\delta,\overline c)}) := 0$ if
$i = d_{\xi}+1,\ldots,\delta$.

Then, for every root $\xi \in \mathcal{R}$, the image of $\lambda$ contains 
the root space $\mathfrak{g}_{\xi}$.

\paragraph{End of the proof}

By construction, the image of $\lambda$, and therefore the image 
of $L\rho_{w}$ contains the sum of the root spaces
$\bigoplus_{\xi\in\mathcal{R}} \mathfrak{g}_{\xi} =
\bigoplus_{\xi\in\mathcal{R}} \mathfrak{u}_{\xi}$ 
and the image of $L\rho_{f,u}$ is ${u}_0$. Therefore the image of $L\rho$ 
contains $\mathfrak{u}$ and the image of $\rho$ contains $U = R_u(G)$. 
The image of $\rho$ contains also $C_G(D)$ and a fortiori $D$. Finally 
the image of $\rho$ is $G$.
\end{proof}

Using proposition \ref{TL:borelth} we get the following result.

\begin{coro}
\label{NC:borelcoro}
Let $G$ be a \emph{Borel subgroup} of a \emph{connected reductive} 
algebraic group, then it is the local Galois group of a meromorphic 
linear $q$-difference system whose Newton polygon has \emph{integral slopes}.
\end{coro}



\section{About the global inverse problem}
\label{section:GIP}

We have a ``glueing" lemma.

\begin{lemm}
\label{GIP:glueing}
\begin{itemize}
\item[(i)]
Let $A^{(0)}$ (resp. $A^{(\infty)}$) be an object of $\EE_1^{(0)}$ 
(resp. $\EE_1^{(\infty)}$). We suppose that $A^{(0)}$ and
$A^{(\infty)}$ are in Birkhoff-Guenther normal form and that 
$A_f^{(0)} = A_f^{(\infty)} \in \GL_n(\C)$. Let $G_1^{(0)}$ (resp. $G_1^{(\infty)}$) 
be the Galois group of $A^{(0)}$ (resp. $A^{(\infty)}$) and $G$ the Galois 
group of the global system defined by $(A^{(0)},I_n,A^{(\infty)})$. 
Then $G$ is the Zariski closure in $\GL_n(\C)$ of the subgroup generated 
by $G_1^{(0)}$ and $G_1^{(\infty)}$.
\item[(ii)]
Let $G^+$ and $G^-$ be two \emph{connected} algebraic subgroups 
of $\GL_n(\C)$ satisfying the conditions of theorem \ref{NC:thsuff} 
(or equivalently such that they are local Galois group of meromorphic 
linear $q$-difference systems whose Newton polygon have \emph{integral 
slopes}). We suppose that $G^+$ and $G^-$ admit a \emph{same} maximal 
torus. We denote $G$ the Zariski closure in $\GL_n(\C)$ of the subgroup 
generated by $G^+$ and $G^-$. Then $G$ is the global Galois group of a 
meromorphic linear $q$-difference system whose Newton polygon has 
\emph{integral slopes} at $0$ and $\infty$.
\end{itemize}
\end{lemm}

\begin{proof}
\begin{itemize}
\item[(i)]
follows easily from theorem \ref{theo:structureglobale}.
\item[(ii)]
Going back to the proof of theorem \ref{NC:thsuff}, we can find $A^{(0)}$ 
(reps. $A^{(\infty)}$) such that $G^+$ (reps. $G^-$) is the Galois group 
of $A^{(0)}$ (reps. $A^{(\infty)}$) and such that  $A_f^{(0)}=A_f^{(\infty)}$ 
(we choose $A_f^{(0)}$ such that the subgroup generated by its semi-simple 
part is Zariski dense in $T$). Then the result follows from (i).
\end{itemize}
\end{proof}
 
\begin{prop}
\label{NC:thsuff2}
Let $G$ be a \emph{connected reductive} linear algebraic group, 
then $G$ is the global Galois group of a meromorphic linear $q$-difference 
system whose Newton polygons at $0$ and $\infty$ have \emph{integral slopes}. \\
Moreover it is possible to get a $q$-difference system admitting $G$ as 
a Galois group with a \emph{trivial} (generalized) Birkhoff connection 
matrix and such that the local groups at $0$ and $\infty$ are \emph{Borel 
subgroups}.
\end{prop}

\begin{proof}
If the maximal torus of $G$ is \emph{trivial}, then the conditions of 
the proposition \ref{UGR:propdir} are satisfied, therefore $G$ is the 
Galois group of a \emph{local} regular singular equation. It is easy 
to conclude using \cite{JSGAL}.

We can suppose that $G\subset \GL_n(\C)$ and that the maximal torus $T$ 
of $G$ is \emph{not trivial} and in diagonal form.

We denote $G^+$ and $G^-$ two opposite Borel subgroups of $G$ and we choose 
as explained above a coweight $\chi$ of $T$ such that $\chi$ is a 
$\Theta$-coweight for $G^+$ and $\chi^{-1}$ is a $\Theta$-coweight for $G^-$.
Using \ref{NC:borelcoro} we prove that $G^+$ (resp. $G^-$) is the local 
Galois group of a meromorphic linear $q$-difference system whose Newton
polygon has integral slopes. We end the proof using lemma \ref{GIP:glueing}.
\end{proof}

\begin{theo}
\label{NC:thsuff3}
Let $G$ be a \emph{connected} linear algebraic group,. We suppose that 
the dimension of the vector space $R_u(G)/(G,R_u(G))$ is at most $2$.
Then $G$ is the global Galois group of a rational linear $q$-difference 
system whose Newton polygons at $0$ and $\infty$ have \emph{integral slopes}.
\end{theo}

In particular we can apply this result to a connected group. 
It generalizes proposition \ref{NC:thsuff2}.

\begin{proof}
If the maximal torus of $G$ is \emph{trivial}, then the conditions of 
proposition \ref{UGR:propdir} are satisfied, therefore $G$ is the 
Galois group of a \emph{local} regular singular equation. It is easy 
to conclude using \cite{JSGAL}.

We can suppose that $G\subset \GL_n(\C)$ and that the maximal torus $T$ 
of $G$ is \emph{not trivial} and in diagonal form.

\begin{lemm}
There exists a coweight $\chi$ on $T$ which is \emph{non null} 
on each \emph{root} $\xi$ for the adjoint action of $T$ on the Lie 
algebra $\mathfrak{g}$ of $G$:  $<\xi,\chi>\, \neq 0$.
\end{lemm}

\begin{proof}
The proof is a variant of an argument used above.

Let $\Phi:(\C^*)^{\mu}\rightarrow T$ be an isomorphism of tori.

Let $\Sigma = \{\xi_1,\ldots ,\xi_{\nu}\}$ be the set of of roots of $G$. 
For $i = 1,\ldots,\nu$, we set $f_i:= L\xi_i \circ L\Phi$.
We interpret $f_1,\ldots ,f_{\mu'}$ as linear forms on $\R^{\mu}$. 
There exists $p = (p_1,\ldots ,p_{\mu})\in\Z^{\mu}$ such that $f_i(p) \neq  0$ 
for all $i = 1,\ldots ,\nu$ (\cf\ lemma \ref{TL:lem}). We define a coweight
$\chi: \C^* \rightarrow T$ by 
$\Phi^{-1} \circ \chi: t \mapsto (t_1:=t^{p_1},\ldots ,t_{\mu}:=t^{p_{\mu}})$,
then, for $i = 1,\ldots ,\nu$, we set 
$v_i := f_i\circ L(\Phi^{-1}\circ\chi) := L\xi_i \circ L\chi$. 
We have  $v_i(1) = f_i(p)\neq 0$, then 
$<L\xi_i ,L\chi>\ =\, <\xi_i,\chi>\, \neq 0$.
\end{proof}

We return to the proof of the theorem. We will suppose that we are in 
the ``worst case" that is $\dim~R_u(G)/(G,R_u(G)) = 2$, the reader 
will easily adapt the proof to the other cases.
 
We denote by $\mathcal{R}\subset \mathfrak{g}^*$ the set of roots.

The commutative group $V(G)\approx R_u(G)/(G,R_u(G))$ is topologically 
generated by two elements (\cf\ lemma \ref{nbgenerators}).

The Lie algebra of $V(G)$ is the image of $\mathfrak{g}_0^n$ induced 
by the quotient map (\cf\ lemma \ref{generators}). Hence there exist 
$N^+,N^- \in \mathfrak{g}_0^n$ whose images generate the Lie algebra 
of $V(G)$. Then the Lie algebra $\mathfrak{g}$ is generated by $\C N^+,\C N^-$ 
and the Lie algebra of $L(G)$, therefore by 
$\C N^+,\C N^-$, $\mathfrak{t}$ and the root spaces $\mathfrak{g}_{\xi}$, 
$\xi \in \mathcal{R}$ (\cf\ lemma \ref{generators}).
 
We set $\mathcal{R}^+ := \{\xi \in \mathcal{R} \tq <\xi,\chi>\, < 0\}$ and 
$\mathcal{R}^- := \{\xi \in \mathcal{R} \tq <\xi,\chi>\, > 0\}$. We have 
a partition $\mathcal{R} = \mathcal{R}^+ \cup \mathcal{R}^-$.

We denote by $G^+$ (reps. $G^-$) the algebraic subgroup of $G$ topologically
generated by, $T$, $\exp~(\C N^+)$ and the 
$\exp~\mathfrak{g}_{\xi}$, $\xi \in \mathcal{R}^+$
(resp. $T$, $\exp~(\C N^-)$ and the $\exp~\mathfrak{g}_{\xi}$, 
$\xi \in \mathcal{R}^-$). The group $G$ is clearly topologically generated 
by $G^+$ and $G^-$.

Then $\chi$ defines a $\Theta$-structure on  $G^+$ and $\chi^{-1}$ defines 
a $\Theta$-structure on  $G^-$. Using \ref{NC:thsuff} we prove that $G^+$ 
(resp. $G^-$) is the local Galois group of a meromorphic linear 
$q$-difference system whose Newton polygon has integral slopes. We end 
the proof using lemma \ref{GIP:glueing}.
\end{proof}

\begin{rema}
In fact as we noticed above, we proved more than what is stated in 
the proposition. In some sense the only singularities of the constructed 
equation are $0$ and $\infty$ (\cf \cite{JSAIF}). 
This is a first step towards a $q$-analog version of the Abhyankar 
conjecture. The reader will compare with the solution of the differential 
Abhyankar conjecture due to the first author.
\end{rema}

\begin{theo}
\label{NC:thsuff4}
If a complex linear algebraic group $G$ is the $q$-difference Galois 
group of a rational system, then $V(G):=G/L(G)$ is the $q$-difference 
Galois group of a rational regular-singular system.
\end{theo}

The proof is ``Tannakian" and similar to the first part of the proof 
of the theorem \ref{NC:th0}.

Conversely we can conjecture that, using a variant of the proof of 
the proposition \ref{NC:thsuff3}, the condition of the theorem is not 
only necessary but that it is also sufficient (the reader will compare 
with the proof of the corresponding result in the differential case by 
the first author).


\appendix


\section{Pronilpotent completions}
\label{section:pronil}

To a family $(x_i)_{i\in I}$, we associate the \emph{free Lie algebra} 
$\Lib\bigl((x_i)_{i\in I}\bigr)$ generated over $\C$. We will denote 
$\Lib\,\hat{}\bigl((x_i)_{i\in I}\bigr)$ the completion of 
$\Lib\bigl((x_i)_{i\in I}\bigr)$ for the descending central filtration:
\[
L\,\hat{}:=\Lib\,\hat{}\bigl((x_i)_{i\in I}\bigr) = 
{\lim\limits_{\overleftarrow{n\in\N}}}\, L/L^n,
\]
with $L:=\Lib\bigl((x_i)_{i\in I}\bigr)$ and $L^1 := L$, $L^{n+1}:=[L,L^n]$.

If $I$ is \emph{finite}, we refer to \cite{DG} for the following properties. 
Then each $L/L^n$ is a finite dimensional nilpotent complex Lie algebra, 
therefore it is an \emph{algebraic} Lie algebra and $L\hat{}$ is a 
pronilpotent proalgebraic Lie algebra.

The functor ``Lie algebra" is an equivalence between the category of unipotent 
algebraic groups and the category of finite dimensional nilpotent Lie algebras. 
We shall denote $\exp$ the inverse equivalence.

We set:
\[
\exp (L\,\hat{}\,):={\lim\limits_{\overleftarrow{n\in\N}}}\, \exp(L/L^n).
\]
It is a prounipotent algebraic group, whose Lie algebra is $L\,\hat{}$.

If $I$ is \emph{infinite}, then the situation is more complicated. The 
dimension of each nilpotent Lie algebra $L/L^n$ is infinite and the 
pronilpotent completion $L\hat{}$ is not satisfying for our purposes. 
Therefore we will introduce another completion of $L$, the 
f-pronilpotent completion $L^{\dag}$.

Let $J\subset I$ be a \emph{finite} subset. We have a natural map of Lie 
algebras:
\[
p_J:\Lib\bigl((x_i)_{i\in I}\bigr)\rightarrow \Lib\bigl((x_j)_{j\in J}\bigr), 
\]
defined by
$p_J(x_i):=0$ if $i\notin J$ and $p_J(x_i):=x_i$ if $i\in J$. We define 
similarly maps
$p_{J_1,J_2}:\Lib\bigl((x_i)_{i\in J_2}\bigr) \rightarrow 
\Lib\bigl((x_j)_{j\in J_1}\bigr)$
if $J_1\subset J_2\subset I$ ($J_2$ finite).

Going to the nilpotent completions, we get maps:
\[
\hat p_J:\Lib\,\hat{}\bigl((x_i)_{i\in I}\bigr) \rightarrow 
\Lib\,\hat{}\bigl((x_j)_{j\in J}\bigr), \quad
p_{J_1,J_2}:\Lib\,\hat{}\bigl((x_i)_{i\in J_2}\bigr) \rightarrow 
\Lib\,\hat{}\bigl((x_j)_{j\in J_1}\bigr)
\]

The $\Lib\,\hat{}\bigl((x_j)_{j\in J}\bigr)$ ($J\subset I$, $J$ finite) 
are pronilpotent proalgebraic Lie algebras and the $p_{J_1,J_2}$ 
($J_2\subset I$ finite, $J1\subset J_2$) are morphisms of proalgebraic 
Lie algebras. 

We thus get a projective system of  prounipotent proalgebraic Lie algebras 
and, by definition, the f-pronilpotent completion 
$L^{\dag} := \Lib^{\dag}\bigl((x_i)_{i\in I}\bigr)$ of 
$L:=\Lib\bigl((x_i)_{i\in I}\bigr)$ is the projective limit of this system,
\[
L^{\dag} := 
{\lim\limits_{\overleftarrow{J\subset I}}}\, \Lib\,\hat{}\bigl((x_j)_{j\in J}\bigr),
\quad J~\text{finite}.
\]
It can be interpreted as a projective limit of prounipotent proalgebraic 
Lie algebras. Then we can pass to groups, using the functor $\exp$, and 
we can define a projective limit of unipotent groups $\exp L^{\dag}$, whose 
Lie algebra is $L^{\dag}$.

The natural map $L\rightarrow L^{\dag}$ is injective and dominant (its image
is dense).

\begin{rema}
If $I$ is \emph{finite}, then $L^{\dag}=L\, \hat{}$. \\
If $I$ is \emph{infinite}, then we have maps 
$L \rightarrow L\,\hat{}\rightarrow L^{\dag}$ and 
$L\hat{} \rightarrow L^{\dag}$ is not an isomorphism.
\end{rema}

We shall consider now some actions of an \emph{abelian} proalgebraic 
group $G$ on a free Lie algebra $L$ and the corresponding ``semi-direct 
products" $L\rtimes G$.

In what follows we will suppose that each one dimensional complex vector 
space $\C x_i$ is stable under the action of $G$ and that the action of $G$ 
on $\C x_i$ is, for all $i\in I$, algebraic. Therefore the representations
$\rho_i: G \rightarrow \C^*$, given by $g \in G \mapsto \rho_i(g)$, with 
$g(x_i)=\rho_i(g)x_i$ are rational, they are \emph{weights} on $G$. \\

By definition a \emph{representation} $\rho$ of $L \rtimes G$ is the data 
of a rational linear representation $\rho'$ of $G$ 
($\rho':G \rightarrow \GL(V)$), together with a representation $d\rho''$ 
of $L$ in the same space ($d\rho'':L\rightarrow \End(V)$), required to be 
compatible with the corresponding adjoint actions. We consider the 
corresponding Tannakian category $\Rep\, (L\rtimes G)$. \\

In what follows we will suppose that:
\begin{itemize}
\item[(i)]
for all weight on $G$, there exists only a \emph{finite set} of $i\in I$ 
such that $\rho_i=\rho$;
\item[(ii)]
for every representation $\rho=(\rho',d\rho'')$ of $L\rtimes G$, 
the image of $d\rho''$ is a \emph{nilpotent} subalgebra of $\End(V)$.
\end{itemize}

\begin{lemm}
Let $\rho=(\rho',d\rho'')$ be a representation of $L\rtimes G$. 
Then there exists only a finite set of $i\in I$ such that 
$d\rho''(x_i)\neq 0$.
\end{lemm}
\begin{proof}
Let $\rho=(\rho',d\rho'')$ be a representation of $L\rtimes G$ in 
a finite dimensional space $V$. Let $i\in I$, for all $g\in G$:
\[
\Ad_{\rho'(g)}\bigl(d\rho''(x_i)\bigr)=d\rho''\bigl(g(x_i)\bigr)
= d\rho''\bigl(\rho_i(g)x_i\bigr)=\rho_i(g)d\rho''(x_i).
\]
We suppose that $d\rho''(x_i)\neq 0$. There exists $g_0\in G$ such that 
$\rho_i(g_0)\neq 1$,  then 
$\Ad_{\rho'(g_0)}\bigl(d\rho''(x_i)\bigr)=\rho_i(g_0)d\rho''(x_i)$, therefore 
there exists a \emph{root} $\xi$ for the adjoint action of $\rho_1(G)$ on 
$\End V$ such that $d\rho''(x_i)$ belongs to the corresponding root space 
and we have $\rho_i = \xi \circ \rho'$. The number of roots $\xi$ is finite, 
the result follows, using the condition (i).
\end{proof}

If $J \subset I$ is a finite subset such that, for all $i \in I \setminus J$,  
$d\rho''(x_i) = 0$, then the representation $d\rho''$ factors by 
$\Lib\bigl((x_j)_{j\in J}\bigr)$ and, as the image of $d\rho''$ is nilpotent, 
it factors by $\Lib\,\hat{}\bigl((x_j)_{j\in J}\bigr)$. Therefore the natural 
map:
\[
\Lib\bigl((x_i)_{i\in I}\bigr)\rightarrow \Lib^{\dag}\bigl((x_i)_{i\in I}\bigr), 
\]
induces an isomorphism:
\[
\Rep\Bigl(\Lib^{\dag}\bigl((x_i)_{i\in I}\bigr)\rtimes G\Bigr) \rightarrow 
\Rep\Bigl(\Lib\bigl((x_i)_{i\in I}\bigr)\rtimes G\Bigr).
\]

\begin{prop}
\label{prop:calcul}
Under the above conditions, the tannakian group of the tannakian category
$Rep\Bigl(\Lib\bigl((x_i)_{i\in I}\bigr)\rtimes G\Bigr)$ is isomorphic to
$\Lib^{\dag}\bigl((x_i)_{i\in I}\bigr)\rtimes G$. More precisely, if we have 
a $G$-equivariant morphism of prounipotent proalgebraic Lie algebras 
$\varphi:
\Lib^{\dag}\bigl((x_i)_{i\in I}\rightarrow \Lambda$ inducing an isomorphism:
\[
\Rep\Bigl(\Lambda\rtimes G\Bigr) \rightarrow 
\Rep\Bigl(\Lib^{\dag}\bigl((x_i)_{i\in I}\bigr)\rtimes G\Bigr),
\]
then $\varphi$ is an isomorphism.
\end{prop}
\bigskip

\begin{exem}
Our main example is:
$$
I := \{\iota = (\delta,\bar c,i) \vert (\delta,\bar c) \in \N^*\times \Eq, 
i=1,\ldots ,\delta\} \cup\{0\},
$$
with $x_{\iota} := \dot\Delta^{(\delta,\bar c)}_i$ if $\iota \neq 0$ and 
$x_0:=\dot \Delta^{(0)}$. Then $L:=\Lib\bigl((x_\iota)_{\iota\in I}\bigr)$, 
$G:=G^{(0)}_{p,1,s}$. The weight $\rho_i$ are defined by:
$$
\rho_{\iota} := \delta \bar c,
$$ 
$\Eq$ being interpreted as the group of weights on $\Hom_{gr}(\Eq,\C^*)$) 
if $\iota\neq 0$ and $\rho_1:=1$. \\
It is easy ro check that the conditions (i), (ii) are satisfied.
\end{exem}

Using the proposition \ref{prop:calcul}, we prove that
\[
L^{\dag}\rightarrow\tilde{\mathfrak{st}} 
\]
is an isomorphism of pronilpotent proalgebraic Lie algebras.
it follows that
\[
\exp(L^{\dag})\rtimes G^{(0)}_{p,1,s} \rightarrow
\exp(\tilde{\mathfrak{st}})\,\rtimes G^{(0)}_{p,1,s}=G^{(0)}_1
\]
is an isomorphism of proalgebraic groups, giving a \emph{transcendental} 
explicit description of the $q$-difference universal local Galois 
group $G^{(0)}_1$.
\bigskip

\paragraph{Variations.}

It is possible to use (more complicated) variants of the above formalism 
for various problems of local classification of dynamical systems.

 \begin{enumerate}
\item Local classification of meromorphic linear differential equations. 
In that case condition (ii) is not satisfied.
\item Local classification of meromorphic linear difference equations.
\item Local classification of meromorphic saddle nodes in the plane. 
In that case it is  necessary to use some \emph{infinite dimensional} 
representations. As an exercise the reader can explicit this example 
using the dictionnary between the Martinet-Ramis classification and 
the Ecalle resurgent classification detailed in \cite{Sauz}.
\end{enumerate}

There are also some analogies with the wild ramification phenomena in 
the classical Galois theory of local fields, but that is another story.


\backmatter

\bibliography{qwildgroupIII}

\end{document}